\documentclass[11pt,reqno]{amsart}
\usepackage[margin=1in]{geometry}
\usepackage{amsfonts}
\usepackage{amsmath}
\usepackage{amsthm} %
\usepackage{amssymb}
\usepackage{epsfig, graphics}
\usepackage{hyperref}
\hypersetup{
    colorlinks = true,
    linkcolor = {blue},
    citecolor = {red}
}
\usepackage{url}
\usepackage{color}
\usepackage{setspace}
\usepackage{float}
\usepackage{subfig}
\usepackage{colortbl}
\usepackage{bm}
\usepackage{bbm}
\usepackage{mathtools}
\mathtoolsset{showonlyrefs}
\usepackage{dsfont}
\usepackage{mathrsfs}
\usepackage{fancyvrb}
\usepackage{comment}

\usepackage{enumitem}

\usepackage[square, numbers, comma, sort&compress]{natbib}

\usepackage[margin=1in]{geometry}

\newtheorem{theorem}{Theorem}[section]
\newtheorem{remark}[theorem]{Remark}

\newtheorem{definition}[theorem]{Definition}

\newtheorem{proposition}[theorem]{Proposition}
\newtheorem{lemma}[theorem]{Lemma}

\newtheorem{assumption}{Assumption}

\newcounter{subassumption}[assumption]

\makeatletter
\renewcommand{\p@subassumption}{\theasu}%
\makeatother

\usepackage[dvipsnames]{xcolor}

\newcommand\cP{\mathcal P}

\def\cY{\mathcal Y}

\newcommand\EE {\mathbb E}
\newcommand\FF {\mathbb F}

\newcommand\RR {\mathbb R}

\newcommand\PP {\mathbb P}

\newcommand\cA {\mathcal A}

\newcommand\cF {\mathcal F}

\newcommand{\cR}{\mathcal{R}}
\newcommand{\cQ}{\mathcal{Q}}

\newcommand{\kL}{\mathfrak{L}}

\newcommand{\bS}{\mathbb{S}}

\newcommand{\wt}{\widetilde}
\newcommand{\wh}{\widehat}
\newcommand{\lan}{\langle}
\newcommand{\ran}{\rangle}
\newcommand{\cd}{\cdot}

\newcommand{\cM}{\mathcal{M}}

\newcommand{\sU}{\mathscr{U}}

\makeatletter
\newtheoremstyle{nodot}  %
  {\topsep}  %
  {\topsep}  %
  {\itshape} %
  {}         %
  {\bfseries} %
  {}         %
  { }        %
  {\thmname{#1} \thmnumber{#2} \thmnote{(#3)}} %
\makeatother

\theoremstyle{nodot}

\title[Long-time behavior and turnpike properties of LQ GMFC problems]{Long-time behavior and turnpike properties of linear-quadratic graphon mean field control problems}
\author[E. Bayraktar]{Erhan Bayraktar}
\thanks{Erhan Bayraktar is supported in part by National Science Foundation Grant \href{https://www.nsf.gov/awardsearch/showAward?AWD_ID=2507940\&HistoricalAwards=false}{DMS-2507940} and the Susan M. Smith Chair.}
\address{Department of Mathematics\\
University of Michigan\\
Ann Arbor, MI 48109 \\
United States
}
\email{erhan@umich.edu}

\author[Z. Cao]{Zhongyuan Cao}
\address{NYU-ECNU Institute of Mathematical Sciences \\
NYU Shanghai \\
Shanghai, 200126 \\
People’s Republic of China
}
\email{zc3151@nyu.edu}

\author[J. Jian]{Jiamin Jian}
\address{Department of Mathematics\\
University of Michigan\\
Ann Arbor, MI 48109\\
United States
}
\email{jiaminj@umich.edu }

\date{}

\begin{document}
\maketitle

\begin{abstract}
We investigate the asymptotic behavior and turnpike properties of graphon mean field control (GMFC) problems in the linear-quadratic setting. We consider both a finite-horizon GMFC problem and its associated ergodic counterpart, in which the controlled dynamics are governed by a graphon mean field stochastic differential equation with heterogeneous interactions. The optimal controls and state trajectories for both problems are characterized by systems of Riccati equations together with systems of generalized differential and algebraic equations on suitable Hilbert spaces. Under a stabilizability condition and appropriate positivity assumptions on the graphon-induced operators, we establish the unique solvability of the ergodic control problem and derive exponential convergence estimates for the finite-horizon system to its stationary limit. As a consequence, we establish an exponential turnpike property for the optimal pair and prove the convergence of the time-averaged value function for the finite-horizon GMFC problem.


\end{abstract}


\section{Introduction}

Mean field control problems arise naturally as infinite-population limits of social optimization problems, in which a central planner coordinates a large number of interacting agents to minimize a collective cost. Such problems have been extensively studied over the past decade and have found broad applications, particularly in modeling cooperative social optimization for large populations of interacting agents. Foundational contributions include \cite{Nourian-Caines-Malhame-Huang-2012, Bensoussan-2013, CD18I}, which established the basic formulation and analytical framework of mean field control. 

Subsequent theoretical developments in mean field control have employed both the stochastic maximum principle \cite{Andersson-Djehiche-2011, Buckdahn-Li-Ma-2016} and dynamic programming approaches \cite{Pham-Wei-2017}. Infinite-horizon mean field control has also been investigated in \cite{Bayraktar-Zhang-2023} through the solvability of the associated infinite-horizon McKean--Vlasov forward-backward stochastic differential equations (FBSDEs). In the linear-quadratic (LQ) setting, \cite{Yong-2013} derived explicit solutions through FBSDEs and Riccati equations, with further extensions addressing nonhomogeneous terms \cite{Sun-2017, Li-Sun-Yong-2016}, random coefficients \cite{Xiong-Xu-2025}, and infinite-horizon control problems \cite{Huang-Li-Yong-2015}. For mean field control problems with common noise, in which the limiting state dynamics become conditional McKean--Vlasov processes, we refer to \cite{Pham-Wei-2017, Li-Mou-Wu-Zhou-2025}.

Classical mean field control models typically assume that agents are homogeneous or exchangeable, an assumption that fails to capture the heterogeneous and asymmetric interaction structures arising in many large-scale systems in economics, finance, engineering, social networks, and systemic risk. The theory of graphons \cite{lovasz2012} provides a natural and rigorous framework for modeling such systems in the infinite-population limit. A graphon is a symmetric measurable function $G$ on $[0, 1]^2$ that serves as the continuum limit of dense graph sequences. In graphon mean field models, agents are indexed by labels in $[0, 1]$, and the strength of interaction between two agents is governed by the value of the kernel $G$ at their respective labels.

Mean field control problems in non-exchangeable systems with heterogeneous interactions have received comparatively little attention. In the linear-quadratic setting, GMFC problems were studied in \cite{Crescenze-Feo-Pham-2025, Dunyak-Caines-2026, Gao-Caines-2019, Liang-Wang-Zhang-2021, Xu-Gou-Huang-2024}. Using dynamic programming
on flows of families of probability measures, \cite{Crescenzo-Pham-2026} studied GMFC problems via PDE methods. A version of the stochastic maximum principle for the optimal control of non-exchangeable systems was developed in \cite{Kharroubi-Mekkaoui-Pham-2025}, and a comprehensive probabilistic analysis of general GMFC problems, including well-posedness of the associated FBSDEs, a Pontryagin maximum principle, and a propagation-of-chaos result, was carried out in \cite{Cao-Lauriere-2025}. Graphon mean field FBSDEs arising from a form of Pontryagin's maximum principle
were analyzed in \cite{Aurell-Carmona-Lauriere-2022} for the linear-quadratic case. More generally, the well-posedness and propagation of chaos for graphon mean field SDEs and coupled graphon mean field FBSDEs have been studied in \cite{Bayraktar-Chakraborty-Wu-2023} and \cite{Bayraktar-Wu-Zhang-2023}, respectively. On the game-theoretic side, the framework of graphon mean field games (GMFGs) and the associated GMFG equations were formulated in \cite{Caines2021graphon, Caines2022graphon}.

On the other hand, the long-time behavior and turnpike properties of optimization problems have attracted considerable attention in recent years. We refer to \cite{Trelat-Zuazua-2025} for a comprehensive survey of turnpike theory and its various applications. The turnpike phenomenon refers to the tendency of optimal controls and state trajectories, over long time horizons, to remain close to those of a corresponding stationary or ergodic problem for most of the time interval, except near the temporal boundaries. This concept originated in mathematical economics and has become a central topic in deterministic optimal control; see \cite{Porretta-Zuazua-2013, DGSW-2014, Trelat-Zuazua-2015, Gugat-Herty-Segala-2024} and the references therein. The study of turnpike phenomena in stochastic optimal control is comparatively more recent. A significant breakthrough was achieved by Sun, Wang, and Yong \cite{Sun-Wang-Yong-2022}, who established the first rigorous turnpike results for stochastic linear-quadratic optimal control problems, initiating a systematic investigation of stochastic turnpike behavior. 
Subsequent works extended this theory to broader classes of problems, including mean field-type control problems \cite{Sun-Yong-2024, Bayraktar-Jian-2025}, mean field control problems with common noise \cite{Bayraktar-Jian-2026}, and control problems with regime switching \cite{Mei-Wang-Yong-2025, Mei-Wang-Yong-2-2025}. Driven by the growing interest in understanding long-term equilibrium behavior in large interacting particle systems, the study of turnpike properties in multiplayer differential games \cite{Cohen-Jian-2025} and mean field games \cite{CLLP-2013, Cardaliaguet-Porretta-2019, Cirant-Porretta-2021} has emerged as a rapidly advancing field.

Despite these developments, the long-time behavior and turnpike properties of GMFC problems have not yet been studied in the literature, to the best of our knowledge. The present paper aims to bridge this gap by investigating turnpike properties for GMFC problems in the linear-quadratic setting.

We consider both a finite-horizon GMFC problem and its associated ergodic counterpart, in which the controlled dynamics are governed by a graphon mean field SDE with heterogeneous mean field interactions. The central objective is to characterize the long-time behavior of optimal controls and state trajectories in the finite-horizon problem by studying their convergence to the optimal solution of the ergodic problem. Our approach is inspired by the framework introduced in \cite{Jian-Jin-Song-Yong-2026} for the turnpike property of LQ stochastic control problems and subsequently extended in \cite{Bayraktar-Jian-2025} to mean field control problems. The present work further develops this framework to accommodate the graphon structure and the associated label-indexed heterogeneous interactions. Our main result establishes an exponential turnpike property: away from the temporal boundaries, the optimal pair of the finite-horizon GMFC problem remains exponentially close to the optimal pair of the associated ergodic GMFC problem. More precisely, under a suitable stabilizability condition and appropriate positivity assumptions on the graphon-induced operators, we prove that there exist positive constants $K$ and $\lambda$, independent of $T$, such that 
$$\mathbb{E} \Big[\int_I |X_T^u(t) - X_{\infty}^u(t)|^2 du \Big] + \mathbb{E} \Big[\int_I |\alpha_T^u(t) - \alpha_{\infty}^u(t)|^2 du \Big] \leq K \big(e^{-\lambda t} + e^{-\lambda(T-t)} \big),$$
for all $t \in [0, T]$. Here, $I = [0, 1]$ is the label set, while $(X^u_T, \alpha^u_T)_{u \in I}$ and $(X^u_{\infty}, \alpha^u_{\infty})_{u \in I}$ denote the optimal state-control pairs for the finite-horizon and ergodic GMFC problems, respectively, with possibly different initial states.

\subsection{Main contributions and organization of the paper}

The main contributions of this paper are summarized as follows: 
\begin{itemize}
\item[(i)] \textit{Solvability of the ergodic GMFC problem.} We formulate the ergodic GMFC problem in the linear-quadratic setting and then characterize its solution through a new system of algebraic equations associated with the graphon structure. A key step is the derivation of a fundamental relation for the cost functional of the ergodic problem, expressed as the sum of a quadratic term in the control and a constant term, using the solution to this algebraic system; see Lemma \ref{l:fundamental_relation_ergodic}. Under a suitable stabilizability condition, we then establish the existence and uniqueness of a stabilizing solution and prove the solvability of the ergodic control problem; see Proposition \ref{p:solvability_ergodic_control}. The central technical challenge lies in establishing the unique solvability of the algebraic system. To this end, we introduce Assumption \ref{a:operators}, a joint positive semidefiniteness condition on the operator-valued matrices formed from the solutions to the Riccati equations and the coefficients of the state dynamics and cost functional. This condition is sufficient to guarantee the unique solvability of the algebraic system.

\item[(ii)] \textit{Exponential convergence of Riccati systems.} For both the finite-horizon and ergodic GMFC problems, the solutions are characterized by systems of generalized ordinary differential equations and algebraic equations posed on suitable Hilbert spaces. A central step in establishing the turnpike property is the derivation of quantitative exponential convergence estimates between the solution of the finite-horizon system and the stabilizing solution of the corresponding algebraic system arising in the ergodic control problem; see Lemma~\ref{l:exponential_estimates_P_Pi}, Proposition~\ref{p:exponential_estimates_Lambda}, and Proposition~\ref{p:exponential_estimates_p}. The derivation of these estimates requires a careful treatment of the graphon-induced structure and the associated systems of generalized equations on Hilbert spaces.

\item[(iii)] \textit{Turnpike property.} Building on the convergence estimates for the generalized systems, we establish an exponential turnpike property for the optimal pair of the finite-horizon GMFC problem, showing that it remains exponentially close to the optimal pair of the ergodic GMFC problem for most of the time horizon; see Theorem \ref{t:turnpike_property}. Moreover, we prove the convergence of the time-averaged value function in 
Lemma \ref{l:convergence_value_function}. To the best of our knowledge, this is the first turnpike result for GMFC problems with heterogeneous interactions encoded by a graphon.
\end{itemize}

The remainder of the paper is organized as follows: Section \ref{s:finite_control} presents the formulation and solvability of the finite-horizon GMFC problem. Specifically, the probabilistic setup and problem setup are introduced in Section \ref{s:probabilistic_setup} and Section \ref{s:problem_setup_finite_control}, respectively, the system of generalized Riccati differential equations is derived in Section \ref{sec:System_Riccati_equations_finite}, and the solvability of the Riccati system and finite-horizon control problem is established in Section \ref{s:main_result_finite_control}. Section \ref{s:ergodic_control} is devoted to the ergodic GMFC problem: the stabilizability condition for the state equation is introduced in Section \ref{s:stabilizability_condition}, the system of algebraic equations is derived in Section \ref{s:system_algebraic_equations}, and the solvability of the ergodic control problem together with the characterization of its optimal pair are established in Section \ref{s:main_result_ergodic_control}. In Section \ref{s:convergence_Riccati_system}, we derive the quantitative exponential convergence estimates between the solution of the finite-horizon system and the stabilizing solution of the algebraic system associated with the ergodic control problem. The main results of the paper are presented in Section \ref{s:main_result}. Specifically, the turnpike property for the optimal pairs of the finite-horizon GMFC problem is established in Section \ref{s:turnpike_properties}, and the convergence of the time-averaged value function is proved in Section \ref{s:convergence_value}. Finally, Section \ref{s:solvability_algebraic} provides a sufficient condition for the unique solvability of the generalized algebraic equations arising in the ergodic problem.

We close this section by introducing some frequently used notation.

\subsection{Notation}


For given positive integers $d, m \in \mathbb N$, we use $\mathbb R^d$ and $\mathbb R^{d \times m}$ to denote the standard $d$-dimensional real Euclidean space and the Euclidean space consisting of all $d \times m$ real matrices, respectively. Additionally, let $\mathbb S^d$, $\mathbb S^d_{+}$, and $\mathbb S^{d}_{++}$ denote the sets of all $d \times d$ real symmetric matrices, symmetric positive semidefinite matrices, and symmetric positive definite matrices, respectively. For matrices $P, Q \in \mathbb S^d$, we write $P \geq Q$ (respectively, $P > Q$) if and only if the matrix $P - Q$ is positive semidefinite (respectively, positive definite).
We use $\bm{I}_d$ to denote the $d \times d$ identity matrix.

We denote by $\lan \cd, \cd \ran$ the inner product of two vectors, and by $|\cdot|$ the Euclidean norm on the corresponding Euclidean vector space. Moreover, we use the superscript $\top$ to indicate the transpose operation of matrices. For a matrix $A$, we denote by $\|A\|$ and $\|A\|_F$ its spectral norm and Frobenius norm, respectively.

We denote by $C([0,T];\RR^d)$ the set of continuous mappings from $[0,T]$ to $\RR^d$, and by $C^1([0, T]; \RR^d)$ the subset consisting of functions whose first-order derivatives are continuous. 
Let $E$ be a Banach space endowed with the norm $\|\cdot\|_{E}$. We denote by $\cP_2(E)$ the space of probability measures on $E$ with finite second moments, equipped with the 2-Wasserstein distance. For a linear operator $\mathcal{L}$ on $E$, its operator norm is defined by $\|\mathcal L\|_{op} = \sup_{\|x\|_{E} = 1} \|\mathcal{L}(x)\|_{E}$. Moreover, for $I = [0,1]$ and $p \in \{1, 2\}$, we denote by $L^p(I; E)$ (respectively, $L^{\infty}(I; E)$) the set of elements $\phi = (\phi^u)_{u \in I}$ such that $u \mapsto \phi^u \in E$ is measurable and $\int_I \|\phi^u\|^p_{E} du < \infty$ (respectively, $\sup_{u \in I} \|\phi^u\|_{E} < \infty$). We also denote by $L^2(I \times I; E)$ (respectively $L^{\infty}(I \times I; E)$) the set of measurable functions $h: I \times I \to E$ such that $\|h\|^2_{L^2(I \times I)}:= \int_I \int_I \|h(u, v)\|^2_{E} dv du < \infty$ (respectively $\sup_{u, v \in I} \|h(u, v)\|_{E} < \infty$). Specifically, for a kernel $h \in L^2(I \times I; \mathbb{R}^{d \times d})$, we introduce Hilbert--Schmidt integral operators, denoted by $\bm{h}$, i.e., for $\phi \in L^2(I; \mathbb{R}^d)$, 
\begin{equation}
\label{eq:Hilbert_Schmidt_operator}
\bm{h}(\phi)(u) := \int_I h(u, v) \phi^v dv, \quad \|\bm{h}\|_{op} \leq \|h\|_{L^2(I \times I)} = \Big(\int_I \int_I \|h(u, v)\|^2_{F} \, dvdu \Big)^{\frac{1}{2}}.
\end{equation}
We also write
$$\lan \bm{h}(\phi), \phi \ran_{L^2} := \int_I \int_I \lan \phi^u, h(u, v) \phi^v \ran dv du.$$



\section{Finite-horizon control problem}
\label{s:finite_control}


\subsection{Probabilistic setup}
\label{s:probabilistic_setup}

We first state our probabilistic setting. Let $T>0$ be a finite horizon, let $(\Omega,\cF,\PP)$ be a complete probability space, and let $I:=[0,1]$ be endowed with its Borel $\sigma$-algebra. 
For every $u\in I$, we consider an $\RR$-valued one-dimensional Brownian motion $W^u$ and an independent random variable $\zeta^u$ that is uniformly distributed on $(0,1)$. For simplicity, we consider dynamics driven by one-dimensional Brownian motions, but all the results extend to multidimensional Brownian motions. We assume that $\{(W^u,\zeta^u)\,:\, u\in I\}$ is an independent family; that is, $(W^u,\zeta^u)$ and $(W^v,\zeta^v)$ are independent for $u,v\in I$ with $u\neq v$. For every $u\in I$, we denote by $(\cF^{W^u}_t)_{t\ge 0}$ the natural filtration generated by $W^u$, by $\sigma(\zeta^u)$ the $\sigma$-algebra generated by $\zeta^u$, and by $\FF^u=(\cF^{u}_t)_{t\ge 0}$ the filtration defined by
$\cF^{u}_t= \cF^{W^u}_t\vee \sigma(\zeta^u) \vee \mathcal{N}_{\PP}, t\ge0,
$
where $\mathcal{N}_{\PP}$ is the family of $\PP$-null sets. For a random variable $X$, we denote its law by $\text{Law}(X)$.

We denote by $\kL^2_t$ the set of collections $\xi=(\xi^u)_{u \in I}$ of $\RR^d$-valued random variables such that for each $u\in I$, $\xi^u$ is $\cF_t^u$-measurable, the map $I\ni u\mapsto \text{Law}(W^u_{\cdot\wedge t},\zeta^u,\xi^u)\in \cP_2(C([0,T];\RR)\times (0,1)\times \RR^d))$ is measurable, and $\int_I\EE[|\xi^u|^2]d u <\infty$.



\subsection{Finite-horizon control problem}
\label{s:problem_setup_finite_control}

In this section, we formulate the finite-horizon linear-quadratic graphon mean field control problem over $[0, T]$. We begin by specifying the class of graphons considered throughout the paper. In \cite{lovasz2012}, a graphon is defined as a measurable symmetric function from $I\times I$ to $I$. In this paper, we consider a more general class of graphons, which is suitable for modeling both dense graphs and graphs in the intermediate regime. For simplicity, we refer to these functions simply as graphons. More precisely, a graphon is a measurable function $G: I \times I \rightarrow \RR$ belonging to $L^2(I \times I)$. 

\textbf{Admissible controls.} In this paper, we take $\RR^m$ as the action space. Given $\xi \in \kL^2_0$, an $I$-collection of $\RR^m$-valued stochastic processes $\alpha = (\alpha^u)_{u\in I}$ is called an admissible control profile if there exists a Borel measurable function $a:I\times[0,T]\times C([0,T];\RR)\times (0,1)\to \RR^m$ such that, for each $u\in I$, $\alpha^u(t)=a(u,t,W^u_{\cdot\wedge t},\zeta^u)$ and $\int_I\int_0^T \EE[|\alpha^u(t)|^2]dtdu <\infty$. In particular, for each $u \in I$, the process $(\alpha^u(t))_{t \geq 0}$ is $\mathbb{F}^u$-predictable. We denote the set of admissible control profiles by $\cA^{\xi}[0, T]$ and simply write $\cA[0, T]$ whenever no confusion can arise.


For a given admissible control profile $\alpha = (\alpha^u)_{u\in I}\in \cA[0, T]$, the corresponding state process is governed by the following controlled graphon mean field dynamics, represented by a continuum system of SDEs: 
\begin{equation}
\label{eq:state_process}
\begin{cases}
dX^u(t) = \big\{AX^u(t) + \bar{A} \bar X^u(t) + \wt{A} [G\bar{X}(t)]^u + B\alpha^u(t) + b \big\}dt 
\\
\hspace{0.7in} +\big\{CX^u(t) + \bar{C} \bar X^u(t) + \wt{C}[G\bar{X}(t)]^u + D \alpha^u(t) + \sigma \big\}dW^u(t), \\
X^u(0) = \xi^u,
\end{cases}
\end{equation}
for each $u \in I$, where $\xi \in \kL^2_0$ is the initial condition, $b, \sigma \in \RR^d$, $A, \bar A, \wt{A}, C,\bar C, \wt{C} \in \RR^{d\times d}$, and $B,D\in \RR^{d\times m}$. Here, we define $\bar{X}^u(t) := \EE[X^u(t)]$, and use $X(t) = (X^u(t))_{u\in I}$ and $\bar{X}(t) = (\bar X^u(t))_{u\in I}$ to denote the entire collections of state processes and their expectations, respectively. We also use the notation
\begin{equation}\label{eq:GX}
    [G\bar{X}(t)]^u=\int_IG(u,v)\bar X^v(t)dv.
\end{equation}
We use the superscript $\cdot^{\alpha}$ when it is necessary to emphasize the dependence of a process on the control profile $\alpha$. When the context is clear, we omit this superscript to simplify the notation.

\textbf{Cost functional.} Given $\xi \in \mathfrak{L}^2_0$, the objective of the finite-horizon GMFC problem is to minimize, over all $\alpha \in \mathcal A[0, T]$, a cost functional of the form
\begin{equation}
\label{eq:cost_functional_finite_horizon}
\begin{aligned}
J_T(\xi, \alpha) &:= \;  \int_I\EE\Big[  \int_0^T  \big(\langle X^u(s), Q X^u(s) \rangle + 2\langle \alpha^u(s), S X^u(s) \rangle + \langle \alpha^u(s), R \alpha^u(s) \rangle \\
&\hspace{0.9in} + \langle \bar{X}^u(s), \bar{Q} \bar{X}^u(s) \rangle + \langle \bar{X}^u(s), \wt{Q} [G\bar{X}(s)]^u\rangle \\
&\hspace{0.9in} + \langle [G\bar{X}(s)]^u, \check{Q} [G\bar{X}(s)]^u\rangle \big)  ds \Big] du,
\end{aligned}
\end{equation}
where $Q, \bar{Q}, \check{Q} \in \bS^d_{+}$, $\wt{Q} \in \mathbb{R}^{d \times d}$, $S \in \RR^{m\times d}$, and $R\in\bS^m_{++}$. 

By \cite[Theorem 2.6]{Crescenzo-Pham-2026}, for any given admissible control profile $\alpha = (\alpha^u)_{u\in I} \in \cA[0, T]$ and any initial state $\xi\in \kL^2_0$, the state equation \eqref{eq:state_process} admits a unique solution. Moreover, the map $I\ni u\mapsto \text{Law}(X^u,W^u,\zeta^u)\in \cP_2(C([0,T];\RR^d)\times C([0,T];\RR)\times (0,1))$ is measurable, which in particular gives the measurability of $u\mapsto (\EE[\langle X^u(s), X^u(s) \rangle], \EE[\langle \alpha^u(s),X^u(s) \rangle],\EE[\langle \alpha^u(s),\alpha^u(s) \rangle],\bar{X}^u(s))$, for all $s\in[0,T]$. As a consequence, for each 
$t\in[0,T]$, $[G\bar{X}(t)]^u$ is well defined for all $u\in I$ and the map $u\mapsto [G\bar{X}(t)]^u$ is measurable. It therefore follows that all integrals with respect to $u$ in the cost functional \eqref{eq:cost_functional_finite_horizon} are well defined.

\begin{remark}
\label{r:relation_coefficient}
To relate our formulation to the GMFC problem (2.1)-(2.3) considered in \cite{Crescenze-Feo-Pham-2025}, we match the coefficients in the state dynamics and the cost functional as follows:
\begin{equation}
\label{eq:coefficients_compare}
\begin{aligned}
& A^u = A, \, B^u = B, \, \beta^u = b, \, C^u = C, \, D^u = D, \, \gamma^u = \sigma, \\
& G_A (u, v) = \wt{A} G(u, v) + \bar{A} \delta(u-v), \, G_C (u, v) = \wt{C} G(u, v) + \bar{C} \delta(u-v), \\
& Q^u = Q, \, R^u = R, \, 0 = S, \, 0 = \check{Q}, \, H^u = 0, \, G_H(u, v) = 0, \\
& G_Q (u, v) = \wt{Q} G(u, v) + \bar{Q} \delta(u-v),
\end{aligned}
\end{equation}
where $\delta$ denotes the Dirac delta function. This identification follows from the identity
\begin{equation*}
\begin{aligned}
\bar{A} \bar X^u(t) + \wt{A} [G\bar{X}(t)]^u = \int_I \big(\bar{A} \delta(u-v) + \wt{A}G(u, v) \big) \bar{X}^v(t) dv,
\end{aligned}
\end{equation*}
with analogous relations holding for the terms involving $(\bar C, \wt C)$ and $(\bar Q, \wt Q)$. 
In \eqref{eq:coefficients_compare}, the notation on the left-hand side corresponds to the framework of \cite{Crescenze-Feo-Pham-2025}, while the right-hand side corresponds to our setting. We emphasize that \cite{Crescenze-Feo-Pham-2025} assumes that the graphons $G_A, G_C$, and $G_Q$ belong to $L^2(I \times I; \mathbb R^{d \times d})$. The presence of the Dirac term $\delta$ implies that the graphons in our formulation are distributions rather than square-integrable functions and therefore do not satisfy this $L^2$-integrability requirement. Because our cost functional includes the state-control cross term $\langle \alpha^u(s), S X^u(s) \rangle$, it is more general than the one considered in \cite{Crescenze-Feo-Pham-2025}. Since the classes of admissible controls and initial conditions considered in \cite{Crescenze-Feo-Pham-2025} do not guarantee the measurability of the map $u\mapsto \EE[\int_0^T\langle \alpha^u(s), S X^u(s) \rangle ds]$, we adopt classes of admissible controls and initial conditions analogous to those introduced in \cite{Crescenzo-Pham-2026}, rather than those used in \cite{Crescenze-Feo-Pham-2025}.
\end{remark}

\subsection{System of Riccati equations}
\label{sec:System_Riccati_equations_finite}

Following the standard approach in linear-quadratic control theory, we begin with a quadratic ansatz for the value function of the form
\begin{equation}
\label{eq:value_function_ansatz_finite}
\begin{aligned}
\mathbb{V}_T(t, \xi) &= \int_I \EE[\langle \xi^u, P_T(t) \xi^u \rangle] du + \int_I \langle \bar{\xi}^u, \Pi_T(t) \bar{\xi}^u \rangle du + \int_I \int_I \langle \bar{\xi}^u, \Lambda_T(t)(u, v) \bar{\xi}^v \rangle dv du \\
& \hspace{0.3in} + 2 \int_I \langle \bar{\xi}^u, p_T(t)(u)\rangle du + \int_I \kappa_T(t)(u) du,
\end{aligned}
\end{equation}
where $\bar{\xi}^u = \mathbb{E}[\xi^u]$ for all $u \in I$, and $P_T, \Pi_T \in C([0, T]; \bS^d)$, $\Lambda_T \in C([0, T]; L^2(I \times I; \RR^{d \times d}))$, $p_T \in C([0, T]; L^2(I; \RR^d))$, and $\kappa_T \in C([0, T]; L^1(I; \RR))$ are suitable functions to be determined. A standard strategy in finite- and infinite-dimensional linear-quadratic control problems is to use the ansatz \eqref{eq:value_function_ansatz_finite} and It\^o's formula to decompose the cost functional $J_T$ in \eqref{eq:cost_functional_finite_horizon} into the sum of a quadratic term in the control and a constant term, leading to suitable Riccati equations. To this end, we define
\begin{equation*}
\begin{aligned}
\mathcal V_T(t) &:= \int_I \EE[\langle X^u(t), P_T(t) X^u(t) \rangle] du + \int_I \langle \bar{X}^u(t), \Pi_T(t) \bar{X}^u(t) \rangle du \\
& \hspace{0.3in} + \int_I \int_I \langle \bar{X}^u(t), \Lambda_T(t)(u, v) \bar{X}^v(t) \rangle dv du + 2 \int_I \langle \bar{X}^u(t), p_T(t)(u)\rangle du + \int_I \kappa_T(t)(u) du.
\end{aligned}
\end{equation*}
By the definition of $\mathcal V_T$, we can apply the standard It\^o's formula for a.e. $u \in I$ and then integrate the resulting identity over $I$. Moreover, we introduce the notation
\begin{equation*}
\begin{aligned}
\mathcal Y_T(t) &:= \mathcal V_T(t) + \int_0^t \int_I \mathbb E\big[\langle X^u(s), Q X^u(s) \rangle + 2\langle \alpha^u(s), S X^u(s) \rangle + \langle \alpha^u(s), R \alpha^u(s) \rangle \\
&\hspace{0.9in} + \langle \bar{X}^u(s), \bar{Q} \bar{X}^u(s) \rangle + \langle \bar{X}^u(s), \wt{Q} [G\bar{X}(s)]^u\rangle \\
&\hspace{0.9in} + \langle [G\bar{X}(s)]^u, \check{Q} [G\bar{X}(s)]^u\rangle \big] du ds,
\end{aligned}
\end{equation*}
and
\begin{equation*}
\begin{aligned}
b^u(t) &:= AX^u(t) + \bar{A} \bar X^u(t) + \wt{A} [G\bar{X}(t)]^u + B\alpha^u(t) + b, \\
\sigma^u(t) &:= CX^u(t) + \bar{C} \bar X^u(t) + \wt{C}[G\bar{X}(t)]^u + D \alpha^u(t) + \sigma.
\end{aligned}
\end{equation*}
Taking expectations gives the dynamics of $\bar{X}^u$:
$$\dot{\bar{X}}^u(t) = \bar{b}^u(t), \quad \bar{X}^u(0) = \bar{\xi}^u$$
for a.e. $u \in I$, where 
$$\bar{b}^u(t) := (A +\bar{A}) \bar X^u(t) + \wt{A} [G\bar{X}(t)]^u + B\bar{\alpha}^u(t) + b.$$
Then, we obtain
\begin{equation*}
\begin{aligned}
& d \mathcal Y_T(t) \\
:= \ & \int_I \EE \big[\lan b^u(t), P_T(t) X^u(t)\ran + \lan X^u(t), \dot{P}_T(t) X^u(t) \ran + \lan X^u(t), P_T(t) b^u(t) \ran + \lan \sigma^u(t), P_T(t) \sigma^u(t) \ran \big] du dt \\
&\hspace{0.1in} + \int_I \lan \bar{b}^u(t), \Pi_T(t) \bar{X}^u(t) \ran + \lan \bar{X}^u(t), \dot{\Pi}_T(t) \bar{X}^u(t) \ran + \lan \bar{X}^u(t), \Pi_T(t) \bar{b}^u(t) \ran du dt \\
&\hspace{0.1in} + \int_I \int_I \lan \bar{b}^u(t), \Lambda_T(t)(u, v) \bar{X}^v(t) \ran + \lan \bar{X}^u(t), \dot{\Lambda}_T(t)(u, v) \bar{X}^v(t) \ran + \lan \bar{X}^u(t), \Lambda_T(t)(u, v) \bar{b}^v(t) \ran dv du dt \\
&\hspace{0.1in} + \int_I 2 \lan \bar{b}^u(t), p_T(t)(u) \ran + 2 \lan \bar{X}^u(t), \dot{p}_T(t)(u) \ran du dt + \int_I \dot{\kappa}_T(t)(u) du dt \\
&\hspace{0.1in} + \int_I \mathbb E\big[\langle X^u(s), Q X^u(s) \rangle + 2\langle \alpha^u(s), S X^u(s) \rangle + \langle \alpha^u(s), R \alpha^u(s) \rangle + \langle \bar{X}^u(s), \bar{Q} \bar{X}^u(s) \rangle \\
&\hspace{0.7in} + \langle \bar{X}^u(s), \wt{Q} [G\bar{X}(s)]^u\rangle + \langle [G\bar{X}(s)]^u, \check{Q} [G\bar{X}(s)]^u\rangle \big] du dt.
\end{aligned}
\end{equation*}

Next, we rewrite the above expression in a suitable form by considering the quadratic terms involving $X^u, \bar{X}^u$, and $\alpha^u$; the cross term involving $\bar{X}^u$ and $\bar{X}^v$; and the linear terms in $\bar{X}^u$ and $\alpha^u$. To simplify the notation, for $P, \Pi \in \bS^d_{+}$, we define
\begin{equation*}
\begin{aligned}
& \cQ(P) = A^\top P + PA +C^\top PC + Q, \\
& \bar{\cQ}(P, \Pi) = \bar{A}^\top P + P \bar{A} + C^\top P \bar{C} + \bar{C}^\top P C + \bar{C}^\top P \bar{C} + (A+\bar{A})^\top \Pi + \Pi(A+\bar{A}) + \bar{Q}, \\
& \wh{\cQ}(P, \Pi) =(A+\bar{A})^\top (P + \Pi) + (P + \Pi) (A + \bar{A}) + (C + \bar{C})^\top P(C + \bar{C}) + Q + \bar{Q}, \\
& \cM(P) = B^\top P + D^\top P C + S, \quad \bar{\cM}(P, \Pi) = B^\top \Pi + D^\top P \bar{C}, \\
& \wh{\cM}(P, \Pi) = B^\top (P + \Pi) + D^\top P (C + \bar{C}) + S, \quad \cR(P) = R + D^\top P D.
\end{aligned}
\end{equation*}
It is immediate that $\wh{\cQ}(P, \Pi) = \cQ(P) + \bar{\cQ}(P, \Pi)$ and $\wh{\cM}(P, \Pi) = \cM(P) + \bar{\cM}(P, \Pi)$. Moreover, for $P, \Pi \in \bS^d_{+}$, $\Lambda \in L^2(I \times I; \RR^{d \times d})$, and $p \in L^2(I; \mathbb R^d)$, we define
\begin{equation*}
\begin{aligned}
& \Psi_1(P, \Pi, \Lambda)(u, v) \\
= \ & (P + \Pi) \wt{A} G(u, v) + \wt{A}^\top (P + \Pi) G(v, u) + (C + \bar{C})^\top P \wt{C} G(u, v) + \wt{C}^\top P(C+ \bar{C}) G(v, u) \\
& \hspace{0.1in}  + \int_I \wt{C}^\top G(w,u) P \wt{C} G(w,v) dw + (A + \bar{A})^\top \Lambda(u, v) + \Lambda(v, u)^\top (A + \bar{A}) \\
& \hspace{0.1in} + \int_I \Lambda(w,u)^\top \wt{A} G(w,v) dw + \int_I \wt{A}^\top G(w,u) \Lambda(w, v) dw + \wt{Q} G(u, v) \\
& \hspace{0.1in} + \int_I G(w,u) \check{Q} G(w, v) dw,
\end{aligned}
\end{equation*}
\begin{equation*}
\begin{aligned}
\Psi_2(P, \Pi, \Lambda, p)(u) 
&= (A + \bar{A})^\top p(u) + \wt{A}^\top \int_I G(v, u) p(v) dv + (P+\Pi) b + (C + \bar{C})^\top P \sigma \\
& \hspace{0.3in}  + \wt{C}^\top \int_I G(v,u) P \sigma dv + \int_I \frac{1}{2} \big(\Lambda(u, v) + \Lambda(v, u)^\top \big)  b dv,
\end{aligned}
\end{equation*}
and 
$$\Upsilon(P, \Lambda)(u, v) = B^\top \Lambda(u, v) + D^\top P \wt{C} G(u, v).$$

Using the notation introduced above, we derive that
\begin{equation*}
\begin{aligned}
d \mathcal Y_T(t) &= \int_I \EE\big[\lan (\dot{P}_T(t) + \cQ(P_T(t))) X^u(t), X^u(t) \ran \big] du dt \\
&\hspace{0.3in} + \int_I \lan (\dot{\Pi}_T(t) + \bar{\cQ}(P_T(t), \Pi_T(t))) \bar{X}^u(t), \bar{X}^u(t) \ran du dt \\
&\hspace{0.3in} + \int_I\int_I \lan (\dot{\Lambda}_T(t) + \Psi_1(P_T(t), \Pi_T(t), \Lambda_T(t)))(u, v) \bar{X}^v(t), \bar{X}^u(t) \ran dv du dt \\
&\hspace{0.3in}  + 2 \int_I \lan \dot{p}_T(t)(u) + \Psi_2(P_T(t), \Pi_T(t), \Lambda_T(t), p_T(t))(u), \bar{X}^u(t) \ran du dt \\
&\hspace{0.3in} + \int_I \dot{\kappa}_T(t)(u) + \lan \sigma, P_T(t) \sigma \ran + 2 \lan b, p_T(t)(u) \ran du dt \\
&\hspace{0.3in} + \int_I \EE[\lan \alpha^u(t), \cR(P_T(t)) \alpha^u(t) \ran + \lan \alpha^u(t), 2 \mathcal Z^u_T(t) \ran] du dt,
\end{aligned}
\end{equation*}
where 
\begin{equation*}
\begin{aligned}
\mathcal Z^u_T(t) &:= \cM(P_T(t)) X^u(t) + \bar{\cM}(P_T(t), \Pi_T(t)) \bar{X}^u(t) + D^\top P_T(t) \sigma + B^\top p_T(t)(u) \\
& \hspace{0.3in} + \int_I \Upsilon(P_T(t), \Lambda_T(t))(u, v) \bar{X}^v(t) dv.
\end{aligned}
\end{equation*}
By completing the square with respect to $\alpha^u$, we have
\begin{equation*}
\begin{aligned}
& \int_I \EE[\lan \alpha^u(t), \cR(P_T(t)) \alpha^u(t) \ran + \lan \alpha^u(t), 2 \mathcal Z^u_T(t) \ran] du \\
= \ & \int_I \EE \big[\lan \cR(P_T(t)) \{\alpha^u(t) + \cR(P_T(t))^{-1} \mathcal Z^u_T(t)\}, \alpha^u(t) + \cR(P_T(t))^{-1} \mathcal Z^u_T(t) \ran \big] du \\
& \hspace{0.1in} - \int_I \EE\big[\lan \cM(P_T(t))^\top \cR(P_T(t))^{-1} \cM(P_T(t)) X^u(t), X^u(t) \ran \big] du \\
& \hspace{0.1in} - \int_I \lan \{ \cM(P_T(t))^\top \cR(P_T(t))^{-1} \bar{\cM}(P_T(t), \Pi_T(t)) + \bar{\cM}(P_T(t), \Pi_T(t))^\top \cR(P_T(t))^{-1} \cM(P_T(t)) \\
& \hspace{0.5in} + \bar{\cM}(P_T(t), \Pi_T(t))^\top \cR(P_T(t))^{-1} \bar{\cM}(P_T(t), \Pi_T(t))\} \bar{X}^u(t), \bar{X}^u(t) \ran du \\
& \hspace{0.1in} - \int_I \big\lan 2 \wh{M}(P_T(t), \Pi_T(t))^\top \cR(P_T(t))^{-1} (D^\top P_T(t) \sigma + B^\top p_T(t)(u)) \\
& \hspace{0.5in} + 2 \int_I \Upsilon(P_T(t), \Lambda_T(t))(v, u)^\top \cR(P_T(t))^{-1} (D^\top P_T(t) \sigma + B^\top p_T(t)(v)) dv, \bar{X}^u(t) \big\ran du \\
& \hspace{0.1in} - \int_I (D^\top P_T(t) \sigma + B^\top p_T(t)(u))^\top \cR(P_T(t))^{-1} (D^\top P_T(t) \sigma + B^\top p_T(t)(u)) du \\
& \hspace{0.1in} - \int_I\int_I \big\lan \bar{X}^u(t), \big\{\wh{M}(P_T(t), \Pi_T(t))^\top \cR(P_T(t))^{-1} \Upsilon(P_T(t), \Lambda_T(t))(u, v) \\
& \hspace{0.5in} + \Upsilon(P_T(t), \Lambda_T(t))(v, u)^\top \cR(P_T(t))^{-1} \wh{M}(P_T(t), \Pi_T(t)) \\
& \hspace{0.5in} + \int_I \Upsilon(P_T(t), \Lambda_T(t))(w, u)^\top \cR(P_T(t))^{-1} \Upsilon(P_T(t), \Lambda_T(t))(w, v) dw \big\} \bar{X}^v(t) \big\ran dv du.
\end{aligned}
\end{equation*}
Substituting this expression into the preceding equation for $d\mathcal Y_T(t)$, we obtain the desired representation of $d \mathcal Y_T(t)$ and derive the corresponding system of equations for $P_T, \Pi_T, \Lambda_T, p_T$, and $\kappa_T$.

Next, we present the system of equations satisfied by the coefficient functions in \eqref{eq:value_function_ansatz_finite}. We start with the Riccati equation for $P_T$:
\begin{equation}
\label{eq:ODE_P_T}
\begin{cases}
\dot{P}_T(t) + \cQ(P_T(t)) - \cM(P_T(t))^\top \cR(P_T(t))^{-1} \cM(P_T(t)) = 0, \\
P_T(T) = 0 \in \bS^d.
\end{cases}
\end{equation}
Note that \eqref{eq:ODE_P_T} is a standard Riccati equation. A solution to \eqref{eq:ODE_P_T} is a function $P_T \in C^1([0, T]; \bS^d_{+})$ such that \eqref{eq:ODE_P_T} holds for all $t \in [0, T]$. Given such a solution $P_T$ to \eqref{eq:ODE_P_T}, we next consider the equation satisfied by $\bar{\Pi}_T := P_T + \Pi_T$:
\begin{equation}
\label{eq:ODE_Pi_T}
\begin{cases}
\dot{\bar{\Pi}}_T(t) + \wh{\cQ}(P_T(t), \Pi_T(t)) - \wh{\cM}(P_T(t), \Pi_T(t))^\top \cR(P_T(t))^{-1} \wh{\cM}(P_T(t), \Pi_T(t)) = 0, \\
\bar{\Pi}_T(T) = 0 \in \bS^d.
\end{cases}
\end{equation}
The Riccati equation for $\bar{\Pi}_T$ in \eqref{eq:ODE_Pi_T} is also standard. We call $\bar{\Pi}_T$ a solution to \eqref{eq:ODE_Pi_T} if it is a function in $C^1([0, T]; \bS^d_{+})$ such that \eqref{eq:ODE_Pi_T} holds for all $t \in [0, T]$.

Given solutions $P_T$ and $\bar{\Pi}_T$ to \eqref{eq:ODE_P_T} and \eqref{eq:ODE_Pi_T}, respectively, we introduce the following abstract Riccati equation on the Hilbert space $L^2(I \times I; \RR^{d \times d})$:
\begin{equation}
\begin{cases}
\label{eq:ODE_Lambda_T}
\dot{\Lambda}_T(t)(u, v) + F(t, \Lambda_T(t))(u, v) = 0, \\
\Lambda_T(T)(u, v) = 0 \in \mathbb{R}^{d \times d},
\end{cases}
\end{equation}
for a.e. $u, v \in I$, where $F: [0, T] \times L^2(I \times I; \RR^{d \times d}) \to L^2(I \times I; \RR^{d \times d})$ is defined by
\begin{equation*}
\begin{aligned}
F(t, \Lambda)(u, v) &:= \Psi_1(P_T(t), \Pi_T(t), \Lambda)(u, v) - \wh{\cM}(P_T(t), \Pi_T(t))^\top \cR(P_T(t))^{-1} \Upsilon(P_T(t), \Lambda)(u, v) \\
& \hspace{0.5in} - \Upsilon(P_T(t), \Lambda)(v, u)^\top \cR(P_T(t))^{-1} \wh{\cM}(P_T(t), \Pi_T(t)) \\
& \hspace{0.5in} - \int_I \Upsilon(P_T(t), \Lambda)(w, u)^\top \cR(P_T(t))^{-1} \Upsilon(P_T(t), \Lambda)(w, v) dw,
\end{aligned}
\end{equation*}
for $\Lambda \in L^2(I \times I; \RR^{d \times d})$, $t\in [0, T]$, for a.e. $u, v \in I$. It is clear that
$$F(t, \Lambda)(v, u)^\top = F(t, \Lambda^\top)(u, v), \quad \forall t \in [0, T], \, \Lambda \in L^2(I \times I; \mathbb{R}^{d \times d}), \text{ and a.e. } u, v \in I.$$
By a solution to \eqref{eq:ODE_Lambda_T}, we mean a function $\Lambda_T \in C^1([0, T]; L^2(I \times I; \RR^{d \times d}))$ satisfying \eqref{eq:ODE_Lambda_T} for all $t \in [0, T]$. Thus, if $\Lambda_T$ is a solution to \eqref{eq:ODE_Lambda_T}, then it satisfies
\begin{equation}
\label{eq:Lambda_T_symmetric}
\Lambda_T(t)(v, u)^\top = \Lambda_T(t)(u, v), \quad \forall t \in [0, T], \text{ and a.e. } u, v \in I.
\end{equation}
Note that \eqref{eq:ODE_Lambda_T} is not an operator-valued Riccati equation of the type typically considered in the classical literature on control problems in Hilbert spaces.

Given solutions $P_T, \bar{\Pi}_T$, and $\Lambda_T$ to \eqref{eq:ODE_P_T}, \eqref{eq:ODE_Pi_T}, and \eqref{eq:ODE_Lambda_T}, respectively, we introduce the following linear equation with terminal condition for $p_T$ on $L^2(I; \RR^d)$:
\begin{equation}
\label{eq:ODE_p_T}
\begin{cases}
\dot{p}_T(t)(u) + \wt{F}(t, p_T(t))(u) = 0, \\
p_T(T)(u) = 0 \in \mathbb{R}^d,
\end{cases}
\end{equation}
for a.e. $u \in I$, where $\wt{F}:[0, T] \times L^2(I; \RR^d) \to L^2(I; \RR^d)$ is defined by
\begin{equation*}
\begin{aligned}
\wt{F}(t, p)(u) &:= \Psi_2(P_T(t), \Pi_T(t), \Lambda_T(t), p)(u) - \wh{\cM}(P_T(t), \Pi_T(t))^\top \cR(P_T(t))^{-1} \Gamma(P_T(t), p)(u) \\
& \hspace{0.5in} - \int_I \Upsilon(P_T(t), \Lambda_T(t))(v, u)^\top \cR(P_T(t))^{-1} \Gamma(P_T(t), p)(v) dv,
\end{aligned}
\end{equation*}
for all $t \in [0, T]$ and $p \in L^2(I; \RR^d)$ and for a.e. $u \in I$, where $\Gamma: \bS^d \times L^2(I; \RR^d) \to L^2(I;\RR^m)$ is defined by
$$\Gamma(P_T(t), p)(u) := D^\top P_T(t) \sigma + B^\top p(u).$$
Equation \eqref{eq:ODE_p_T} is a standard linear ODE on the Hilbert space $L^2(I; \RR^d)$. A solution to \eqref{eq:ODE_p_T} is a function $p_T \in C^1([0, T]; L^2(I; \RR^d))$ such that \eqref{eq:ODE_p_T} holds for all $t \in [0, T]$.

Finally, we introduce the following linear differential equation for $\kappa_T$:
\begin{equation}
\label{eq:ODE_kappa_T}
\begin{cases}
\dot{\kappa}_T(t)(u) - \langle \Gamma(P_T(t), p_T(t))(u), \cR(P_T(t))^{-1} \Gamma(P_T(t), p_T(t))(u) \rangle \\
\hspace{0.5in} + \langle P_T(t) \sigma, \sigma \rangle + 2 \langle p_T(t)(u), b \rangle  = 0, \\
\kappa_T(T)(u) = 0,
\end{cases}
\end{equation}
for all $t \in [0, T]$ and for a.e. $u \in I$. A solution to \eqref{eq:ODE_kappa_T} is a function $\kappa_T \in C^1([0, T]; L^1(I; \RR))$ such that \eqref{eq:ODE_kappa_T} holds for all $t \in [0, T]$.

Using the notation introduced above in the system of differential equations, we can summarize the differential $d \mathcal Y_T(t)$ as follows:
\begin{equation*}
\begin{aligned}
& d \mathcal Y_T(t) \\
= \ & \int_I \EE\big[\lan (\dot{P}_T(t) + \cQ(P_T(t)) - \cM(P_T(t))^\top \cR(P_T(t))^{-1} \cM(P_T(t))) X^u(t), X^u(t) \ran \big] du dt \\
&\hspace{0.1in} + \int_I \Big\lan \Big\{\frac{d}{dt}(\bar{\Pi}_T(t) - P_T(t)) + \wh{\cQ}(P_T(t), \Pi_T(t)) - \wh{\cM}(P_T(t), \Pi_T(t))^\top \cR(P_T(t))^{-1} \wh{\cM}(P_T(t), \Pi_T(t)) \\
&\hspace{0.8in} - \cQ(P_T(t)) + \cM(P_T(t))^\top \cR(P_T(t))^{-1} \cM(P_T(t)) \Big\} \bar{X}^u(t), \bar{X}^u(t) \Big\ran du dt \\
&\hspace{0.1in} + \int_I\int_I \lan (\dot{\Lambda}_T(t)(u, v) + F(t, \Lambda_T(t))(u, v)) \bar{X}^v(t), \bar{X}^u(t) \ran dv du dt \\
&\hspace{0.1in}  + 2 \int_I \lan \dot{p}_T(t)(u) + \wt{F}(t, p_T(t))(u), \bar{X}^u(t) \ran du dt \\
&\hspace{0.1in} + \int_I \big(\dot{\kappa}_T(t)(u) - \langle \Gamma(P_T(t), p_T(t))(u), \cR(P_T(t))^{-1} \Gamma(P_T(t), p_T(t))(u) \rangle \\
&\hspace{0.8in} + \langle P_T(t) \sigma, \sigma \rangle + 2 \langle p_T(t)(u), b \rangle \big) du dt \\
&\hspace{0.1in} + \int_I \EE \big[\lan \cR(P_T(t)) \{\alpha^u(t) + \cR(P_T(t))^{-1} \mathcal Z^u_T(t)\}, \alpha^u(t) + \cR(P_T(t))^{-1} \mathcal Z^u_T(t) \ran \big] du dt.
\end{aligned}
\end{equation*}
Thus, if $P_T, \bar{\Pi}_T, \Lambda_T, p_T$, and $\kappa_T$ are solutions to the system of differential equations \eqref{eq:ODE_P_T}, \eqref{eq:ODE_Pi_T}, \eqref{eq:ODE_Lambda_T}, \eqref{eq:ODE_p_T}, and \eqref{eq:ODE_kappa_T}, respectively, then $d \mathcal Y_T(t)$ reduces to
\begin{equation}
\label{eq:differential}
d \mathcal Y_T(t) = \int_I \EE \big[\lan \cR(P_T(t)) \{\alpha^u(t) + \cR(P_T(t))^{-1} \mathcal Z^u_T(t)\}, \alpha^u(t) + \cR(P_T(t))^{-1} \mathcal Z^u_T(t) \ran \big] du dt.
\end{equation}

The identity \eqref{eq:differential} is essential for deriving the so-called fundamental relation of the finite-horizon GMFC problem. In particular, it yields a decomposition of the cost functional $J_T(\xi,\alpha)$ in \eqref{eq:cost_functional_finite_horizon} into a control-independent constant term and a nonnegative quadratic functional of the control process, which is crucial for characterizing and solving the finite-horizon GMFC problem.

\begin{proposition}
\label{p:fundamental_relation}
Let $P_T, \bar{\Pi}_T, \Lambda_T, p_T$, and $\kappa_T$ be solutions to the system of differential equations \eqref{eq:ODE_P_T}, \eqref{eq:ODE_Pi_T}, \eqref{eq:ODE_Lambda_T}, \eqref{eq:ODE_p_T}, and \eqref{eq:ODE_kappa_T}, respectively. Then, for all $\xi \in \kL_0^2$ and $\alpha \in \mathcal A[0, T]$,
\begin{equation}
\label{eq:fundamental_relation}
\begin{aligned}
J_T(\xi, \alpha) &= \mathbb{V}_T(0, \xi) + \int_0^T \int_I \EE \big[\lan \cR(P_T(t)) \{\alpha^u(t) + \cR(P_T(t))^{-1} \mathcal Z^u_T(t)\}, \\
&\hspace{1.5in} \alpha^u(t) + \cR(P_T(t))^{-1} \mathcal Z^u_T(t) \ran \big] du dt,
\end{aligned}
\end{equation}
where $\mathbb{V}_T(t, \xi)$ is defined in \eqref{eq:value_function_ansatz_finite}.
\end{proposition}
\begin{proof}
By integrating \eqref{eq:differential} over $[0, T]$, for all $\xi \in \kL_0^2$ and $\alpha \in \mathcal A[0, T]$, we obtain
\begin{equation*}
\begin{aligned}
& \int_0^T \int_I \EE \big[\lan \cR(P_T(t)) \{\alpha^u(t) + \cR(P_T(t))^{-1} \mathcal Z^u_T(t)\}, \alpha^u(t) + \cR(P_T(t))^{-1} \mathcal Z^u_T(t) \ran \big] du dt \\
= \ & \mathcal V_T(T) - \mathcal V_T(0) + \int_0^T \int_I \mathbb E\big[\langle X^u(s), Q X^u(s) \rangle + 2\langle \alpha^u(s), S X^u(s) \rangle + \langle \alpha^u(s), R \alpha^u(s) \rangle \\
&\hspace{0.5in} + \langle \bar{X}^u(s), \bar{Q} \bar{X}^u(s) \rangle + \langle \bar{X}^u(s), \wt{Q} [G\bar{X}(s)]^u\rangle + \langle [G\bar{X}(s)]^u, \check{Q} [G\bar{X}(s)]^u\rangle \big] du ds \\
= \ & - \mathcal V_T(0) + J_T(\xi, \alpha) = - \mathbb{V}_T(0, \xi) + J_T(\xi, \alpha),
\end{aligned}
\end{equation*}
where we use the fact that $X^u(0) = \xi^u$ and the terminal conditions for $P_T, \bar{\Pi}_T, \Lambda_T, p_T$, and $\kappa_T$. Then, the desired result follows.
\end{proof}

\subsection{Solvability of the finite-horizon control problem}
\label{s:main_result_finite_control}

To ensure the solvability of the system of differential equations \eqref{eq:ODE_P_T}, \eqref{eq:ODE_Pi_T}, \eqref{eq:ODE_Lambda_T}, \eqref{eq:ODE_p_T}, and \eqref{eq:ODE_kappa_T}, as well as that of the finite-horizon GMFC problem \eqref{eq:state_process}-\eqref{eq:cost_functional_finite_horizon}, we impose the following assumption.

\begin{assumption}
\label{a:coefficients_solvability_finite}
The matrices $Q, \bar{Q} \in \mathbb{S}^d$, $S \in \mathbb{R}^{m \times d}$, and $R \in \mathbb{S}^m_{++}$ satisfy
$$Q - S^\top R^{-1} S \in \mathbb S^d_{++}, \quad Q + \bar{Q} - S^\top R^{-1} S \in \mathbb S^d_{++}.$$
Moreover, $\check{Q}$ is positive semidefinite, and the kernel $\wt{Q}G(u,v)$ defines a positive operator on $L^2(I; \mathbb R^d)$, i.e.,
$$\int_I \int_I \lan \phi^u, \wt{Q} G(u, v) \phi^v \ran dv du \geq 0, \quad \forall \phi \in L^2(I; \mathbb{R}^d).$$
\end{assumption}
\begin{remark}
\label{r:cost_function_finite}
We give the following remarks regarding Assumption \ref{a:coefficients_solvability_finite}:
\begin{itemize}
\item[\textnormal{(i)}] Under the above assumption, one can show that
$$J_T(\xi, \alpha) \geq 0, \quad \forall \xi \in \kL_0^2 \text{ and } \alpha \in \mathcal{A}[0, T].$$ 
Indeed, for all $x \in \mathbb{R}^d$ and $a \in \mathbb{R}^m$,
\begin{equation*}
\lan x, Qx \ran + 2 \lan a, Sx \ran + \lan a, Ra\ran = \lan x, (Q - S^\top R^{-1} S) x \ran + \lan a + R^{-1}Sx, R(a + R^{-1}Sx) \ran.
\end{equation*}
Thus, the nonnegativity of $J_T$ follows by completing the square. In fact, this observation resolves the main challenge pointed out in \cite[Remark 6.2]{Crescenze-Feo-Pham-2025}, where the cross term in state and control is excluded from the cost functional $J_T$ because its contribution is not automatically nonnegative, which in turn leads to difficulties in establishing the unique solvability of $\Lambda_T$ in \eqref{eq:ODE_Lambda_T}.

\item[\textnormal{(ii)}] To illustrate the assumption that $\wt{Q}G(u,v)$ is a positive kernel on $L^2(I; \mathbb R^d)$, we provide several explicit examples. Assume that $\wt{Q}$ is positive semidefinite, and let $\wt{Q}^{1/2}$ denote its symmetric positive semidefinite square root. Consider the following cases:
\begin{itemize}
\item $G(u, v) = g$ for some constant $g \geq 0$ for all $u, v \in I$. This case reduces to the standard mean field control case, see, for instance, \cite{Bayraktar-Jian-2025}; 
\item $G(u, v) = g(u) g(v)$ for all $u, v \in I$, where $g \in L^2(I; \mathbb{R})$. Then, for all $\phi \in L^2(I; \mathbb{R}^d)$, 
$$\int_I \int_I \lan \phi^u, \wt{Q} G(u, v) \phi^v \ran dv du = \Big|\int_I g(u) \wt{Q}^{1/2} \phi^u du \Big|^2 \geq 0$$
and $\|G\|_{L^2(I \times I)} = \|g\|^2_{L^2(I)} < \infty$. More generally, we could let
$G(u, v) = \sum_{i = 1}^{k} g_i(u) g_i(v)$ for $g_1, \dots, g_k \in L^2(I; \mathbb{R})$; 
\item $G(u,v) = \sum_{i=1}^{\infty} \lambda_i e_{i}(u) e_{i}(v)$ for all $u, v \in I$, where $\{e_i\}_{i \geq 1}$ is an orthonormal family in $L^2(I; \mathbb{R})$, and $\lambda_i > 0$ with $\sum_{i=1}^{\infty} \lambda_i^2 < \infty$;
\item $G$ is a positive kernel on $L^2(I;\mathbb{R})$. Typical examples include the Gaussian kernel, i.e., $G(u, v) = e^{-(u-v)^2}$, and the Brownian covariance kernel, i.e., $G(u, v) = \min(u, v)$, for all $u, v \in I$.    
\end{itemize}

\item[\textnormal{(iii)}] The positivity assumption on the kernel $\wt{Q}G(u,v)$, viewed as an operator on $L^2(I; \mathbb R^d)$, is imposed primarily to prove the unique solvability of $\Lambda_T$ in \eqref{eq:ODE_Lambda_T}; see Proposition \ref{p:solvability_system_finite}. It also serves as a sufficient condition for the coercivity of the finite-horizon cost functional $J_T$ in \eqref{eq:cost_functional_finite_horizon}. 
\end{itemize}
\end{remark}

\begin{proposition}
\label{p:solvability_system_finite}
Let Assumption \ref{a:coefficients_solvability_finite} hold. Then, for any $T > 0$, the differential Riccati equations \eqref{eq:ODE_P_T} and \eqref{eq:ODE_Pi_T} admit a unique solution pair $(P_T, \bar{\Pi}_T) \in C^1([0, T]; \mathbb{S}^d_{+}) \times C^1([0, T]; \mathbb{S}^d_{+})$. Moreover, the abstract Riccati equation \eqref{eq:ODE_Lambda_T} admits a unique solution $\Lambda_T \in C^1([0, T]; L^2(I \times I; \RR^{d \times d}))$, while the linear equations \eqref{eq:ODE_p_T} and \eqref{eq:ODE_kappa_T} admit unique solutions $p_T \in C^1([0, T]; L^2(I; \RR^d))$ and $\kappa_T \in C^1([0, T]; L^{1}(I; \RR))$, respectively.
\end{proposition}
\begin{proof}
First, under Assumption \ref{a:coefficients_solvability_finite}, the unique solvability of the differential Riccati equations \eqref{eq:ODE_P_T} and \eqref{eq:ODE_Pi_T} follows from \cite[Lemma 2.1]{Sun-Yong-2024}. Moreover, we have $P_T(t) \geq 0$ and $\bar{\Pi}_T(t) \geq 0$ for each $t \in [0, T]$. Since $R \in \mathbb{S}^m_{++}$, $\cR(P_T(t))^{-1}$ is well defined and $\|\cR(P_T(t))^{-1}\| \leq \|R^{-1}\|$, which is uniformly bounded for all $t \in [0, T]$.  

Next, we prove the unique solvability of the abstract Riccati equation \eqref{eq:ODE_Lambda_T}. We first show that \eqref{eq:ODE_Lambda_T} admits a unique solution on a local time horizon $[T -\delta, T]$ for some $\delta > 0$, and then extend the solution to the entire interval $[0, T]$ by obtaining an a priori estimate for $\Lambda_T$.

\textit{Step 1}: Let $\delta, r > 0$, and consider the Banach space $C([T- \delta, T]; L^2(I \times I; \mathbb{R}^{d \times d}))$ equipped with the norm $\|\Lambda\|_c^2 = \sup_{t \in [T-\delta, T]} \|\Lambda(t)\|^2_{L^2(I \times I)}$. Define
$$B(r) = \big\{\Lambda \in C([T- \delta, T]; L^2(I \times I; \mathbb{R}^{d \times d})): \sup_{t \in [T - \delta, T]} \|\bm{\Lambda}(t)\|_{op} \leq r \big\},$$
where $\bm{\Lambda}(t)$ is the Hilbert--Schmidt integral operator induced by the kernel $\Lambda(t) \in L^2(I \times I; \mathbb{R}^{d \times d})$; see the definition in \eqref{eq:Hilbert_Schmidt_operator}. Then, $B(r)$ is a closed and convex subset of $C([T- \delta, T]; L^2(I \times I; \mathbb{R}^{d \times d}))$. We define a mapping $\mathcal{T}$ on $B(r)$ as follows:
$$\mathcal{T}(\Lambda_T)(t) = \int_t^T F(s, \Lambda_T(s)) ds, \quad \forall t \in [T - \delta, T].$$
Let $\mathcal{T}(\bm{\Lambda}_T)$ denote the integral operator induced by the kernel $\mathcal{T}(\Lambda_T)$. Since 
$$\|\bm{F}(t, \Lambda_T(t))\|_{op} \leq c_T \big(1 + \|\bm{\Lambda}_T(t)\|_{op} + \|\bm{\Lambda}_T(t)\|^2_{op} \big), \quad \forall t \in [T - \delta, T]$$
for some constant $c_T$, we obtain
$$\|\mathcal{T}(\bm{\Lambda}_T)(t)\|_{op} \leq c_T \int_{T-\delta}^T \big(1 + \|\bm{\Lambda}_T(t)\|_{op} + \|\bm{\Lambda}_T(t)\|^2_{op} \big) dt \leq c_T \delta(1+r+r^2)$$
for all $t \in [T-\delta, T]$. By choosing $\delta > 0$ sufficiently small such that $c_T \delta(1+r+r^2) < r$, we see that $\mathcal{T}$ maps $B(r)$ into itself. Next, we show that $\mathcal{T}$ is a contraction mapping on $B(r)$. Note that, for $\Lambda_T^1, \Lambda_T^2 \in C([T- \delta, T]; L^2(I \times I; \mathbb{R}^{d \times d}))$,
$$\|F(t, \Lambda^1_T(t)) - F(t, \Lambda_T^2(t))\|_{L^2(I \times I)} \leq c_T \big(1 + \|\bm{\Lambda}_T^1(t)\|_{op} + \|\bm{\Lambda}_T^2(t)\|_{op} \big) \|\Lambda_T^1(t) - \Lambda_T^2(t)\|_{L^2(I \times I)},$$
we then derive the following estimate:
\begin{equation*}
\begin{aligned}
& \sup_{t \in [T-\delta, T]} \|\mathcal{T}(\Lambda^1_T)(t) - \mathcal{T}(\Lambda_T^2)(t)\|_{L^2(I \times I)} \\
\leq \ & c_T \int_{T - \delta}^T \big(1 + \|\bm{\Lambda}_T^1(t)\|_{op} + \|\bm{\Lambda}_T^2(t)\|_{op} \big) \|\Lambda_T^1(t) - \Lambda_T^2(t)\|_{L^2(I \times I)} dt \\
\leq \ & c_T \delta (1+2r) \sup_{t \in [T-\delta, T]} \|\Lambda_T^1(t) - \Lambda_T^2(t)\|_{L^2(I \times I)}. 
\end{aligned}
\end{equation*}
Hence, we may further choose $\delta > 0$ sufficiently small such that $c_T \delta (1+2r) < 1$. Then, $\mathcal{T}$ is a contraction mapping on $B(r)$. By the Banach fixed point theorem, the equation \eqref{eq:ODE_Lambda_T} admits a unique solution on $[T - \delta, T]$.

\textit{Step 2}: Let $T_0 = T - \delta$. By the identity \eqref{eq:Lambda_T_symmetric}, $\bm{\Lambda}_T(t)$ is a self-adjoint operator for each $t \in [T_0, T]$. Recall that the norm of a self-adjoint operator can be expressed as
$$\|\bm{\Lambda}_T(t)\|_{op} = \sup_{\|\phi\|_{L^2(I)} \neq 0} \frac{|\lan \bm{\Lambda}_T(t)(\phi), \phi \ran_{L^2}|}{\|\phi\|^2_{L^2(I)}}.$$
Next, we give an estimate for the right-hand side. Set $b = \sigma = 0$. Then, by uniqueness of the solutions to the linear equations \eqref{eq:ODE_p_T} and \eqref{eq:ODE_kappa_T}, we have $p_T(t)(u) = 0$ and $\kappa_T(t) = 0$ for all $t \in [T_0, T]$ and a.e. $u \in I$. Since $\mathcal{R}(P_T(t))$ is positive definite, applying the fundamental relation \eqref{eq:fundamental_relation} with $\alpha = 0$, and replacing the initial time $0$ with $t \in [T_0, T]$ and the corresponding cost functional by $J_T(t, \xi, \alpha)$, we obtain
$$J_T(t, \xi, 0) \geq \mathbb{V}_T(t, \xi) \geq \int_I \int_I \langle \bar{\xi}^u, \Lambda_T(t)(u, v) \bar{\xi}^v \rangle dv du $$
for all $\xi \in \kL^2_{t}$. By standard SDE estimates and Gr\"onwall's inequality, there exists a constant $c_T$ such that, for all $\xi \in \kL_t^2$ and $\alpha \in \mathcal{A}[T_0, T]$, 
$$\int_I \mathbb{E} \big[\sup_{s \in [t, T]} |X^u(s)|^2 \big] du \leq c_T \Big((T-t)(|b|^2 + |\sigma|^2) + \int_I \mathbb{E}[|\xi^u|^2] du + \int_I \int_t^T \mathbb{E}[|\alpha^u(s)|^2] dsdu\Big).$$
Consequently, by the definition of $J_T$ in \eqref{eq:cost_functional_finite_horizon}, possibly with a different constant $c_T$,
$$|J_T(t, \xi, \alpha)| \leq c_T \Big((T-t)(|b|^2 + |\sigma|^2) + \int_I \mathbb{E}[|\xi^u|^2] du + \int_I \int_t^T \mathbb{E}[|\alpha^u(s)|^2] dsdu\Big).$$
In particular, when $b = \sigma = 0$ and $\alpha = 0$, we obtain the following estimate:
$$|J_T(t, \xi, 0)| \leq c_T \int_I \mathbb{E}[|\xi^u|^2] du,$$
which implies that
$$\int_I \int_I \langle \bar{\xi}^u, \Lambda_T(t)(u, v) \bar{\xi}^v \rangle dv du \leq c_T \int_I \mathbb{E}[|\xi^u|^2] du.$$
On the other hand, by \eqref{eq:fundamental_relation} again, letting $\alpha^u(t) = - \mathcal{R}(P_T(t))^{-1}\mathcal{Z}^u_T(t)$, the cost functional reduces to
\begin{equation*}
J_T(t, \xi, \alpha) = \int_I \EE[\langle \xi^u, P_T(t) \xi^u \rangle] du + \int_I \langle \bar{\xi}^u, \Pi_T(t) \bar{\xi}^u \rangle du + \int_I \int_I \langle \bar{\xi}^u, \Lambda_T(t)(u, v) \bar{\xi}^v \rangle dv du.
\end{equation*}
Note that, by Assumption \ref{a:coefficients_solvability_finite}, we have $J_T(t, \xi, \alpha) \geq 0$ for all $\xi \in \kL_t^2$ and $\alpha \in \mathcal{A}[T_0, T]$. Thus, we deduce that
\begin{equation*}
\begin{aligned}
\int_I \int_I \langle \bar{\xi}^u, \Lambda_T(t)(u, v) \bar{\xi}^v \rangle dv du &\geq - \int_I \EE[\langle \xi^u - \bar{\xi}^u, P_T(t) (\xi^u - \bar{\xi}^u) \rangle] du - \int_I \langle \bar{\xi}^u, \bar{\Pi}_T(t) \bar{\xi}^u \rangle du \\
& \geq - c_T \int_I \mathbb{E}[|\xi^u|^2] du
\end{aligned}
\end{equation*}
for some possibly different constant $c_T > 0$. Therefore, one has the following estimate:
\begin{equation*}
\Big|\int_I \int_I \langle \bar{\xi}^u, \Lambda_T(t)(u, v) \bar{\xi}^v \rangle dvdu \Big| \leq c_T \int_I \mathbb{E}[|\xi^u|^2] du.
\end{equation*}
Taking $\xi^u = \phi^u$ for some deterministic function $\phi$, we obtain $\|\bm{\Lambda}_T(t)\|_{op} \leq c_T$ for all $t \in [T_0, T]$.  

\textit{Step 3}: From the result established in \textit{Step 1}, the length $\delta$ of the local existence interval depends on $\Lambda_T(t)$ only through the norm $\|\bm{\Lambda}_T(t)\|_{op}$, for which we have an a priori estimate in \textit{Step 2}. Thus, the local existence argument can be iterated backward in time a finite number of times, yielding the unique global solvability of \eqref{eq:ODE_Lambda_T} on the entire interval $[0, T]$.  

Next, given the solutions $P_T, \bar{\Pi}_T$, and $\Lambda_T$ to \eqref{eq:ODE_P_T}, \eqref{eq:ODE_Pi_T}, and \eqref{eq:ODE_Lambda_T}, respectively, equation \eqref{eq:ODE_p_T} is linear in $p_T$ on $L^2(I; \mathbb{R}^d)$. Thus, standard ODE theory gives a unique solution to \eqref{eq:ODE_p_T}. Finally, equation \eqref{eq:ODE_kappa_T} can be solved pointwise in $u$ by integrating both sides in time, and uniqueness follows from the classical theory of linear ODEs.
\end{proof}

In what follows, we establish a verification theorem, and derive the optimal control together with an explicit characterization of the value function for the finite-horizon GMFC problem \eqref{eq:state_process}-\eqref{eq:cost_functional_finite_horizon}. To simplify the notation, we introduce
\begin{equation*}
\begin{aligned}
& \Theta_T^*(t) := - \cR(P_T(t))^{-1} \cM(P_T(t)), \\
& \bar{\Theta}_T^*(t) := - \cR(P_T(t))^{-1} \wh{\cM}(P_T(t), \Pi_T(t)), \\
& \theta^*_T(t)(u) := - \cR(P_T(t))^{-1} \Gamma(P_T(t), p_T(t))(u), \\
& \wt{\Theta}^*_T(t)(u, v) := - \cR(P_T(t))^{-1}
\Upsilon(P_T(t), \Lambda_T(t))(u, v)
\end{aligned}
\end{equation*}
for all $t \in [0, T]$ and a.e. $u, v \in I$.

\begin{proposition}
\label{p:solvability_finite_control}
Assume that Assumption \ref{a:coefficients_solvability_finite} holds. Let $P_T$ and $\bar{\Pi}_T$ be the unique solutions to the system of standard Riccati differential equations \eqref{eq:ODE_P_T} and \eqref{eq:ODE_Pi_T} on $[0, T]$, respectively. Let $\Lambda_T$ be the unique solution to the abstract Riccati equation \eqref{eq:ODE_Lambda_T} on $[0, T]$, and let $p_T$ and $\kappa_T$ be the unique solutions to the linear equations \eqref{eq:ODE_p_T} and \eqref{eq:ODE_kappa_T} on $[0, T]$, respectively. Define $\Pi_T = \bar{\Pi}_T - P_T$. Then,
\begin{itemize}
\item The unique optimal control $\alpha_T = (\alpha_T^u)_{u \in I}$ for the finite-horizon GMFC problem \eqref{eq:cost_functional_finite_horizon} is given by the following feedback form:
\begin{equation}
\label{eq:optimal_control_finite}
\alpha_T^u(t) = \Theta^*_T(t) (X_T^u(t) - \bar{X}_T^u(t)) + \bar{\Theta}^*_T(t) \bar{X}_T^u(t) + \theta^*_T(t)(u) + \int_I \wt{\Theta}^*_T(t)(u, v)\bar{X}_T^v(t) dv
\end{equation}
for all $t \in [0, T]$, where $X_T = (X_T^u)_{u \in I}$ is the solution to the corresponding closed-loop equation 
\begin{equation}
\label{eq:optimal_state_finite}
\begin{cases}
dX_T^u(t) = \big\{(A+B\Theta^*_T(t)) X_T^u(t) + [\bar{A} + B (\bar{\Theta}^*_T(t) - \Theta^*_T(t))] \bar X_T^u(t) + B\theta^*_T(t)(u) + b \\
\hspace{1in} + \int_I \big(\wt{A}G(u, v) + B \wt{\Theta}^*_T(t)(u, v) \big) \bar{X}_T^v(t) dv\big\}dt 
\\
\hspace{0.7in} +\big\{(C+D\Theta^*_T(t)) X_T^u(t) + [\bar{C} + D (\bar{\Theta}^*_T(t) - \Theta^*_T(t))] \bar X_T^u(t) + D \theta_T^*(t)(u) + \sigma \\
\hspace{1in} + \int_I \big(\wt{C}G(u, v) + D \wt{\Theta}^*_T(t)(u, v)\big) \bar{X}_T^v(t) dv\big\}dW^u(t), \quad t \in [0, T], \\
X^u(0) = \xi^u, \quad u \in I.
\end{cases}
\end{equation}

\item The value function $V_T(\xi) = \inf_{\alpha \in \mathcal A[0, T]} J_T(\xi, \alpha)$ of the finite-horizon GMFC problem is given by $V_T(\xi) = \mathbb{V}_T(0, \xi)$, where $\mathbb{V}_T$ is defined in \eqref{eq:value_function_ansatz_finite}.
\end{itemize}
\end{proposition}
\begin{proof}
This is a direct consequence of Proposition \ref{p:solvability_system_finite} and the fundamental relation \eqref{eq:fundamental_relation}. First, it is clear that $\alpha_T \in \mathcal{A}[0, T]$. Then, by the fundamental relation \eqref{eq:fundamental_relation}, we have $J_T(\xi, \alpha) \geq \mathbb{V}_T(0, \xi)$ for all $\alpha \in \mathcal{A}[0, T]$. Moreover, $J_T(\xi, \alpha) = \mathbb{V}_T(0, \xi)$ if and only if $\alpha$ takes the form of $\alpha_T$ in \eqref{eq:optimal_control_finite}. The uniqueness of the solution to the closed-loop system \eqref{eq:optimal_state_finite} completes the proof.
\end{proof}


\section{Ergodic control problem}
\label{s:ergodic_control}

We now turn to the ergodic formulation of the linear-quadratic graphon mean field control problem. To proceed, we first define the admissible control set for the ergodic control problem. Given $\xi \in \kL^2_0$, let $\mathcal{A}^{\xi}[0, \infty)$ denote the set of $I$-indexed collections of $\RR^m$-valued stochastic processes $\alpha = (\alpha^u)_{u\in I}$ satisfying the following conditions:
\begin{itemize}
\item For every $T > 0$, $\alpha|_{[0, T]} \in \cA[0, T]$, where $\alpha|_{[0, T]}$ is the restriction of $\alpha$ to $[0, T]$;
\item The corresponding state process $X = (X^u)_{u \in I}$ in \eqref{eq:state_process} satisfies the transversality condition
$$\lim_{T \to \infty} \frac{1}{T} \mathbb{E}\Big[\int_{I}|X^u(T)|^2 du \Big] = 0,$$
and the long-run average integrability condition
$$\limsup_{T \to \infty} \frac{1}{T} \mathbb{E} \Big[\int_I \int_0^T(|X^u(t)|^2 + |\alpha^u(t)|^2) dt du \Big] < \infty.$$
\end{itemize}

Whenever no confusion can arise, we simply write $\mathcal{A}[0, \infty)$ for $\mathcal{A}^{\xi}[0, \infty)$. The corresponding cost functional is then defined as follows:

\textbf{Cost functional.} Given $\xi \in \mathfrak{L}^2_0$, the objective is to minimize the following cost functional over $\alpha \in \mathcal A[0, \infty)$:
\begin{equation}
\begin{aligned}
\label{eq:cost_functional_infinite_horizon}
J_{\infty}(\xi,\alpha) &:= \; \limsup_{T \to \infty} \frac{1}{T} \int_I \int_0^T \EE\big[ \langle X^u(s), Q X^u(s) \rangle + 2\langle \alpha^u(s), S X^u(s) \rangle + \langle \alpha^u(s), R \alpha^u(s) \rangle \\
&\hspace{1.1in} + \langle \bar{X}^u(s), \bar{Q} \bar{X}^u(s) \rangle + \langle \bar{X}^u(s), \wt{Q} [G\bar{X}(s)]^u\rangle \\
&\hspace{1.1in}  + \langle [G\bar{X}(s)]^u, \check{Q} [G\bar{X}(s)]^u\rangle \big] ds d u.
\end{aligned}
\end{equation}

\subsection{Homogeneous state equation and stabilizability condition}
\label{s:stabilizability_condition}

In what follows, we consider the homogeneous controlled state dynamics over the infinite horizon $[0, \infty)$ and introduce an appropriate stabilizability condition for the system. This condition plays a crucial role in establishing the solvability of the system of algebraic Riccati equations associated with the ergodic GMFC problem. 

Given $\xi \in \kL^2_0$, we consider a collection of homogeneous state processes $X = (X^u)_{u \in I}$, controlled by the collection $\alpha = (\alpha^u)_{u \in I} \in \sU^{\xi}[0, \infty)$, governed by
\begin{equation}
\label{eq:homo_state_process}
\begin{cases}
dX^u(t) = \big\{AX^u(t) + \bar{A} \bar X^u(t) + \wt{A} [G\bar{X}(t)]^u + B\alpha^u(t) \big\}dt 
\\
\hspace{0.7in} +\big\{CX^u(t) + \bar{C} \bar X^u(t) + \wt{C}[G\bar{X}(t)]^u + D \alpha^u(t) \big\}dW^u(t), \\
X^u(0) = \xi^u.
\end{cases}
\end{equation}
Here, $\sU^{\xi}[0, \infty)$ denotes the set of $I$-collections of $\RR^m$-valued stochastic processes $\alpha = (\alpha^u)_{u\in I}$ such that $\alpha|_{[0, T]} \in \cA[0, T]$ for all $T > 0$, and $\int_I\int_0^\infty \EE[|\alpha^u(t)|^2]dtd u <\infty$. When the dependence on $\xi$ is clear, we simply write $\sU[0, \infty)$. A standard contraction argument shows that, for any $\alpha \in \sU[0, \infty)$, the homogeneous system \eqref{eq:homo_state_process} admits a unique solution $X = (X^u)_{u \in I}$ on every finite time interval. Moreover, for each $T > 0$,
$$\int_I\EE \Big[\sup_{t \in [0, T]} |X^u(t)|^2 \Big]d u <\infty.$$

Next, we introduce a stabilizability condition for the homogeneous system \eqref{eq:homo_state_process}.

\begin{definition}
\label{d:L^2-stabilizability}
The homogeneous system \eqref{eq:homo_state_process} is graphon-$L^2$-stabilizable if there exists a tuple $(\Theta, \bar{\Theta}, \wt{\Theta}) \in \RR^{m \times d} \times \RR^{m \times d} \times L^2(I \times I; \RR^{m \times d})$, called the graphon-$L^2$-stabilizer of \eqref{eq:homo_state_process}, such that if $X = (X^u)_{u \in I}$ is the solution to the following homogeneous closed-loop system
\begin{equation}
\label{eq:closed_loop_homo_state_process}
\begin{cases}
dX^u(t) = \big\{(A+B \Theta) X^u(t) + (\bar{A} + B(\bar{\Theta} - \Theta)) \bar{X}^u(t) \\
\hspace{1in} + \int_I \big(\wt{A} G(u, v) + B \wt{\Theta}(u, v)\big) \bar{X}^v(t) dv \big\}dt 
\\
\hspace{0.7in} +\big\{(C+D \Theta) X^u(t) + (\bar{C} + D(\bar{\Theta} - \Theta)) \bar{X}^u(t) \\
\hspace{1in} + \int_I \big(\wt{C} G(u, v) + D \wt{\Theta}(u, v)\big) \bar{X}^v(t) dv \big\} dW^u(t), \\
X^u(0) = \xi^u
\end{cases}
\end{equation}
and
$$\alpha^u(t) = \Theta(X^u(t) - \bar{X}^u(t)) + \bar{\Theta} \bar{X}^u(t) + \int_I \wt{\Theta}(u, v) \bar{X}^v(t) dv$$
for $t \in [0, \infty)$, then
$$\EE \Big[ \int_0^{\infty} \int_I \big(|X^u(t)|^2 + |\alpha^u(t)|^2 \big) du dt \Big] < \infty.$$
\end{definition}

To guarantee that the homogeneous system \eqref{eq:homo_state_process} is graphon-$L^2$-stabilizable, we impose the following assumption.

\begin{assumption}
\label{a:stabilizability}
The controlled ODE 
\begin{equation}
\label{eq:homo_ode}
\begin{cases}
\frac{d}{dt} \bar{X}^u(t) = (A + \bar{A}) \bar{X}^u(t) + \wt{A}[G\bar{X}(t)]^u + B\bar{\alpha}^u(t), \\
\bar{X}^u(0) = \bar{\xi}^u
\end{cases}
\end{equation}
is graphon-$L^2$-stabilizable, i.e., there exist a matrix $\bar{\Theta} \in \RR^{m \times d}$ and $\wt{\Theta} \in L^2(I \times I; \RR^{m \times d})$ such that, for any collection of initial conditions $\bar{\xi} = (\bar{\xi}^u)_{u \in I}$ with $\bar{\xi}^u = \EE[\xi^u]$, where $\xi \in \kL_0^2$, the solution to
\begin{equation}
\label{eq:homo_ode_closed_form}
\begin{cases}
\frac{d}{dt} \bar{X}^u(t) = (A + \bar{A} + B \bar{\Theta}) \bar{X}^u(t) + \int_I \big(\wt{A} G(u, v) + B \wt{\Theta}(u, v) \big) \bar{X}^v(t) dv, \\
\bar{X}^u(0) = \bar{\xi}^u
\end{cases}
\end{equation}
satisfies
$$\int_0^{\infty} \int_I |\bar{X}^u(t)|^2 du dt < \infty.$$
In this case, $(\bar{\Theta}, \wt{\Theta})$ is called a stabilizer of the controlled ODE \eqref{eq:homo_ode}. Moreover, the controlled SDE, denoted by $[A, C; B, D]$,
\begin{equation}
\label{eq:homo_SDE}
\begin{cases}
dX^u(t) = (AX^u(t) + B\alpha^u(t)) dt 
+ (CX^u(t) + D \alpha^u(t)) dW^u(t), \\
X^u(0) = \xi^u
\end{cases}
\end{equation}
is also graphon-$L^2$-stabilizable, i.e., there exists a matrix $\Theta \in \RR^{m \times d}$ such that, for any collection of initial conditions $\xi = (\xi^u)_{u \in I} \in \kL_0^2$, 
the solution to the closed-loop system
\begin{equation}
\label{eq:homo_SDE_closed_form}
\begin{cases}
dX^u(t) = (A + B \Theta) X^u(t) dt 
+ (C + D \Theta) X^u(t) dW^u(t), \\
X^u(0) = \xi^u,
\end{cases}
\end{equation}
which is denoted by $[A+B\Theta, C+D\Theta]$, satisfies
$$\EE \Big[\int_I \int_0^{\infty}  |X^u(t)|^2 dt du \Big] < \infty.$$
In this case, we call $\Theta$ a stabilizer of the controlled SDE \eqref{eq:homo_SDE}.
\end{assumption}

The equation \eqref{eq:homo_ode_closed_form} is a standard linear ODE on the Hilbert space $L^2(I; \mathbb{R}^d)$. We define the corresponding linear operator as follows: for $\phi \in L^2(I; \mathbb{R}^d)$ and $u \in I$,
\begin{equation*}
\mathcal{L}_{\text{s}}(\phi)(u) = (A + \bar{A} + B \bar{\Theta}) \phi^u + \int_I \big(\wt{A} G(u, v) + B \wt{\Theta}(u, v)\big) \phi^v dv.
\end{equation*}
It is clear that $\mathcal{L}_{\text{s}}$ is a bounded linear operator on $L^2(I;\mathbb{R}^d)$. The corresponding adjoint operator $\mathcal{L}_{\text{s}}^*$ is given by 
$$\mathcal{L}_{\text{s}}^*(\phi)(u) = (A + \bar{A} + B \bar{\Theta})^\top \phi^u + \int_I \big(\wt{A}^\top G(v, u) + \wt{\Theta}(v, u)^\top B^\top \big) \phi^v dv.$$
By Assumption \ref{a:stabilizability} and the Datko--Pazy theorem (for example, see \cite[Theorem 1.8, Chapter V]{Engel-Nagel-2000}), the semigroup $\{e^{t\mathcal{L}_{\text{s}}}\}_{t \geq 0}$ is uniformly exponentially stable; that is, there exist positive constants $K$ and $\lambda$, depending only on the operator $\mathcal{L}_{\text{s}}$, such that
\begin{equation}
\label{eq:operator_s_uniform_stable}
\|e^{t\mathcal{L}_{\text{s}}}\|_{op} \leq K e^{-\lambda t}, \quad \forall t \geq 0,
\end{equation}
where the semigroup $\{e^{t\mathcal{L}_{\text{s}}}\}_{t \geq 0}$ is generated by $\mathcal{L}_{\text{s}}$ and is defined by
$$e^{t\mathcal{L}_{\text{s}}} = \sum_{k=0}^{\infty} \frac{t^k \mathcal{L}_{\text{s}}^k}{k!}.$$

\begin{remark}
\label{r:stabilizability}
In Assumption \ref{a:stabilizability}, we require that $\mathcal{L}_{\text{s}}$ be a strictly stable bounded operator on $L^2(I; \mathbb{R}^d)$. A sufficient condition is that there exist $\bar{\Theta}$ and $\wt{\Theta}$ such that \textnormal{(i)} the local matrix part $A + \bar{A} + B \bar{\Theta}$ is strongly stable; and \textnormal{(ii)} the integral operator part is sufficiently small in operator norm. More precisely, if we find $\bar{\Theta}$ and $\wt{\Theta}$ such that $A + \bar{A} + B \bar{\Theta} + (A + \bar{A} + B \bar{\Theta})^\top \leq - 2a \bm{I}_d$ for some $a > 0$, and $\|\wt{G}\|_{op} < a$, where $\wt{G}(u, v) := \wt{A} G(u, v) + B \wt{\Theta}(u, v)$, then $\mathcal{L}_{\text{s}}$ is exponentially stable.
\end{remark}

Next, we show that Assumption \ref{a:stabilizability} ensures the graphon-$L^2$-stabilizability of the homogeneous system
\eqref{eq:homo_state_process}. 

\begin{lemma}
\label{l:stabilizability_homo_SDE}
Suppose that Assumption \ref{a:stabilizability} holds. Then, the homogeneous system
\eqref{eq:homo_state_process} is graphon-$L^2$-stabilizable.  
\end{lemma}
\begin{proof}
First, for any $\alpha = (\alpha^u)_{u \in I} \in \sU[0, \infty)$, the homogeneous system \eqref{eq:homo_state_process} admits a unique solution. Taking expectations in \eqref{eq:homo_state_process}, we obtain the controlled ODE \eqref{eq:homo_ode}. By Assumption \ref{a:stabilizability}, it is graphon-$L^2$-stabilizable. Thus, there exist a matrix $\bar{\Theta} \in \RR^{m \times d}$ and $\wt{\Theta} \in L^2(I \times I; \RR^{m \times d})$ such that the solution to \eqref{eq:homo_ode_closed_form} satisfies
$\int_0^{\infty} \int_I |\bar{X}^u(t)|^2 du dt < \infty$. Using Assumption \ref{a:stabilizability} again, since the controlled system \eqref{eq:homo_SDE} is also graphon-$L^2$-stabilizable, there exists a matrix $\Theta \in \RR^{m \times d}$ such that the solution to \eqref{eq:homo_SDE_closed_form} satisfies
$\EE [\int_I \int_0^{\infty}  |X^u(t)|^2 dt du] < \infty$. Hence, by \cite[Proposition 3.5]{Huang-Li-Yong-2015}, for $\bm{I}_d \in \bS^d_{++}$ denoting the $d$-dimensional identity matrix, the following Lyapunov equation 
$$P(A + B\Theta) + (A+ B\Theta)^\top P + (C+ D\Theta)^\top P (C+D\Theta) + \bm{I}_d = 0$$
admits a unique solution $P \in \bS^d_{++}$. For each $u \in I$, let the control $\alpha^u$ take the form
$$\alpha^u(t) = \Theta(X^u(t) - \bar{X}^u(t)) + \bar{\Theta} \bar{X}^u(t) + \int_I \wt{\Theta}(u, v) \bar{X}^v(t) dv.$$
Next, we show that the corresponding closed-loop system is graphon-$L^2$-stabilizable. Set
$\check{X}^u(t) = X^u(t) - \bar{X}^u(t)$ for $t \geq 0$. Then
\begin{equation*}
\begin{aligned}
d\check{X}^u(t) &= (A + B\Theta) \check{X}^u(t) dt + \Big\{(C+D\Theta) \check{X}^u(t) + (C+ \bar{C} + D \bar{\Theta}) \bar{X}^u(t) \\
& \hspace{0.3in} + \int_I \big(\wt{C} G(u, v) + D \wt{\Theta}(u, v)\big) \bar{X}^v(t) dv \Big\} d W^u(t)
\end{aligned}
\end{equation*}
with $\check{X}^u(0) = \xi^u - \bar{\xi}^u$. Applying It\^o's formula to $\lan P \check{X}^u(t), \check{X}^u(t) \ran$, we derive that
\begin{equation*}
\begin{aligned}
& \frac{d}{dt} \EE \big[\lan P \check{X}^u(t), \check{X}^u(t) \ran \big] \\
= \ & \EE \big[\lan \{P(A + B\Theta) + (A+ B\Theta)^\top P + (C+ D\Theta)^\top P (C+D\Theta)\} \check{X}^u(t), \check{X}^u(t) \ran \big] \\
& \hspace{0.2in} + \Big\lan P \Big\{(C+ \bar{C} + D \bar{\Theta}) \bar{X}^u(t) + \int_I \big(\wt{C} G(u, v) + D \wt{\Theta}(u,v)\big) \bar{X}^v(t) dv \Big\}, \\
& \hspace{0.7in} (C+ \bar{C} + D \bar{\Theta}) \bar{X}^u(t) + \int_I \big(\wt{C} G(u, v) + D \wt{\Theta}(u, v)\big) \bar{X}^v(t) dv \Big\ran.
\end{aligned}
\end{equation*}
Since $P \in \bS^d_{++}$ is the solution to the preceding Lyapunov equation, $G \in L^2(I \times I; \RR)$, and $\wt{\Theta} \in L^2(I \times I; \RR^{m \times d})$, integrating from $0$ to $t$ gives
\begin{equation*}
\begin{aligned}
& \EE \Big[ \int_0^t \int_I |\check{X}^u(s)|^2 du ds \Big] \\
= \ & \EE \Big[ \int_I \lan P \check{X}^u(0), \check{X}^u(0) \ran du \Big] - \EE \Big[ \int_I \lan P \check{X}^u(t), \check{X}^u(t) \ran du \Big] \\
& \hspace{0.2in} + \int_0^t \int_I \Big\lan P \Big\{(C+ \bar{C} + D \bar{\Theta}) \bar{X}^u(s) + \int_I (\wt{C} G(u, v) + D \wt{\Theta}(u, v)) \bar{X}^v(s) dv \Big\}, \\
& \hspace{1in} (C+ \bar{C} + D \bar{\Theta}) \bar{X}^u(s) + \int_I (\wt{C} G(u, v) + D \wt{\Theta}(u, v)) \bar{X}^v(s) dv \Big\ran du ds \\
\leq \ & K \EE \Big[ \int_I |\check{X}^u(0)|^2 du \Big] + K \int_0^t \int_I |\bar{X}^u(s)|^2 du ds
\end{aligned}
\end{equation*}
for some constant $K > 0$. Since
$\EE [ \int_I |\xi^u|^2 du] < \infty$ and $\int_0^{\infty} \int_I |\bar{X}^u(s)|^2 du ds < \infty$, letting $t \to \infty$, we have
\begin{equation*}
\EE \Big[ \int_0^{\infty} \int_I |\check{X}^u(t)|^2 du dt \Big] \leq K \EE \Big[ \int_I |\xi^u|^2 du \Big] + K \int_0^{\infty} \int_I |\bar{X}^u(s)|^2 du ds < \infty
\end{equation*}
for a possibly different constant $K > 0$. Combining this estimate with $\int_0^{\infty} \int_I |\bar{X}^u(s)|^2 du ds < \infty$, we conclude that the homogeneous system \eqref{eq:homo_state_process} is graphon-$L^2$-stabilizable.
\end{proof}

\subsection{System of algebraic Riccati equations}
\label{s:system_algebraic_equations}

In this section, we derive the system of algebraic equations associated with the ergodic control problem. As in the finite-horizon setting, we consider the following quadratic ansatz:
\begin{equation}
\label{eq:value_function_ansatz_ergodic}
\mathbb{V}(\xi) = \int_I \EE[\langle \xi^u, P \xi^u \rangle] du + \int_I \langle \bar{\xi}^u, \Pi \bar{\xi}^u \rangle du + \int_I \int_I \langle \bar{\xi}^u, \Lambda(u, v) \bar{\xi}^v \rangle dv du + 2 \int_I \langle \bar{\xi}^u, p(u)\rangle du,
\end{equation}
where $P, \Pi \in \bS^d$, $\Lambda \in L^2(I \times I; \RR^{d \times d})$, and $p \in L^2(I; \RR^d)$ are to be determined. Then, we introduce
\begin{equation*}
\begin{aligned}
\mathcal V(t) &:= \int_I \EE[\langle X^u(t), P X^u(t) \rangle] du + \int_I \langle \bar{X}^u(t), \Pi \bar{X}^u(t) \rangle du \\
& \hspace{0.3in} + \int_I \int_I \langle \bar{X}^u(t), \Lambda(u, v) \bar{X}^v(t) \rangle dv du + 2 \int_I \langle \bar{X}^u(t), p(u)\rangle du,
\end{aligned}
\end{equation*}
and let
\begin{equation*}
\begin{aligned}
\mathcal Y(t) &:= \mathcal V(t) + \int_0^t \int_I \mathbb E\big[\langle X^u(s), Q X^u(s) \rangle + 2\langle \alpha^u(s), S X^u(s) \rangle + \langle \alpha^u(s), R \alpha^u(s) \rangle \\
&\hspace{0.9in} + \langle \bar{X}^u(s), \bar{Q} \bar{X}^u(s) \rangle + \langle \bar{X}^u(s), \wt{Q} [G\bar{X}(s)]^u\rangle \\
&\hspace{0.9in} + \langle [G\bar{X}(s)]^u, \check{Q} [G\bar{X}(s)]^u\rangle \big] du ds.
\end{aligned}
\end{equation*}
To proceed, for given $P, \Pi \in \bS^d$, we define the mapping $F^{\infty}: L^2(I \times I; \RR^{d \times d}) \to L^2(I \times I; \RR^{d \times d})$ by
\begin{equation*}
\begin{aligned}
F^{\infty}(\Lambda)(u, v) &:= \Psi_1(P, \Pi, \Lambda)(u, v) - \wh{\cM}(P, \Pi)^\top \cR(P)^{-1} \Upsilon(P, \Lambda)(u, v) \\
& \hspace{0.5in} - \Upsilon(P, \Lambda)(v, u)^\top \cR(P)^{-1} \wh{\cM}(P, \Pi) \\
& \hspace{0.5in} - \int_I \Upsilon(P, \Lambda)(w, u)^\top \cR(P)^{-1} \Upsilon(P, \Lambda)(w, v) dw.
\end{aligned}
\end{equation*}
Moreover, given $P, \Pi \in \bS^d$ and $\Lambda \in L^2(I \times I; \RR^{d \times d})$, we define the mapping $\wt{F}^{\infty}: L^2(I; \RR^d) \to L^2(I; \RR^d)$ by
\begin{equation*}
\begin{aligned}
\wt{F}^{\infty}(p)(u) &:= \Psi_2(P, \Pi, \Lambda, p)(u) - \wh{\cM}(P, \Pi)^\top \cR(P)^{-1} \Gamma(P, p)(u) \\
& \hspace{0.5in} - \int_I \Upsilon(P, \Lambda)(v, u)^\top \cR(P)^{-1} \Gamma(P, p)(v) dv.
\end{aligned}
\end{equation*}
Using an argument similar to that in Section \ref{sec:System_Riccati_equations_finite}, we obtain
\begin{equation*}
\begin{aligned}
d \mathcal Y(t) & = \int_I \EE\big[\lan (\cQ(P) - \cM(P)^\top \cR(P)^{-1} \cM(P)) X^u(t), X^u(t) \ran \big] du dt \\
&\hspace{0.3in} + \int_I \big\lan \big\{\wh{\cQ}(P, \Pi) - \wh{\cM}(P, \Pi)^\top \cR(P)^{-1} \wh{\cM}(P, \Pi) - \cQ(P) \\
&\hspace{0.8in} + \cM(P)^\top \cR(P)^{-1} \cM(P) \big\} \bar{X}^u(t), \bar{X}^u(t) \big\ran du dt \\
&\hspace{0.3in} + \int_I\int_I \lan F^{\infty}(\Lambda)(u, v) \bar{X}^v(t), \bar{X}^u(t) \ran dv du dt  + 2 \int_I \lan \wt{F}^{\infty}(p)(u), \bar{X}^u(t) \ran du dt \\
&\hspace{0.3in} + \int_I \big(- \langle \Gamma(P, p)(u), \cR(P)^{-1} \Gamma(P, p)(u) \rangle + \langle P \sigma, \sigma \rangle + 2 \langle p(u), b \rangle \big) du dt \\
&\hspace{0.3in} + \int_I \EE \big[\lan \cR(P) \{\alpha^u(t) + \cR(P)^{-1} \mathcal Z^u(t)\}, \alpha^u(t) + \cR(P)^{-1} \mathcal Z^u(t) \ran \big] du dt,
\end{aligned}
\end{equation*}
where
\begin{equation*}
\mathcal Z^u(t) := \cM(P) X^u(t) + \bar{\cM}(P, \Pi) \bar{X}^u(t) + D^\top P \sigma + B^\top p(u) + \int_I \Upsilon(P, \Lambda)(u, v) \bar{X}^v(t) dv.
\end{equation*}
Accordingly, we consider the following system of algebraic equations associated with the ergodic GMFC problem. We start with the system of standard algebraic Riccati equations satisfied by $P$ and $\bar{\Pi} = P + \Pi$:
\begin{equation}
\label{eq:Riccati_P_Pi_ergodic}
\begin{cases}
\cQ(P) - \cM(P)^\top \cR(P)^{-1} \cM(P) = 0, \\
\wh{\cQ}(P, \Pi) - \wh{\cM}(P, \Pi)^\top \cR(P)^{-1} \wh{\cM}(P, \Pi) = 0.
\end{cases}
\end{equation}
Next, given a solution $(P, \bar{\Pi}) \in \bS^d_{++} \times \bS^d_{++}$ to \eqref{eq:Riccati_P_Pi_ergodic}, we introduce the following abstract algebraic Riccati equation on the Hilbert space $L^2(I \times I; \RR^{d \times d})$:
\begin{equation}
\label{eq:Riccati_Lambda_ergodic}
F^{\infty}(\Lambda)(u, v) = 0
\end{equation}
for a.e. $u, v \in I$. Then, given a solution $(P, \bar{\Pi}) \in \bS^d_{++} \times \bS^d_{++}$ to \eqref{eq:Riccati_P_Pi_ergodic} and a solution $\Lambda \in L^2(I \times I; \RR^{d \times d})$ to \eqref{eq:Riccati_Lambda_ergodic}, we introduce the linear equation on $L^2(I; \RR^{d})$: 
\begin{equation}
\label{eq:Riccati_p_ergodic}
\wt{F}^{\infty}(p)(u) = 0
\end{equation}
for a.e. $u \in I$.

\subsection{Solvability of the ergodic control problem}
\label{s:main_result_ergodic_control}

Under Assumptions \ref{a:coefficients_solvability_finite} and \ref{a:stabilizability}, it follows from \cite[Lemma 2.2]{Sun-Yong-2024} that the system of standard algebraic Riccati equations \eqref{eq:Riccati_P_Pi_ergodic} admits a unique solution pair $(P, \bar{\Pi}) \in \mathbb S^{d}_{++} \times \mathbb S^d_{++}$. In addition, for all $t \in [0, T]$, one has $P_T(t) \leq P$ and $\bar{\Pi}_T(t) \leq \bar{\Pi}$,  where $P_T$ and $\bar{\Pi}_T$ are the solutions to the differential Riccati equations \eqref{eq:ODE_P_T} and \eqref{eq:ODE_Pi_T}, respectively. 
For simplicity of presentation, we first assume that the equations \eqref{eq:Riccati_Lambda_ergodic} and \eqref{eq:Riccati_p_ergodic} admit unique solutions. A sufficient condition ensuring the unique solvability of the system \eqref{eq:Riccati_Lambda_ergodic}-\eqref{eq:Riccati_p_ergodic} will be provided in Section \ref{s:solvability_algebraic}.

\begin{assumption}
\label{a:solvability_Lambda_p}
The system of algebraic equations \eqref{eq:Riccati_Lambda_ergodic}-\eqref{eq:Riccati_p_ergodic} admits a unique solution with $\Lambda \in L^2(I \times I; \RR^{d \times d})$ and $p \in L^2(I; \RR^{d})$ such that, for all $\phi \in L^2(I; \mathbb{R}^d)$,
$$\lan \bm{\Lambda}(\phi), \phi \ran_{L^2} = \int_I \int_I \lan \phi^u, \Lambda(u, v) \phi^v \ran dv du \geq 0.$$
Moreover, for each fixed $t \geq 0$, the solution $\Lambda_T(t)$ of \eqref{eq:ODE_Lambda_T} converges to the solution $\Lambda$ of \eqref{eq:Riccati_Lambda_ergodic} in $L^2(I \times I; \mathbb{R}^{d \times d})$ as $T \to \infty$.
\end{assumption}

Note that, if $(P, \bar{\Pi})$, $\Lambda$, and $p$ are solutions to the equations \eqref{eq:Riccati_P_Pi_ergodic}, \eqref{eq:Riccati_Lambda_ergodic}, and \eqref{eq:Riccati_p_ergodic}, respectively, then $d \mathcal{Y}(t)$ can be simplified as follows:
\begin{equation}
\label{eq:differential_ergodic}
\begin{aligned}
d \mathcal Y(t) & = \int_I \big(- \langle \Gamma(P, p)(u), \cR(P)^{-1} \Gamma(P, p)(u) \rangle + \langle P \sigma, \sigma \rangle + 2 \langle p(u), b \rangle \big) du dt \\
&\hspace{0.3in} + \int_I \EE \big[\lan \cR(P) \{\alpha^u(t) + \cR(P)^{-1} \mathcal Z^u(t)\}, \alpha^u(t) + \cR(P)^{-1} \mathcal Z^u(t) \ran \big] du dt.
\end{aligned}
\end{equation}
The identity \eqref{eq:differential_ergodic} yields the fundamental relation for the ergodic control problem.

\begin{lemma}
\label{l:fundamental_relation_ergodic}
Suppose that Assumptions \ref{a:coefficients_solvability_finite}, \ref{a:stabilizability}, and \ref{a:solvability_Lambda_p} hold. Let $(P, \bar{\Pi})$, $\Lambda$, and $p$ be the unique solutions to the system of equations \eqref{eq:Riccati_P_Pi_ergodic}, \eqref{eq:Riccati_Lambda_ergodic}, and \eqref{eq:Riccati_p_ergodic}. Then, for all $\xi \in \kL^2_0$ and $\alpha \in \mathcal A[0, \infty)$,
\begin{equation}
\label{eq:fundamental_relation_ergodic}
\begin{aligned}
J_{\infty}(\xi, \alpha) &= \limsup_{T \to \infty} \frac{1}{T} \int_0^T \int_I \EE \big[\lan \cR(P) \{\alpha^u(t) + \cR(P)^{-1} \mathcal Z^u(t)\}, \alpha^u(t) + \cR(P)^{-1} \mathcal Z^u(t) \ran \big] du dt \\
&\hspace{0.3in} + \int_I \big(- \langle \Gamma(P, p)(u), \cR(P)^{-1} \Gamma(P, p)(u) \rangle + \langle P \sigma, \sigma \rangle + 2 \langle p(u), b \rangle \big) du.
\end{aligned}
\end{equation}
\end{lemma}
\begin{proof}
The desired result follows directly from the differential identity \eqref{eq:differential_ergodic} and the argument used in Proposition \ref{p:fundamental_relation}. More precisely, integrating \eqref{eq:differential_ergodic} over $[0, T]$, dividing both sides by $T$, and then taking the limit superior as $T \to \infty$, we obtain
\begin{equation*}
\begin{aligned}
& \limsup_{T \to \infty} \frac{1}{T} \int_0^T \int_I \EE \big[\lan \cR(P) \{\alpha^u(t) + \cR(P)^{-1} \mathcal Z^u(t)\}, \alpha^u(t) + \cR(P)^{-1} \mathcal Z^u(t) \ran \big] du dt \\
&\hspace{0.3in} + \int_I \big(- \langle \Gamma(P, p)(u), \cR(P)^{-1} \Gamma(P, p)(u) \rangle + \langle P \sigma, \sigma \rangle + 2 \langle p(u), b \rangle \big) du \\
= \ & \limsup_{T \to \infty} \frac{1}{T}(\cY(T) - \cY(0)) = J_{\infty}(\xi, \alpha).
\end{aligned}
\end{equation*}
Here, we have used the definition of $J_{\infty}(\xi, \alpha)$ in \eqref{eq:cost_functional_infinite_horizon}, together with the fact that
$$\limsup_{T \to \infty}\frac{1}{T}(\mathcal{V}(T) - \mathcal{V}(0)) = 0.$$
This follows from the transversality condition in the definition of $\cA[0,\infty)$, and the quadratic growth of $\mathcal V(T)$ with respect to the state variable.
\end{proof}

We now establish the solvability of the ergodic GMFC problem with the state dynamics \eqref{eq:state_process} and the cost functional \eqref{eq:cost_functional_infinite_horizon}. For simplicity of notation, we define
\begin{equation}
\label{eq:Theta_notation}
\begin{aligned}
& \Theta^* := - \cR(P)^{-1} \cM(P), \\
& \bar{\Theta}^* := - \cR(P)^{-1} \wh{\cM}(P, \Pi), \\
& \theta^*(u) := - \cR(P)^{-1} \Gamma(P, p)(u), \\
& \wt{\Theta}^*(u, v) := - \cR(P)^{-1}
\Upsilon(P, \Lambda)(u, v)
\end{aligned}
\end{equation}
for a.e. $u, v \in I$. Moreover, we introduce the following feedback control 
\begin{equation}
\label{eq:optimal_control_ergodic}
\alpha_{\infty}^u(t) = \Theta^* \big(X_{\infty}^u(t) - \bar{X}_{\infty}^u(t) \big) + \bar{\Theta}^* \bar{X}_{\infty}^u(t) + \theta^*(u) + \int_I \wt{\Theta}^*(u, v) \bar{X}_{\infty}^v(t) dv
\end{equation}
for all $t \in [0, \infty)$, where $X_{\infty} = (X_{\infty}^u)_{u \in I}$ is the solution to the corresponding closed-loop equation 
\begin{equation}
\label{eq:optimal_state_ergodic}
\begin{cases}
dX_{\infty}^u(t) = \big\{(A+B\Theta^*) X_{\infty}^u(t) + [\bar{A} + B (\bar{\Theta}^* - \Theta^*)] \bar{X}_{\infty}^u(t) + B\theta^*(u) + b \\
\hspace{1in} + \int_I (\wt{A}G(u, v) + B \wt{\Theta}^*(u, v)) \bar{X}_{\infty}^v(t) dv\big\}dt 
\\
\hspace{0.7in} +\big\{(C+D\Theta^*) X_{\infty}^u(t) + [\bar{C} + D (\bar{\Theta}^* - \Theta^*)] \bar{X}_{\infty}^u(t) + D \theta^*(u) + \sigma \\
\hspace{1in} + \int_I (\wt{C}G(u, v) + D \wt{\Theta}^*(u, v)) \bar{X}_{\infty}^v(t) dv \big\}dW^u(t), \quad t \geq 0, \\
X_{\infty}^u(0) = \xi^u, \quad u \in I,
\end{cases}
\end{equation}
The pair $(X_{\infty}, \alpha_{\infty})$ will be shown to be optimal for the ergodic GMFC problem \eqref{eq:cost_functional_infinite_horizon}.

We begin by establishing the following result, which identifies stabilizers for the controlled ODE \eqref{eq:homo_ode} and the controlled SDE \eqref{eq:homo_SDE}.

\begin{lemma}
\label{l:stabilizer}
Assume that Assumptions \ref{a:coefficients_solvability_finite}, \ref{a:stabilizability}, and \ref{a:solvability_Lambda_p} hold. Let $\Theta^*, \bar{\Theta}^*$, and $\wt{\Theta}^*$ be defined in \eqref{eq:Theta_notation}. Then, $\Theta^*$ is a stabilizer of the controlled SDE \eqref{eq:homo_SDE}, and $(\bar{\Theta}^*, \wt{\Theta}^*)$ is a stabilizer of the controlled ODE \eqref{eq:homo_ode}.
\end{lemma}
\begin{proof}
The first assertion follows directly from \cite[Theorem 5.1]{Huang-Li-Yong-2015}. We next prove that $(\bar{\Theta}^*, \wt{\Theta}^*)$ is a stabilizer for the controlled ODE \eqref{eq:homo_ode}. Set $b = \sigma = 0$, and let $X = (X^u)_{u \in I}$ be the solution to the homogeneous system \eqref{eq:homo_state_process}. Then, equation \eqref{eq:Riccati_p_ergodic} becomes a homogeneous linear equation for $p$. By the uniqueness of its solution, we have $p(u) = 0$ for a.e. $u \in I$. Consequently, $\Gamma(P, p)(u) = 0$ and $\theta^*(u) = 0$ for a.e. $u \in I$. Using the calculation in Section \ref{s:system_algebraic_equations} and the differential identity \eqref{eq:differential_ergodic}, we obtain
$$d \mathcal Y(t) = \int_I \EE \big[\lan \cR(P) \{\alpha^u(t) + \cR(P)^{-1} \mathcal Z^u(t)\}, \alpha^u(t) + \cR(P)^{-1} \mathcal Z^u(t) \ran \big] du dt.$$
Let $\alpha^u(t) = - \cR(P)^{-1} \mathcal{Z}^u(t) = \Theta^* \big(X^u(t) - \bar{X}^u(t) \big) + \bar{\Theta}^* \bar{X}^u(t) + \int_I \wt{\Theta}^*(u, v) \bar{X}^v(t) dv$. Integrating the above differential identity over $[0, T]$, we derive that
\begin{equation*}
\begin{aligned}
\mathcal{V}(0) - \mathcal{V}(T) &= \int_0^T \int_I \mathbb E\big[\langle X^u(s), Q X^u(s) \rangle + 2\langle \alpha^u(s), S X^u(s) \rangle + \langle \alpha^u(s), R \alpha^u(s) \rangle \\
&\hspace{0.9in} + \langle \bar{X}^u(s), \bar{Q} \bar{X}^u(s) \rangle + \langle \bar{X}^u(s), \wt{Q} [G\bar{X}(s)]^u\rangle \\
&\hspace{0.9in} + \langle [G\bar{X}(s)]^u, \check{Q} [G\bar{X}(s)]^u\rangle \big] du ds.
\end{aligned}
\end{equation*}
In this case, $\bar{\alpha}^u(t) = \bar{\Theta}^* \bar{X}^u(t) + \int_I \wt{\Theta}^*(u, v) \bar{X}^v(t) dv$, and $\bar{X} = (\bar{X}^u)_{u \in I}$ is the solution to
\begin{equation*}
\begin{cases}
\frac{d}{dt} \bar{X}^u(t) = (A + \bar{A} + B \bar{\Theta}^*) \bar{X}^u(t) + \int_I \big(\wt{A} G(u, v) + B \wt{\Theta}^*(u, v) \big) \bar{X}^v(t) dv, \\
\bar{X}^u(0) = \bar{\xi}^u.
\end{cases}
\end{equation*}
It follows from Assumption \ref{a:coefficients_solvability_finite} that there exists a constant $c_0 > 0$ such that
$$c_0 \int_0^T \int_{I} \mathbb{E}[|X^u(s)|^2] du ds \leq \mathcal{V}(0) - \mathcal{V}(T), \quad \forall T > 0.$$
Since $P$ and $\bar{\Pi} = P + \Pi$ are both positive definite, and $\Lambda(u, v)$ is a positive kernel on $L^2(I; \mathbb{R}^d)$, we deduce that $\mathcal{V}(T) \geq 0$, and thus
$$\int_0^T \int_{I} \mathbb{E}[|X^u(s)|^2] du ds \leq \frac{1}{c_0} \mathcal{V}(0) \leq K \int_{I} \lan \bar{\xi}^u, \bar{\xi}^u \ran du, \quad \forall  T > 0$$
for some $K > 0$. Letting $T \to \infty$ and applying the monotone convergence theorem, we conclude that
$$\int_0^{\infty} \int_{I} \mathbb{E}[|X^u(s)|^2] du ds \leq K \int_{I} \lan \bar{\xi}^u, \bar{\xi}^u \ran du < \infty,$$
which yields the desired result that $(\bar{\Theta}^*, \wt{\Theta}^*)$ is a stabilizer for the controlled ODE \eqref{eq:homo_ode}.
\end{proof}

Next, we derive several estimates for $X_{\infty} = (X^u_{\infty})_{u \in I}$ defined in \eqref{eq:optimal_state_ergodic}. These estimates are essential for verifying the optimality of the pair $(X^u_{\infty}, \alpha^u_{\infty})_{u \in I}$ in the ergodic control problem and for establishing the turnpike property in Section \ref{s:turnpike_properties}. Throughout the proof, whenever no confusion can arise, we write $\|\cdot\|_{L^2}$ for both $\|\cdot\|_{L^2(I;\mathbb{R}^d)}$ and $\|\cdot\|_{L^2(I \times I; \mathbb{R}^{d \times d})}$. Moreover, we define the linear operator $\mathcal{L}$ associated with the stabilizer $(\Theta^*, \wt{\Theta}^*)$ for \eqref{eq:homo_ode} as follows: for $\phi \in L^2(I; \mathbb{R}^d)$ and $u \in I$,
\begin{equation}
\label{eq:linear_operator}
\mathcal{L}(\phi)(u) = (A + \bar{A} + B \bar{\Theta}^*) \phi^u + \int_I \big(\wt{A} G(u, v) + B \wt{\Theta}^*(u, v)\big) \phi^v dv.
\end{equation}
Thus, Lemma \ref{l:stabilizer} and the estimate \eqref{eq:operator_s_uniform_stable} imply that there exist constants $K, \lambda > 0$ such that
\begin{equation}
\label{eq:operator_uniform_stable}
\|e^{t\mathcal{L}}\|_{op} \leq K e^{-\lambda t}, \quad \forall t \geq 0.
\end{equation}

\begin{lemma}
\label{l:Moment_estimate}
Suppose that Assumptions \ref{a:coefficients_solvability_finite}, \ref{a:stabilizability}, and \ref{a:solvability_Lambda_p} hold. Then, there exists a constant $K > 0$, independent of $t$, such that the solution $X_{\infty}$ to the closed-loop equation \eqref{eq:optimal_state_ergodic} satisfies
\begin{equation}
\label{eq:moment_boundness}
\int_I \big|\bar{X}_{\infty}^u(t) \big|^2 du \leq K \quad \text{and} \quad \mathbb{E}\Big[\int_I |X^u_{\infty}(t)|^2 du \Big] \leq K, \quad \forall t \geq 0.
\end{equation}
\end{lemma}
\begin{proof}
Taking expectations in the SDE \eqref{eq:optimal_state_ergodic}, we obtain
\begin{equation*}
d\bar{X}_{\infty}^u(t) = \Big\{(A + \bar{A} + B \bar{\Theta}^*) \bar{X}_{\infty}^u(t) + B \theta^*(u) + b + \int_I \big(\wt{A} G(u, v) + B \wt{\Theta}^*(u, v)\big) \bar{X}_{\infty}^v(t) dv \Big\} dt
\end{equation*}
with $\bar{X}_{\infty}^u(0) = \bar{\xi}^u$ for $u \in I$, which is equivalent to
\begin{equation*}
\frac{d}{dt} \bar{X}_{\infty}^u(t) = \mathcal{L}(\bar{X}_{\infty}(t))(u) + B \theta^*(u) + b
\end{equation*}
by the definition of the linear operator $\mathcal{L}$ in \eqref{eq:linear_operator}. Next, we verify that $\bar{X}_{\infty}(t) \in L^2(I; \mathbb{R}^d)$ for all $t \geq 0$ and obtain a uniform estimate for $\bar{X}_{\infty}$. By variation of constants, we have
\begin{equation*}
\bar{X}_{\infty}^u(t) = \big(e^{t\mathcal{L}} \bar{X}_{\infty}(0)\big)(u) + \int_0^t e^{(t-s)\mathcal{L}} (B \theta^*(u) + b) ds.
\end{equation*}
By the estimate \eqref{eq:operator_uniform_stable}, since $\theta^* \in L^2(I; \mathbb{R}^m)$ and $\|\bar{\xi} \|_{L^2} \leq K$, we obtain the following uniform bound for $\|\bar{X}_{\infty}(t)\|_{L^2}$:
\begin{equation*}
\begin{aligned}
\|\bar{X}_{\infty}(t)\|_{L^2} & \leq \|e^{t\mathcal{L}}\|_{op} \|\bar{\xi}\|_{L^2} + \int_0^t \|e^{(t-s)\mathcal{L}}\|_{op} \|B \theta^* + b\|_{L^2} ds \\
&\leq K e^{-\lambda t} + K \int_0^t e^{-\lambda(t-s)} ds \leq K, \quad \forall t \geq 0.
\end{aligned}
\end{equation*}

Next, we set $\check{X}^u_\infty(t) = X_{\infty}^u(t) - \bar{X}_{\infty}^u(t)$ for all $t \geq 0$ and $u \in I$. Then, $\check{X}^u_\infty(t)$ satisfies the following SDE:
\begin{equation*}
\begin{aligned}
d \check{X}^u_\infty(t) &= (A + B {\Theta}^*) \check{X}^u_\infty(t) dt + \Big\{(C + D {\Theta}^*) \check{X}^u_\infty(t) + (C+\bar{C} + D \bar{\Theta}^*) \bar{X}_{\infty}^u(t) + D \theta^*(u) + \sigma \\
& \hspace{0.3in} + \int_I \big(\wt{C} G(u, v) + D \wt{\Theta}^*(u, v) \big) \bar{X}_{\infty}^v(t) dv \Big\} dW^u(t) \\
&:= (A + B {\Theta}^*) \check{X}^u_\infty(t) dt + \big\{ (C + D {\Theta}^*) \check{X}^u_\infty(t) + \Sigma^u(t) \big\} dW^u(t)
\end{aligned}
\end{equation*}
with $\check{X}^u_\infty(0) = \xi^u - \bar{\xi}^u$, where
\begin{equation*}
\Sigma^u(t) := (C+\bar{C} + D \bar{\Theta}^*) \bar{X}_{\infty}^u(t) + D \theta^*(u) + \sigma + \int_I \big(\wt{C} G(u, v) + D \wt{\Theta}^*(u, v) \big) \bar{X}_{\infty}^v(t) dv.
\end{equation*}
Since $\theta^* \in L^2(I; \mathbb{R}^m)$, $\wt{\Theta}^* \in L^{2}(I \times I; \mathbb{R}^{m \times d})$, $G \in L^2(I \times I; \mathbb{R})$, and $\|\bar{X}_{\infty}(t)\|_{L^2} \le K$, H\"older's inequality gives
\begin{equation*}
\begin{aligned}
\|\Sigma(t)\|_{L^2}^2 & \leq K \Big(\|C+\bar{C} + D \bar{\Theta}^*\|^2 \|\bar{X}_{\infty}(t)\|^2_{L^2} + \|D \theta^* + \sigma\|^2_{L^2} \\
& \hspace{0.5in} + \big(\|G\|^2_{L^2} +   \|\wt{\Theta}^*\|^2_{L^2} \big) \|\bar{X}_{\infty}(t)\|^2_{L^2} \Big) \leq K,
\end{aligned}
\end{equation*}
for all $t \geq 0$. Since ${\Theta}^*$ is a stabilizer of the homogeneous SDE \eqref{eq:homo_SDE}, the system $[A + B {\Theta}^*, C + D {\Theta}^*]$ is $L^2$-exponentially stable. Thus, the following Lyapunov equation admits a unique solution $\bar{P}\in\bS_{++}^d$:
\begin{equation}
\label{eq:Lyapunov_equation}
\bar{P}(A + B {\Theta}^*) + (A + B {\Theta}^*)^\top \bar{P} + (C + D {\Theta}^*)^\top \bar{P} (C + D{\Theta}^*) + \bm{I}_{d} = 0.
\end{equation}
Let $\beta^{*} > 0$ denote the largest eigenvalue of $\bar{P}$. Applying It\^o's formula, we obtain
\begin{equation*}
\begin{aligned}
& \frac{d}{dt} \mathbb{E}\big[ \langle \bar{P} \check{X}^u_\infty(t),  \check{X}^u_\infty(t) \rangle \big] \\ 
= \ & \mathbb{E} \big[ \langle \{\bar{P}(A + B {\Theta}^*) + (A + B {\Theta}^*)^\top \bar{P} + (C + D {\Theta}^*)^\top \bar{P}(C + D{\Theta}^*) \} \check{X}^u_\infty(t), \check{X}^u_\infty(t) \rangle \\
& \hspace{0.3in} + 2 \langle (C + D{\Theta}^*)^\top \bar{P} \Sigma^u(t), \check{X}^u_\infty(t) \rangle + \langle \Sigma^u(t), \bar{P} \Sigma^u(t) \rangle \big] \\
= \ & -\mathbb{E} \big[|\check{X}^u_\infty(t)|^2 \big] + \langle \Sigma^u(t), \bar{P} \Sigma^u(t) \rangle.
\end{aligned}
\end{equation*}
By Fubini's theorem, we derive that
\begin{equation*}
\begin{aligned}
\frac{d}{dt} \mathbb{E} \Big[ \int_I \langle \check{X}^u_\infty(t), \bar{P} \check{X}^u_\infty(t) \rangle du \Big] & \leq - \frac{1}{\beta^{*}} \mathbb{E} \Big[ \int_I \langle \bar{P}\check{X}^u_\infty(t), \check{X}^u_\infty(t) \rangle du \Big] + K \|\Sigma(t)\|_{L^2}^2 \\
&\leq - \frac{1}{\beta^{*}} \mathbb{E} \Big[\int_I \langle \bar{P}\check{X}^u_\infty(t), \check{X}^u_\infty(t) \rangle du \Big] + K
\end{aligned}
\end{equation*}
for all $t \geq 0$. Then, Gr\"onwall's inequality implies that
\begin{equation*}
\begin{aligned}
\mathbb{E} \Big[ \int_I \langle \bar{P} \check{X}^u_\infty(t), \check{X}^u_\infty(t) \rangle du \Big] & \leq K e^{-\frac{1}{\beta^{*}} t} \mathbb{E} \Big[ \int_I |\check{X}^u_\infty(0)|^2 du \Big] + K \int_0^t e^{-\frac{1}{\beta^{*}} (t-s)} ds \\
& \leq K e^{- \frac{1}{\beta^{*}} t} + K \beta^{*} \leq K
\end{aligned}
\end{equation*}
for all $t \geq 0$. Since $\bar{P} > 0$, we conclude
\begin{equation*}
\mathbb{E} \Big[ \int_I |\check{X}^u_\infty(t)|^2 du \Big] \leq K, \quad \forall t \geq 0.
\end{equation*}
Hence, by the definition of $\check{X}^u_\infty(t)$, we obtain the desired estimate
\begin{equation*}
\mathbb{E} \Big[\int_I  |X_{\infty}^u(t)|^2 du \Big] \leq 2 \mathbb{E} \Big[ \int_I |\check{X}_{\infty}^u(t)|^2 du \Big] + 2 \int_I |\bar{X}_{\infty}^u(t)|^2 du \leq K, \quad \forall t \geq 0.
\end{equation*}
\end{proof}

The following proposition summarizes the main result for the ergodic GMFC problem.

\begin{proposition}
\label{p:solvability_ergodic_control}
Assume that Assumptions \ref{a:coefficients_solvability_finite}, \ref{a:stabilizability}, and \ref{a:solvability_Lambda_p} hold. Let $(P, \bar{\Pi})$, $\Lambda$, and $p$ be the unique solutions to the equations \eqref{eq:Riccati_P_Pi_ergodic}, \eqref{eq:Riccati_Lambda_ergodic}, and \eqref{eq:Riccati_p_ergodic}, respectively, and define $\Pi = \bar{\Pi} - P$. Then:
\begin{itemize}
\item The optimal control and optimal state for the ergodic GMFC problem \eqref{eq:cost_functional_infinite_horizon} are given by \eqref{eq:optimal_control_ergodic} and \eqref{eq:optimal_state_ergodic}, respectively.

\item For all $\xi \in \kL^2_0$, the ergodic value $V_{\infty}(\xi) := \inf_{\alpha \in \mathcal A[0, \infty)} J_{\infty}(\xi, \alpha)$ is given by
\begin{equation}
\label{eq:ergodic_value}
V_{\infty} := V_{\infty}(\xi) = \int_I \big(- \langle \Gamma(P, p)(u), \cR(P)^{-1} \Gamma(P, p)(u) \rangle + \langle P \sigma, \sigma \rangle + 2 \langle p(u), b \rangle \big) du,
\end{equation}
which is constant and independent of the initial state $\xi$.
\end{itemize}
\end{proposition}
\begin{proof}
The proof is analogous to that of Proposition \ref{p:solvability_finite_control} and follows directly from the fundamental relation established in Lemma \ref{l:fundamental_relation_ergodic} and the moment estimate for $X_{\infty}$ obtained in Lemma \ref{l:Moment_estimate}.
\end{proof}


\section{Convergence of coefficient functions}
\label{s:convergence_Riccati_system}

Recall that the value function $V_T$ and the optimal pair $(X_T, \alpha_T)$ for the finite-horizon GMFC problem are characterized by the solutions to the system of differential equations \eqref{eq:ODE_P_T}, \eqref{eq:ODE_Pi_T}, \eqref{eq:ODE_Lambda_T}, \eqref{eq:ODE_p_T}, and \eqref{eq:ODE_kappa_T}. Similarly, the ergodic value $V_{\infty}$ and the optimal pair $(X_{\infty}, \alpha_{\infty})$ of the ergodic GMFC problem are characterized by the solutions to the system of algebraic equations \eqref{eq:Riccati_P_Pi_ergodic}, \eqref{eq:Riccati_Lambda_ergodic} and \eqref{eq:Riccati_p_ergodic}. In this section, we derive several key estimates relating the finite-horizon coefficients $P_T, \bar{\Pi}_T, \Lambda_T, p_T, \Theta_T^*, \bar{\Theta}_T^*, \wt{\Theta}_T^*$ and $\theta_T^*$ to their ergodic counterparts $P, \bar{\Pi}, \Lambda, p, \Theta^*, \bar{\Theta}^*, \wt{\Theta}^*$ and $\theta^*$.
In what follows, we use $K$ and $\lambda$ to denote generic positive constants that may vary from line to line.

\begin{lemma}
\label{l:exponential_estimates_P_Pi}
Suppose that Assumptions \ref{a:coefficients_solvability_finite} and \ref{a:stabilizability} hold. Let $P_T$ and $\bar{\Pi}_T$ be the unique solutions to \eqref{eq:ODE_P_T} and \eqref{eq:ODE_Pi_T}, respectively, and let the pair $(P, \bar{\Pi})$ be the unique solution to \eqref{eq:Riccati_P_Pi_ergodic}. Then, there exist constants $K > 0$ and $\lambda > 0$, independent of $T$, such that
\begin{equation}
\label{eq:exponential_estimate_P_Pi}
\|P - P_T(t)\| + \|\bar{\Pi} - \bar{\Pi}_T(t)\| \leq K e^{-\lambda(T-t)}, \quad \forall t \in [0, T].
\end{equation}
Consequently, there exist constants $K, \lambda > 0$, independent of $T$, such that
\begin{equation}
\label{eq:exponential_estimate_Theta}
\|\Theta^* - \Theta^*_T(t)\| + \|\bar{\Theta}^* - \bar{\Theta}^*_T(t)\| \leq K e^{-\lambda(T-t)}, \quad \forall t \in [0, T].
\end{equation}
\end{lemma}
\begin{proof}
The estimate \eqref{eq:exponential_estimate_P_Pi} follows directly from Lemma 2.3 and Theorem 4.1 in \cite{Sun-Yong-2024}. By Theorem 4.1 in \cite{Bayraktar-Jian-2025}, we conclude the estimate in \eqref{eq:exponential_estimate_Theta}. 
\end{proof}

Next, we establish a similar exponential estimate for the solution to the abstract Riccati equation \eqref{eq:ODE_Lambda_T} and its ergodic counterpart \eqref{eq:Riccati_Lambda_ergodic}. For simplicity of notation, we define
\begin{equation*}
\Delta_P(t) = P - P_T(t), \quad \Delta_{\bar{\Pi}}(t) = \bar{\Pi} - \bar{\Pi}_T(t), \quad \Delta_{\Lambda}(t) = \Lambda - \Lambda_T(t),
\end{equation*}
for all $t \in [0, T]$.  

\begin{proposition}
\label{p:exponential_estimates_Lambda}
Assume that Assumptions \ref{a:coefficients_solvability_finite}, \ref{a:stabilizability}, and \ref{a:solvability_Lambda_p} hold. Let $\Lambda_T$ and $\Lambda$ be the unique solutions to \eqref{eq:ODE_Lambda_T} and \eqref{eq:Riccati_Lambda_ergodic}, respectively. Then, there exist constants $K > 0$ and $\lambda > 0$, independent of $T$, such that
\begin{equation}
\label{eq:exponential_estimate_Lambda}
\|\Lambda - \Lambda_T(t)\|_{L^2(I \times I; \mathbb{R}^{d \times d})} \leq K e^{-\lambda(T-t)}, \quad \forall t \in [0, T].
\end{equation}
Moreover, there exist constants $K, \lambda > 0$, independent of $T$, such that
\begin{equation}
\label{eq:exponential_estimate_tilde_Theta}
\|\wt{\Theta}^* - \wt{\Theta}^*_T(t)\|_{L^2(I \times I; \mathbb{R}^{m \times d})}  \leq K e^{-\lambda(T-t)}, \quad \forall t \in [0, T].
\end{equation}
\end{proposition}
\begin{proof}
From the equations \eqref{eq:ODE_Lambda_T} and \eqref{eq:Riccati_Lambda_ergodic}, $\Delta_{\Lambda}$ satisfies the following differential equation:
\begin{equation*}
\begin{cases}
\dot{\Delta}_{\Lambda}(t)(u,v) + F^{\infty}(\Lambda)(u,v) - F(t, \Lambda_T(t))(u,v) = 0, \\
\Delta_{\Lambda}(T)(u,v) = \Lambda(u,v),
\end{cases} 
\end{equation*}
for all $t \in [0, T]$ and a.e. $u, v \in I$. Equivalently, the differential equation for $\Delta_{\Lambda}$ can be written explicitly as follows:
\begin{equation*}
\begin{aligned}
& \dot{\Delta}_{\Lambda}(t)(u,v) + \Psi_1(P, \Pi, \Lambda)(u,v) - \Psi_1(P_T(t), \Pi_T(t), \Lambda_T(t))(u,v)  \\
&\hspace{0.3in} + \widehat{\mathcal{M}}(P_T(t), \Pi_T(t))^\top \mathcal{R}(P_T(t))^{-1} \Upsilon(P_T(t), \Lambda_T(t))(u,v) \\
&\hspace{0.3in} - \widehat{\mathcal{M}}(P, \Pi)^\top \mathcal{R}(P)^{-1} \Upsilon(P, \Lambda)(u,v) \\
&\hspace{0.3in} + \Upsilon(P_T(t), \Lambda_T(t))(v,u)^\top \mathcal{R}(P_T(t))^{-1} \widehat{\mathcal{M}}(P_T(t), \Pi_T(t)) \\
&\hspace{0.3in} - \Upsilon(P, \Lambda)(v,u)^\top \mathcal{R}(P)^{-1} \widehat{\mathcal{M}}(P, \Pi) \\
&\hspace{0.3in} + \int_I \Upsilon(P_T(t), \Lambda_T(t))(w,u)^\top \mathcal{R}(P_T(t))^{-1} \Upsilon(P_T(t), \Lambda_T(t))(w,v) dw \\
&\hspace{0.3in} - \int_I \Upsilon(P, \Lambda)(w,u)^\top \mathcal{R}(P)^{-1} \Upsilon(P, \Lambda)(w,v) dw = 0. 
\end{aligned}
\end{equation*}
Next, by expanding each term in the differential equation satisfied by $\Delta_{\Lambda}$, we obtain
\begin{equation*}
\begin{aligned}
& \dot{\Delta}_{\Lambda}(t)(u,v) + (A + \bar{A} + B\bar{\Theta}^*)^\top \Delta_{\Lambda}(t)(u,v) + \Delta_{\Lambda}(t)(v,u)^\top (A + \bar{A} + B\bar{\Theta}^*) \\
& \hspace{0.3in} + \int_I \big(\wt{A}G(u,w) + B\wt{\Theta}^*(u,w) \big)^\top \Delta_{\Lambda}(t)(w,v) dw + \int_I \Delta_{\Lambda}(t)(w,u)^\top \big(\wt{A}G(w,v) + B\wt{\Theta}^*(w,v) \big) dw \\
& \hspace{0.3in} + f_1(\Delta_P(t), \Delta_{\bar{\Pi}}(t), \Delta_{\bar{\Theta}^*}(t))(u,v) + f_2(\Delta_P(t))(u,v) + f_3(\Delta_P(t), \Delta_{\Lambda}(t), \Delta_{\bar{\Theta}^*}(t))(u,v) = 0
\end{aligned}
\end{equation*}
with the terminal condition $\Delta_{\Lambda}(T)(u,v) = \Lambda(u,v)$,
where $\Delta_{\bar{\Theta}^*}(t) = \bar{\Theta}^* - \bar{\Theta}^*_T(t)$, 
\begin{equation*}
\begin{aligned}
& f_1(\Delta_P(t), \Delta_{\bar{\Pi}}(t), \Delta_{\bar{\Theta}^*}(t))(u,v) := \Delta_{\bar{\Pi}}(t) \wt{A} G(u,v) + \wt{A}^\top \Delta_{\bar{\Pi}}(t) G(v,u) + \Delta_{\bar{\Theta}^*}(t)^\top \Upsilon(P, \Lambda)(u,v) \\
& \hspace{0.3in} + \Upsilon(P, \Lambda)(v,u)^\top \Delta_{\bar{\Theta}^*}(t) - \Delta_{\bar{\Theta}^*}(t)^\top D^\top \Delta_P(t) \wt{C} G(u, v) - \wt{C}^\top G(v,u) \Delta_P(t) D \Delta_{\bar{\Theta}^*}(t), 
\end{aligned}
\end{equation*}
\begin{equation*}
\begin{aligned}
& f_2(\Delta_P(t))(u,v) := (C + \bar{C} + D\bar{\Theta}^*)^\top \Delta_P(t) \wt{C} G(u,v) + \wt{C}^\top G(v,u) \Delta_P(t) (C + \bar{C} + D\bar{\Theta}^*) \\
& \hspace{0.3in} + \int_I (\wt{C}G(u,w) + D\wt{\Theta}^*(u,w))^\top \Delta_P(t) (\wt{C}G(w,v) + D\wt{\Theta}^*(w,v)) dw \\
& \hspace{0.3in} + \int_I \big\{ \wt{C}^\top G(w,u) \Delta_P(t) D \mathcal{R}(P)^{-1} D^\top \Delta_P(t) \wt{C} G(w,v) \\
&\hspace{0.7in} - \Upsilon(P, \Lambda)(w,u)^\top \mathcal{R}(P)^{-1} D^\top \Delta_P(t) D \mathcal{R}(P_T(t))^{-1} D^\top \Delta_P(t) \wt{C} G(w,v) \\
&\hspace{0.7in} - \wt{C}^\top G(w,u) \Delta_P(t) D \mathcal{R}(P_T(t))^{-1} D^\top \Delta_P(t) D \mathcal{R}(P)^{-1} \Upsilon(P, \Lambda)(w,v) \\
&\hspace{0.7in} + \wt{C}^\top G(w,u) \Delta_P(t) D \mathcal{R}(P_T(t))^{-1} D^\top \Delta_P(t) D \mathcal{R}(P)^{-1} D^\top \Delta_P(t) \wt{C} G(w,v) \big\} dw, 
\end{aligned}
\end{equation*}
and
\begin{equation*}
\begin{aligned}
& f_3(\Delta_P(t), \Delta_{\Lambda}(t), \Delta_{\bar{\Theta}^*}(t))(u,v) := - \Delta_{\bar{\Theta}^*}(t)^\top B^\top \Delta_{\Lambda}(t)(u, v) - \Delta_{\Lambda}(t)(v,u)^\top B \Delta_{\bar{\Theta}^*}(t) \\
& \hspace{0.3in} + \int_I \big\{ \Delta_{\Lambda}(t)(w,u)^\top B \mathcal{R}(P)^{-1} B^\top \Delta_{\Lambda}(t)(w,v) \\
&\hspace{0.7in} + \Delta_{\Lambda}(t)(w,u)^\top B \mathcal{R}(P)^{-1} D^\top \Delta_P(t) \wt{C} G(w,v) \\
&\hspace{0.7in} + \wt{C}^\top G(w,u) \Delta_P(t) D \mathcal{R}(P)^{-1} B^\top \Delta_{\Lambda}(t)(w,v) \\
&\hspace{0.7in} - \Upsilon(P, \Lambda)(w,u)^\top \mathcal{R}(P)^{-1} D^\top \Delta_P(t) D \mathcal{R}(P_T(t))^{-1} B^\top \Delta_{\Lambda}(t)(w,v) \\
&\hspace{0.7in} - \Delta_{\Lambda}(t)(w,v)^\top B \mathcal{R}(P_T(t))^{-1} D^\top \Delta_P(t) D \mathcal{R}(P)^{-1} \Upsilon(P, \Lambda)(w,v) \\
&\hspace{0.7in} + \Delta_{\Lambda}(t)(w,u)^\top B \mathcal{R}(P_T(t))^{-1} D^\top \Delta_P(t) D \mathcal{R}(P)^{-1} B^\top \Delta_{\Lambda}(t)(w,v) \\
&\hspace{0.7in} + \Delta_{\Lambda}(t)(w,u)^\top B \mathcal{R}(P_T(t))^{-1} D^\top \Delta_P(t) D \mathcal{R}(P)^{-1} D^\top \Delta_P(t) \wt{C} G(w,v) \\
&\hspace{0.7in} + \wt{C}^\top  G(w,u) \Delta_P(t) D \mathcal{R}(P_T(t))^{-1} D^\top \Delta_P(t) D \mathcal{R}(P)^{-1} B^\top \Delta_{\Lambda}(t)(w,v) \big\} dw.
\end{aligned}
\end{equation*}

Because $G \in L^2(I \times I; \mathbb{R})$, the estimates \eqref{eq:exponential_estimate_P_Pi} and \eqref{eq:exponential_estimate_Theta} in Lemma \ref{l:exponential_estimates_P_Pi} imply
\begin{equation*}
\|f_1(\Delta_P(t), \Delta_{\bar{\Pi}}(t), \Delta_{\bar{\Theta}^*}(t))\|_{L^2} \leq K_1 e^{-\lambda_1(T-t)}, \quad \forall t \in [0, T]
\end{equation*}
for some constants $K_1, \lambda_1 > 0$ that are independent of $T$. Similarly, since $\|\Delta_P(t)\| \le K$, we have $\|\Delta_P(t)\|^2 \le K\|\Delta_P(t)\|$ for all $t \in [0, T]$. It follows that
\begin{equation*}
\|f_2(\Delta_P(t))\|_{L^2} \leq K_2 e^{-\lambda_2(T-t)}, \quad \forall t \in [0, T]
\end{equation*}
for some constants $K_2, \lambda_2 > 0$, independent of $T$. Next, for $f_3(\Delta_P(t), \Delta_{\Lambda}(t), \Delta_{\bar{\Theta}^*}(t))$, it is clear that
\begin{equation*}
\begin{aligned}
\|f_3(\Delta_P(t), \Delta_{\Lambda}(t), \Delta_{\bar{\Theta}^*}(t))\|_{L^2} &\leq K \big(\|\Delta_P(t)\|^2 + \|\Delta_{\Lambda}(t)\|^2_{L^2} \big) \\
&\leq K_3 e^{-\lambda(T-t)} + K_4 \|\Delta_{\Lambda}(t)\|^2_{L^2}, \quad \forall t \in [0, T]
\end{aligned}
\end{equation*}
for some constants $K_3, K_4 > 0$, both independent of $T$. 

For $\Lambda \in L^2(I \times I; \mathbb{R}^{d \times d})$, we define the following operator 
\begin{equation*}
\begin{aligned}
\wt{\mathcal{L}}(\Lambda)(u,v) &= (A + \bar{A} + B\bar{\Theta}^*)^\top \Lambda(u,v) + \Lambda(v,u)^\top (A + \bar{A} + B\bar{\Theta}^*) \\
& \hspace{0.3in} + \int_I (\wt{A}G(u,w) + B\wt{\Theta}^*(u,w))^\top \Lambda(w,v) dw \\
& \hspace{0.3in} + \int_I \Lambda(w,u)^\top (\wt{A}G(w,v) + B\wt{\Theta}^*(w,v)) dw,
\end{aligned}
\end{equation*}
which can also be defined in terms of the operator $\mathcal{L}$ in \eqref{eq:linear_operator} and its adjoint $\mathcal{L}^*$. From Assumption \ref{a:stabilizability} and the estimate for $\mathcal{L}$ in \eqref{eq:operator_uniform_stable}, we also have
\begin{equation}
\label{eq:operator_uniform_stable_2}
\|e^{t\wt{\mathcal{L}}}\|_{op} \leq \alpha_1 e^{-\beta_1 t}, \quad \forall t \geq 0,
\end{equation}
for some $\alpha_1, \beta_1 > 0$ independent of $T$, where $\{e^{t\wt{\mathcal{L}}}\}_{t \geq 0}$ is the semigroup generated by $\wt{\mathcal{L}}$. Then,
for $0 \leq t \leq k \leq T$,
\begin{equation*}
\begin{aligned}
\Delta_{\Lambda}(t) &= e^{(k-t)\wt{\mathcal{L}}} \Delta_{\Lambda}(k) + \int_t^k e^{(s-t)\wt{\mathcal{L}}} \big\{ f_1(\Delta_P(s), \Delta_{\bar{\Pi}}(s), \Delta_{\bar{\Theta}^*}(s)) \\
& \hspace{0.3in} + f_2(\Delta_P(s)) + f_3(\Delta_P(s), \Delta_{\Lambda}(s), \Delta_{\bar{\Theta}^*}(s)) \big\} ds.
\end{aligned}
\end{equation*}
Hence, we derive the following estimate:
\begin{equation*}
\begin{aligned}
\|\Delta_{\Lambda}(t)\|_{L^2} &\leq \|e^{(k-t)\wt{\mathcal{L}}}\|_{op} \|\Delta_{\Lambda}(k)\|_{L^2} + \int_t^k \|e^{(s-t)\wt{\mathcal{L}}}\|_{op} \big\{ \|f_1(\Delta_P(s), \Delta_{\bar{\Pi}}(s), \Delta_{\bar{\Theta}^*}(s))\|_{L^2} \\
& \hspace{0.3in} + \|f_2(\Delta_P(s))\|_{L^2} + \|f_3(\Delta_P(s), \Delta_{\Lambda}(s), \Delta_{\bar{\Theta}^*}(s))\|_{L^2} \big\} ds.
\end{aligned}
\end{equation*}
From the above estimates, we can choose $\alpha, \beta > 0$ such that $\|e^{t\wt{\mathcal{L}}}\|_{op} \leq \alpha e^{-\beta t}$ and
\begin{equation*}
\begin{aligned}
& \|f_1(\Delta_P(s), \Delta_{\bar{\Pi}}(s), \Delta_{\bar{\Theta}^*}(s))\|_{L^2} + \|f_2(\Delta_P(s))\|_{L^2} + \|f_3(\Delta_P(s), \Delta_{\Lambda}(s), \Delta_{\bar{\Theta}^*}(s))\|_{L^2} \\
\leq \ & \alpha e^{-2\beta(T-s)} + K_4 \|\Delta_{\Lambda}(s)\|^2_{L^2}
\end{aligned}
\end{equation*}
for all $0 \leq s , t \leq T$. Therefore, for $0 \leq t \leq k < T$, we obtain
\begin{equation*}
\begin{aligned}
\|\Delta_{\Lambda}(t)\|_{L^2} &\leq \alpha e^{-\beta(k-t)} \|\Delta_{\Lambda}(k)\|_{L^2} + \int_t^k \alpha e^{-\beta(s-t)} \big(\alpha e^{-2\beta(T-s)} + K_4 \|\Delta_{\Lambda}(s)\|^2_{L^2} \big) ds \\
&\leq \Big(\alpha \|\Delta_{\Lambda}(k)\|_{L^2} + \frac{\alpha^2}{\beta} e^{-2\beta(T-k)}\Big)e^{-\beta(k-t)} + \alpha K_4 \int_t^k e^{-\beta(s-t)} \|\Delta_{\Lambda}(s)\|^2_{L^2} ds.
\end{aligned}
\end{equation*}

By Assumption \ref{a:solvability_Lambda_p}, for each fixed $t \geq 0$, $\Lambda_T(t)$ converges to $\Lambda$ in $L^2(I \times I; \mathbb{R}^{d \times d})$ as $T \to \infty$. Specifically, 
$\lim_{T \to \infty} \|\Delta_{\Lambda}(0)\|_{L^2} = \lim_{T\to\infty}\|\Lambda - \Lambda_T(0)\|_{L^2} = 0$. Thus, we can choose $N > 0$ such that
$$\max \Big\{\alpha \|\Lambda - \Lambda
_T(0)\|_{L^2}, \frac{\alpha^2}{\beta}e^{-2\beta N} \Big\} \leq \frac{\beta}{4\alpha K_4} := \frac{\rho}{2}, \quad \forall T \geq N.$$
We claim that for any $T \geq 2N$, the inequality
\begin{equation*}
\|\Delta_{\Lambda}(t)\|_{L^2} \leq K e^{-\lambda(T-t)}, \quad \forall t \in [0, T]
\end{equation*}
holds for some constants $K, \lambda > 0$, independent of $T$. Let $T \geq 2N$ be fixed but arbitrary, and let $k \geq N$ be the integer such that $N + k \leq T < N + k + 1$. Then, since $\|\Delta_{\Lambda}(k)\|_{L^2} = \|\Lambda - \Lambda_{T-k}(0)\|_{L^2}$, we have
$$\alpha \|\Delta_{\Lambda}(k)\|_{L^2} + \frac{\alpha^2}{\beta} e^{-2\beta(T-k)} \leq \rho,$$
which yields that, for all $0\leq t \leq k$,
\begin{equation*}
\|\Delta_{\Lambda}(t)\|_{L^2} \le \rho e^{-\beta(k-t)} + \alpha K_4 \int_t^k e^{-\beta(s-t)} \|\Delta_{\Lambda}(s)\|^2_{L^2} ds.
\end{equation*}
Set $K_5 = \alpha K_4$ and $h(t) = K_5 e^{\beta t} \|\Delta_{\Lambda}(k-t)\|_{L^2}$  for $0 \leq t \leq k$. Then, we derive the following inequality:
\begin{equation*}
h(t) \le \rho K_5 + \int_0^t e^{-\beta s} h^2(s) ds = \frac{\beta}{2} + \int_0^t e^{-\beta s} h^2(s) ds := \frac{\beta}{2} + H(t).
\end{equation*}
Thus, $H(0) = 0$ and $e^{\beta t} \frac{d}{dt}H(t) = h^2(t) \leq (\beta/2 + H(t))^2$, or equivalently,
\begin{equation*}
\frac{d}{dt}\bigg(\frac{-1}{\frac{\beta}{2}+H(t)}\bigg) = \frac{\frac{d}{dt}H(t)}{\big(\frac{\beta}{2}+H(t) \big)^2} \leq e^{-\beta t}.
\end{equation*}
It implies that 
$$\frac{2}{\beta}-\frac{1}{\frac{\beta}{2}+H(t)}\leq \int_0^t e^{-\beta s} ds \leq \frac{1}{\beta}, \quad \forall t \in [0, k].$$
Hence, we obtain $h(t) \leq \frac{\beta}{2} + H(t) \leq \beta$ for all $t \in [0, k]$. Therefore, since $T - k < N+1$, we deduce the following estimate:
\begin{equation*}
\begin{aligned}
\|\Delta_{\Lambda}(t)\|_{L^2} &= \frac{1}{K_5} e^{-\beta(k-t)} h(k-t) \leq \frac{\beta}{K_5}e^{\beta(T-k)} e^{-\beta(T-t)} \\
&\leq \frac{\beta}{K_5}e^{\beta(N+1)} e^{-\beta(T-t)} \leq K e^{-\beta(T-t)}, \quad \forall t \in [0, k].
\end{aligned}
\end{equation*}
For $t \in [k, T]$, we have $0 \leq T-t \leq T-k < N+1$. Thus, it holds that
\begin{equation*}
\begin{aligned}
\|\Delta_{\Lambda}(t)\|_{L^2} &= \|\Lambda - \Lambda_T(t)\|_{L^2} = \|\Lambda - \Lambda_{T-t}(0)\|_{L^2} \leq \max_{s \in [0, N+1]} \|\Lambda - \Lambda_{s}(0)\|_{L^2} := K_6 \\
&\leq K_6 e^{\beta(N+1)} e^{-\beta(T-t)} \leq K e^{-\beta(T-t)}, \quad \forall t \in [k, T].
\end{aligned}
\end{equation*}

It remains to show that the estimate $\|\Delta_{\Lambda}(t)\|_{L^2} \leq K e^{-\lambda(T-t)}$ for all $t \in [0,T]$ also holds when $T < 2N$. Note that
\begin{equation*}
\begin{aligned}
\|\Delta_{\Lambda}(t)\|_{L^2} &=  \|\Lambda - \Lambda_{T-t}(0)\|_{L^2} \leq \max_{s \in [0, 2N]} \|\Lambda - \Lambda_{s}(0)\|_{L^2} := K_7 \\
&\leq K_7 e^{2N\beta} e^{-\beta(T-t)} \leq K e^{-\beta(T-t)},
\end{aligned}
\end{equation*}
for all $0 \le t \leq T \leq 2N$. Therefore, we obtain
\begin{equation*}
\|\Delta_{\Lambda}(t)\|_{L^2} = \|\Lambda - \Lambda_T(t)\|_{L^2} \le K e^{-\lambda(T-t)}, \quad \forall t \in [0, T],
\end{equation*}
for some constants $K, \lambda > 0$, independent of $T$.

Next, by the definition of $\wt{\Theta}^*_T$ and $\wt{\Theta}^*$, it is clear that
\begin{equation*}
\begin{aligned}
\|\wt{\Theta}^* - \wt{\Theta}^*_T(t)\|_{L^2} &= \|\cR(P_T(t))^{-1}
\Upsilon(P_T(t), \Lambda_T(t)) - \cR(P)^{-1}
\Upsilon(P, \Lambda)\|_{L^2} \\
&\leq \|\cR(P_T(t))^{-1} (\Upsilon(P_T(t), \Lambda_T(t)) -
\Upsilon(P, \Lambda))\|_{L^2}  \\
& \hspace{0.3in} + \|(\cR(P_T(t))^{-1} - \cR(P)^{-1}) \Upsilon(P, \Lambda) \|_{L^2} \\
& \leq \|\cR(P_T(t))^{-1} \| \big(\|B\| \|\Lambda_T(t) - \Lambda\|_{L^2} + \|D^\top(P_T(t) - P) C\| \|G\|_{L^2} \big) \\
& \hspace{0.3in} + \|\cR(P_T(t))^{-1} D^\top(P - P_T(t)) D \cR(P)^{-1}\| \|\Upsilon(P, \Lambda) \|_{L^2}.
\end{aligned}
\end{equation*}
Then, the desired estimate \eqref{eq:exponential_estimate_tilde_Theta} follows directly from \eqref{eq:exponential_estimate_P_Pi} and \eqref{eq:exponential_estimate_Lambda}.
\end{proof}

We next establish an exponential estimate relating the solution $p_T$ of the differential equation \eqref{eq:ODE_p_T} to the solution $p$ of the algebraic equation \eqref{eq:Riccati_p_ergodic}.

\begin{proposition}
\label{p:exponential_estimates_p}
Suppose that Assumptions \ref{a:coefficients_solvability_finite}, \ref{a:stabilizability}, and \ref{a:solvability_Lambda_p} hold. Let $p_T$ and $p$ be the unique solutions to \eqref{eq:ODE_p_T} and \eqref{eq:Riccati_p_ergodic}, respectively. Then, there exist constants $K, \lambda > 0$, independent of $T$, such that
\begin{equation}
\label{eq:exponential_estimate_p}
\|p - p_T(t)\|_{L^2(I; \mathbb R^{d})} \leq K e^{-\lambda(T-t)}, \quad \forall t \in [0, T].
\end{equation}
Furthermore, there exist constants $K, \lambda > 0$, independent of $T$, such that
\begin{equation}
\label{eq:exponential_estimate_theta}
\|\theta^* - \theta^*_T(t)\|_{L^2(I; \mathbb R^{m})} \leq K e^{-\lambda(T-t)}, \quad \forall t \in [0, T].
\end{equation}
\end{proposition}
\begin{proof}
Denote $\Delta_{p}(t) = p - p_T(t)$ for all $t \in [0, T]$. Then, $\Delta_p$ satisfies the following differential equation:
\begin{equation*}
\begin{cases}
\dot{\Delta}_p(t)(u) + \wt{F}^{\infty}(p)(u) - \wt{F}(t, p_T(t))(u) = 0, \\
\Delta_{p}(T)(u) = p(u),
\end{cases}
\end{equation*}
for all $t \in [0, T]$ and a.e. $u \in I$. From the definition of $\wt{F}$ and $\wt{F}^{\infty}$, we obtain
\begin{equation*}
\begin{aligned}
& \dot{\Delta}_p(t)(u) + \Psi_2(P, \Pi, \Lambda, p)(u) - \Psi_2(P_T(t), \Pi_T(t), \Lambda_T(t), p_T(t))(u) \\
& \hspace{0.3in} - \wh{\cM}(P, \Pi)^\top \cR(P)^{-1} \Gamma(P, p)(u) + \wh{\cM}(P_T(t), \Pi_T(t))^\top \cR(P_T(t))^{-1} \Gamma(P_T(t), p_T(t))(u) \\
& \hspace{0.3in}  - \int_I \Upsilon(P, \Lambda)(v, u)^\top \cR(P)^{-1} \Gamma(P, p)(v) dv  \\
& \hspace{0.3in} + \int_I \Upsilon(P_T(t), \Lambda_T(t))(v, u)^\top \cR(P_T(t))^{-1} \Gamma(P_T(t), p_T(t))(v) dv = 0,
\end{aligned}
\end{equation*}
which is equivalent to 
\begin{equation*}
\begin{aligned}
& \dot{\Delta}_p(t)(u) + (A + \bar{A} + B\bar{\Theta}_T^*(t))^\top \Delta_{p}(t)(u) + \int_I \big(\wt{A}G(v,u) + B\wt{\Theta}_T^*(t)(v,u)\big)^\top \Delta_{p}(t)(v) dv \\
& \hspace{0.3in} + f_4(\Delta_P(t), \Delta_{\bar{\Pi}}(t), \Delta_{\Lambda}(t))(u),
\end{aligned}
\end{equation*}
where
\begin{equation*}
\begin{aligned}
& f_4(\Delta_P(t), \Delta_{\bar{\Pi}}(t), \Delta_{\Lambda}(t))(u) \\
= \ & \Delta_{\bar{\Pi}}(t) b + (C + \bar{C})^\top \Delta_P(t) \sigma + \wt{C}^\top \int_I G(v, u) \Delta_P(t) \sigma dv + \frac{1}{2} \int_I \big(\Delta_{\Lambda}(u, v) + \Delta_{\Lambda}(v, u)^\top \big) b dv \\
&- \big(B^\top \Delta_{\bar{\Pi}}(t) + D^\top \Delta_P(t)(C + \bar{C})\big)^\top \mathcal{R}(P)^{-1} \Gamma(P, p)(u) \\
&+ \widehat{\mathcal{M}}(P_T(t), \Pi_T(t))^\top \mathcal{R}(P)^{-1} D^\top \Delta_P(t) D\Gamma(P, p)(u) - \widehat{\mathcal{M}}(P_T(t), \Pi_T(t))^\top \mathcal{R}(P_T(t))^{-1} D^\top \Delta_P(t) \sigma \\
&- \int_I \Delta_{\Lambda}(v, u)^\top B \mathcal{R}(P)^{-1} \Gamma(P, p)(v) dv - \int_I \wt{C}^\top \Delta_P(t) D G(u, v) \mathcal{R}(P)^{-1} \Gamma(P, p)(v) dv \\
&+ \int_I \Upsilon(P_T(t), \Lambda_T(t))(v, u)^\top \mathcal{R}(P)^{-1} D^\top \Delta_P(t) D \mathcal{R}(P_T(t))^{-1} \Gamma(P, p)(v) dv \\
&- \int_I \Upsilon(P_T(t), \Lambda_T(t))(v, u)^\top \mathcal{R}(P_T(t))^{-1} D^\top \Delta_P(t) \sigma dv.
\end{aligned}
\end{equation*}
From the estimates in Lemma \ref{l:exponential_estimates_P_Pi} and Proposition \ref{p:exponential_estimates_Lambda}, we derive that
\begin{equation*}
\|f_4(\Delta_P(t), \Delta_{\bar{\Pi}}(t), \Delta_{\Lambda}(t))\|_{L^2} \leq K e^{-\lambda(T-t)}, \quad \forall t \in [0, T],
\end{equation*}
for some constants $K, \lambda > 0$, independent of $T$. Recall the definition of the linear operator $\mathcal{L}$ in \eqref{eq:linear_operator} and its corresponding adjoint operator $\mathcal{L}^*$. Similarly, for $\phi \in L^2(I; \mathbb{R}^d)$, we define the following linear operator 
\begin{equation*}
\mathcal{L}_T(t, \phi)(u) = (A + \bar{A} + B \bar{\Theta}^*_T(t)) \phi(u) + \int_I \big(\wt{A} G(u, v) + B \wt{\Theta}^*_T(t)(u, v) \big) \phi(v) dv
\end{equation*}
and denote its adjoint operator by $\mathcal{L}^*_T$. 
Then, the equation for $\Delta_{p}(t)$ can be rewritten as
\begin{equation*}
\begin{cases}
\dot{\Delta}_{p}(t)(u) + \mathcal{L}^*_T(t, \Delta_{p}(t))(u) + f_4(\Delta_P(t), \Delta_{\bar{\Pi}}(t), \Delta_{\Lambda}(t))(u) = 0, \\
\Delta_{p}(T)(u) = p(u).
\end{cases}
\end{equation*}
From Assumption \ref{a:stabilizability}, we also have $\|e^{t\mathcal{L}^*}\|_{op} \leq K e^{-\lambda t}$ for all $t \geq 0$. Note that
\begin{equation*}
\begin{aligned}
\dot{\Delta}_{p}(t)(u) + \mathcal{L}^*(\Delta_{p}(t))(u) &= (\mathcal{L}^*(\Delta_{p}(t)) - \mathcal{L}_T^*(t, \Delta_{p}(t)))(u) - f_4(\Delta_P(t), \Delta_{\bar{\Pi}}(t), \Delta_{\Lambda}(t))(u)\\
&= \big(B(\bar{\Theta}^* - \bar{\Theta}^*_T(t)) \big)^\top \Delta_{p}(t)(u) + \int_I \big(B(\wt{\Theta}^* - \wt{\Theta}^*_T(t))(v, u)\big)^\top \Delta_{p}(t)(v) dv \\
& \hspace{0.3in} - f_4(\Delta_P(t), \Delta_{\bar{\Pi}}(t), \Delta_{\Lambda}(t))(u),
\end{aligned}
\end{equation*}
which implies that
\begin{equation*}
\begin{aligned}
\Delta_{p}(t)(u) &= e^{(T-t)\mathcal{L}^*} p(u) - \int_t^T e^{(s-t)\mathcal{L}^*} (\bar{\Theta}^* - \bar{\Theta}^*_T(s))^\top B^\top \Delta_{p}(s)(u) ds \\
& \hspace{0.3in} - \int_t^T e^{(s-t)\mathcal{L}^*} \int_I (\wt{\Theta}^* - \wt{\Theta}^*_T(s))(v, u)^\top B^\top \Delta_{p}(s)(v) dv ds \\
& \hspace{0.3in} + \int_t^T e^{(s-t)\mathcal{L}^*} f_4(\Delta_P(s), \Delta_{\bar{\Pi}}(s), \Delta_{\Lambda}(s))(u) ds.
\end{aligned}
\end{equation*}
Hence, we deduce that
\begin{equation*}
\begin{aligned}
\|\Delta_{p}(t)\|_{L^2} &\leq \|e^{(T-t)\mathcal{L}^*}\|_{op} \|p\|_{L^2} + K \int_t^T \|e^{(s-t)\mathcal{L}^*}\|_{op} \|\bar{\Theta}^* - \bar{\Theta}^*_T(s)\| \|\Delta_{p}(s)\|_{L^2} ds \\
& \hspace{0.3in} + K \int_t^T \|e^{(s-t)\mathcal{L}^*}\|_{op} \|\wt{\Theta}^* - \wt{\Theta}^*_T(s)\|_{L^2} \|\Delta_{p}(s)\|_{L^2} ds \\
& \hspace{0.3in} + \int_t^T \|e^{(s-t)\mathcal{L}^*}\|_{op} \|f_4(\Delta_P(s), \Delta_{\bar{\Pi}}(s), \Delta_{\Lambda}(s))\|_{L^2} ds.
\end{aligned}
\end{equation*}
We can choose constants $K, \lambda > 0$, independent of $T$, such that $\|e^{t\mathcal{L}^*}\|_{op} \leq K e^{-\lambda t}$ for all $t \geq 0$ and
\begin{equation*}
\|\bar{\Theta}^* - \bar{\Theta}^*_T(t)\| + \|\wt{\Theta}^* - \wt{\Theta}^*_T(t)\|_{L^2} + \|f_4(\Delta_P(t), \Delta_{\bar{\Pi}}(t), \Delta_{\Lambda}(t))\|_{L^2} \leq K e^{-2\lambda(T-t)}, \quad \forall t \in [0, T].
\end{equation*}
Then, we obtain the following estimate:
\begin{equation*}
\begin{aligned}
\|\Delta_{p}(t)\|_{L^2} &\leq K e^{-\lambda(T-t)} + K \int_t^T e^{-\lambda(s-t)} e^{-2\lambda(T-s)} \|\Delta_{p}(s)\|_{L^2} ds \\
& \hspace{0.3in} + K \int_t^T e^{-\lambda(s-t)} e^{-2\lambda(T-s)} ds \\
&\leq K e^{-\lambda(T-t)} + K e^{-\lambda(T-t)} \int_t^T e^{-\lambda(T-s)} \|\Delta_{p}(s)\|_{L^2} ds,
\end{aligned}
\end{equation*}
which yields that
\begin{equation*}
    e^{\lambda(T-t)} \|\Delta_{p}(t)\|_{L^2} \le K + K \int_t^T e^{-2\lambda(T-s)} (e^{\lambda(T-s)} \|\Delta_{p}(s)\|_{L^2}) ds.
\end{equation*}
Let $g(t) = e^{\lambda(T-t)} \|\Delta_{p}(t)\|_{L^2}$ and $\wt{g}(\tau) = g(T-\tau)$ for all $t, \tau \in [0, T]$. Then, we have
\begin{equation*}
g(t) \leq K + K \int_t^T e^{-2\lambda(T-s)} g(s) ds,
\end{equation*}
and therefore
\begin{equation*}
\wt{g}(\tau) \leq K + K \int_0^{\tau} e^{-2\lambda s} \wt{g}(s) ds.
\end{equation*}
Hence, Gr\"onwall's inequality implies that
\begin{equation*}
\wt{g}(\tau) \leq K e^{\int_0^{\tau} K e^{-2\lambda s} ds} \leq K e^{\frac{K}{2\lambda}} \leq K
\end{equation*}
for all $\tau \in [0, T]$. Therefore, we obtain $e^{\lambda(T-t)} \|\Delta_{p}(t)\|_{L^2} \leq K$ for all $t \in [0, T]$, which is equivalent to
\begin{equation*}
\|\Delta_{p}(t)\|_{L^2} \le K e^{-\lambda(T-t)}, \quad \forall t \in [0, T].
\end{equation*}
Next, the estimate \eqref{eq:exponential_estimate_theta} follows by the same argument used to establish \eqref{eq:exponential_estimate_tilde_Theta} in Proposition \ref{p:exponential_estimates_Lambda}. This completes the proof.
\end{proof}


\section{Main results}
\label{s:main_result}

\subsection{Turnpike property}
\label{s:turnpike_properties}

We now present the main result of this paper, which establishes a turnpike property characterizing the relationship between the optimal pair for the finite-horizon control problem \eqref{eq:cost_functional_finite_horizon} and that for the ergodic control problem \eqref{eq:cost_functional_infinite_horizon}. 

\begin{theorem}
\label{t:turnpike_property}
Suppose that Assumptions \ref{a:coefficients_solvability_finite}, \ref{a:stabilizability}, and \ref{a:solvability_Lambda_p} hold. Let $(X_T, \alpha_T)$, given by \eqref{eq:optimal_state_finite}-\eqref{eq:optimal_control_finite}, be the optimal pair for the finite-horizon GMFC problem with initial state $\xi_1 \in \kL_0^2$, and let $(X_{\infty}, \alpha_{\infty})$, given by \eqref{eq:optimal_state_ergodic}-\eqref{eq:optimal_control_ergodic}, be the optimal pair for the ergodic GMFC problem with a possibly different initial state $\xi_2 \in \kL_0^2$. Then, there exist positive constants $K$ and $\lambda$, independent of $T$, such
that
\begin{equation}
\label{eq:turnpike_property}
\mathbb{E} \Big[\int_I |X_T^u(t) - X_{\infty}^u(t)|^2 du \Big] + \mathbb{E} \Big[\int_I |\alpha_T^u(t) - \alpha_{\infty}^u(t)|^2 du \Big] \leq K \big(e^{-\lambda t} + e^{-\lambda(T-t)} \big), \quad \forall t \in [0, T].
\end{equation}
Specifically, if $\mathbb{E}[|\xi_1^u - \xi_2^u|^2] = 0$ for a.e. $u \in I$, then
\begin{equation}
\label{eq:turnpike_property_same_initial}
\mathbb{E} \Big[\int_I |X_T^u(t) - X_{\infty}^u(t)|^2 du \Big] + \mathbb{E} \Big[\int_I |\alpha_T^u(t) - \alpha_{\infty}^u(t)|^2 du \Big] \leq K e^{-\lambda(T-t)}, \quad \forall t \in [0, T].
\end{equation}
\end{theorem}

\begin{proof}
For $t \in [0, T]$ and $u \in I$, define  $\wh{X}^u(t) := X_T^u(t) - X_{\infty}^u(t)$. Subtracting the closed-loop system \eqref{eq:optimal_state_ergodic} for $X_{\infty}$ from the closed-loop system \eqref{eq:optimal_state_finite} for $X_{T}$, we have
\begin{equation*}
\begin{aligned}
d\wh{X}^u(t) &= \Big\{(A + B \Theta_T^*(t)) \big(\wh{X}^u(t) - \mathbb{E}[\wh{X}^u(t)] \big) + B(\Theta_T^*(t) - \Theta^*) \big(X_{\infty}^u(t) - \mathbb{E}[X_{\infty}^u(t)] \big) \\
& \hspace{0.4in} + (A + \bar{A} + B \bar{\Theta}_T^*(t)) \mathbb{E}[\wh{X}^u(t)] + B(\bar{\Theta}_T^*(t) - \bar{\Theta}^*) \mathbb{E}[X_{\infty}^u(t)] \\
& \hspace{0.4in} + B(\theta_T^*(t) - \theta^*)(u) + \int_I \big(\wt{A} G(u, v) + B \wt{\Theta}_T^*(t)(u, v) \big) \mathbb{E}[\wh{X}^v(t)] dv \\
& \hspace{0.4in} + \int_I B \big(\wt{\Theta}_T^*(t)(u, v) - \wt{\Theta}^*(u, v) \big) \mathbb{E}[X_{\infty}^v(t)] dv \Big\} dt \\
& \quad + \Big\{(C + D \Theta_T^*(t)) \big(\wh{X}^u(t) - \mathbb{E}[\wh{X}^u(t)] \big) + D (\Theta_T^*(t) - \Theta^*)\big(X_{\infty}^u(t) - \mathbb{E}[X_{\infty}^u(t)] \big) \\
& \hspace{0.4in} + (C + \bar{C} + D \bar{\Theta}_T^*(t)) \mathbb{E}[\wh{X}^u(t)] + D(\bar{\Theta}_T^*(t) - \bar{\Theta}^*) \mathbb{E}[X_{\infty}^u(t)] \\
&\hspace{0.4in} + D(\theta_T^*(t) - \theta^*)(u) + \int_I \big(\wt{C} G(u, v) + D \wt{\Theta}_T^*(t)(u, v) \big) \mathbb{E}[\wh{X}^v(t)] dv \\
&\hspace{0.4in} + \int_I D \big(\wt{\Theta}_T^*(t)(u, v) - \wt{\Theta}^*(u, v) \big) \mathbb{E}[X_{\infty}^v(t)] dv \Big\} dW^u(t)
\end{aligned}
\end{equation*}
with the initial condition $\wh{X}^u(0) = \xi_1^u - \xi_2^u$ for $u \in I$. First, we give an estimate for $\|\mathbb{E}[\wh{X}(t)]\|_{L^2}$. Taking expectations in the above SDE yields
\begin{equation*}
\begin{aligned}
d \mathbb{E}[\wh{X}^u(t)] &= \Big\{(A + \bar{A} + B \bar{\Theta}_T^*(t)) \mathbb{E}[\wh{X}^u(t)] + B(\bar{\Theta}_T^*(t) - \bar{\Theta}^*) \mathbb{E}[X_{\infty}^u(t)] \\
&\hspace{0.3in} + B(\theta_T^*(t) - \theta^*)(u) + \int_I \big(\wt{A} G(u, v) + B \wt{\Theta}_T^*(t)(u, v) \big) \mathbb{E}[\wh{X}^v(t)] dv \\
& \hspace{0.3in} + \int_I B \big(\wt{\Theta}_T^*(t)(u, v) - \wt{\Theta}^*(u, v) \big) \mathbb{E}[X_{\infty}^v(t)] dv \Big\} dt \\
&= \Big\{(A + \bar{A} + B \bar{\Theta}^*) \mathbb{E}[\wh{X}^u(t)] + \int_I \big(\wt{A} G(u, v) + B \wt{\Theta}^*(u, v) \big) \mathbb{E}[\wh{X}^v(t)] dv \\
& \hspace{0.3in} + B(\bar{\Theta}_T^*(t) - \bar{\Theta}^*) \mathbb{E}[\wh{X}^u(t)] + B(\bar{\Theta}_T^*(t) - \bar{\Theta}^*) \mathbb{E}[X_{\infty}^u(t)] \\
& \hspace{0.3in} + \int_I B\big(\wt{\Theta}_T^*(t)(u, v) - \wt{\Theta}^*(u, v) \big) \mathbb{E}[\wh{X}^v(t)] dv \\
& \hspace{0.3in} + \int_I B \big(\wt{\Theta}_T^*(t)(u, v) - \wt{\Theta}^*(u, v) \big) \mathbb{E}[X_{\infty}^v(t)] dv + B(\theta_T^*(t) - \theta^*)(u) \Big\} dt
\end{aligned}
\end{equation*}
with $\mathbb{E}[\wh{X}^u(0)] = \mathbb{E}[\xi_1^u - \xi_2^u]$ for $u \in I$. By the definition of the linear operator $\mathcal{L}$ in \eqref{eq:linear_operator}, we obtain that 
\begin{equation*}
\begin{aligned}
\mathbb{E}[\wh{X}^u(t)] &= e^{t \mathcal{L}} \mathbb{E}[\xi_1^u - \xi_2^u] + \int_0^t e^{(t-s) \mathcal{L}} \Big\{ B(\bar{\Theta}_T^*(s) - \bar{\Theta}^*) \mathbb{E}[\wh{X}^u(s)] + B(\bar{\Theta}_T^*(s) - \bar{\Theta}^*) \mathbb{E}[X_{\infty}^u(s)] \\
&\hspace{0.4in} + \int_I B \big(\wt{\Theta}_T^*(s)(u, v) - \wt{\Theta}^*(u, v) \big) \mathbb{E}[\wh{X}^v(s)] dv \\
& \hspace{0.4in} + \int_I B \big(\wt{\Theta}_T^*(s)(u, v) - \wt{\Theta}^*(u, v) \big) \mathbb{E}[X_{\infty}^v(s)] dv + B(\theta_T^*(s) - \theta^*)(u) \Big\} ds.
\end{aligned}
\end{equation*}
Using the estimates \eqref{eq:operator_uniform_stable}, \eqref{eq:moment_boundness}, \eqref{eq:exponential_estimate_Theta}, \eqref{eq:exponential_estimate_tilde_Theta}, and \eqref{eq:exponential_estimate_theta}, together with Minkowski's and Cauchy--Schwarz inequalities, we derive the following inequalities:
\begin{equation*}
\begin{aligned}
\|\mathbb{E}[\wh{X}(t)]\|_{L^2} & \leq \|e^{t \mathcal{L}}\|_{op} \Big(\int_I \mathbb{E}\big[|\xi_1^u - \xi_2^u|^2 \big] du \Big)^{\frac{1}{2}} \\
& \hspace{0.3in} + \int_0^t \|e^{(t-s)\mathcal{L}}\|_{op} K e^{-\lambda(T-s)} \big( \|\mathbb{E}[\wh{X}(s)]\|_{L^2} + \|\mathbb{E}[X_{\infty}(s)]\|_{L^2} + 1 \big) ds \\
& \leq K e^{-\lambda t} + \int_0^t K e^{-\lambda(t-s)} e^{-\lambda(T-s)} \big(\|\mathbb{E}[\wh{X}(s)]\|_{L^2} + 1 \big) ds \\
& \leq K \big(e^{-\lambda t} + e^{-\lambda(T-t)} \big) + K \int_0^t e^{-\lambda(T+t-2s)} \|\mathbb{E}[\wh{X}(s)]\|_{L^2} ds.
\end{aligned}
\end{equation*}
Gr\"onwall's inequality yields the following estimate:
\begin{equation}
\label{eq:X_hat_mean_estimate}
\|\mathbb{E}[\wh{X}(t)]\|_{L^2} \leq K \big(e^{-\lambda t} + e^{-\lambda(T-t)} \big) e^{\int_0^t K e^{-\lambda(T+t-2s)} ds} \leq K \big(e^{-\lambda t} + e^{-\lambda(T-t)} \big), \quad \forall t \in [0, T],
\end{equation}
where $K$ and $\lambda$ are independent of $t$ and $T$.

Next, we estimate $\wh{X}^u(t) - \mathbb{E}[\wh{X}^u(t)]$. It is clear that
\begin{equation*}
\begin{aligned}
& d \big(\wh{X}^u(t) - \mathbb{E}[\wh{X}^u(t)] \big) \\
= \ & \big\{(A + B \Theta^*) \big(\wh{X}^u(t) - \mathbb{E}[\wh{X}^u(t)] \big) + B(\Theta_T^*(t) - \Theta^*) \big(\wh{X}^u(t) - \mathbb{E}[\wh{X}^u(t)] \big) \\
& \hspace{0.2in} + B(\Theta_T^*(t) - \Theta^*) \big(X_{\infty}^u(t) - \mathbb{E}[X_{\infty}^u(t)] \big) \big\} dt \\
& + \Big\{(C + D \Theta_T^*(t))\big(\wh{X}^u(t) - \mathbb{E}[\wh{X}^u(t)] \big) + D(\Theta_T^*(t) - \Theta^*)\big(\wh{X}^u(t) - \mathbb{E}[\wh{X}^u(t)] \big) \\
& \hspace{0.2in} + D(\Theta_T^*(t) - \Theta^*)\big(X_{\infty}^u(t) - \mathbb{E}[X_{\infty}^u(t)] \big) + (C + \bar{C} + D \bar{\Theta}_T^*(t)) \mathbb{E}[\wh{X}^u(t)] \\
& \hspace{0.2in} + D(\bar{\Theta}_T^*(t) - \bar{\Theta}^*) \mathbb{E}[X_{\infty}^u(t)] + D(\theta_T^*(t) - \theta^*)(u) + \int_I \big(\wt{C} G(u, v) + D \wt{\Theta}_T^*(t)(u, v) \big) \mathbb{E}[\wh{X}^v(t)] dv \\
& \hspace{0.2in} + \int_I D \big(\wt{\Theta}_T^*(t)(u, v) - \wt{\Theta}^*(u, v) \big) \mathbb{E}[X_{\infty}^v(t)] dv \Big\} dW^u(t) \\
:= \ & \big\{(A + B \Theta^*) \big(\wh{X}^u(t) - \mathbb{E}[\wh{X}^u(t)] \big) + B(\Theta_T^*(t) - \Theta^*) \big(\wh{X}^u(t) - \mathbb{E}[\wh{X}^u(t)] \big) + \Sigma_1^u(t) \big\} dt \\
& \hspace{0.2in} + \big\{(C + D \Theta^*) \big(\wh{X}^u(t) - \mathbb{E}[\wh{X}^u(t)] \big) + D(\Theta_T^*(t) - \Theta^*) \big(\wh{X}^u(t) - \mathbb{E}[\wh{X}^u(t)] \big) + \Sigma_2^u(t) \big\} dW^u(t).
\end{aligned}
\end{equation*}
By the estimates \eqref{eq:moment_boundness} and \eqref{eq:exponential_estimate_Theta}, we have
\begin{equation*}
\begin{aligned}
\mathbb{E} \Big[\int_I |\Sigma_1^u(t)|^2 du \Big] &= \mathbb{E} \Big[\int_I \big| B(\Theta_T^*(t) - \Theta^*) \big(X_{\infty}^u(t) - \mathbb{E}[X_{\infty}^u(t)] \big) \big|^2 du \Big] \\
& \leq K \|\Theta_T^*(t) - \Theta^*\|^2 \mathbb{E} \Big[ \int_I |X_{\infty}^u(t)|^2 du \Big] \\
& \leq K e^{-2\lambda(T-t)}, \quad \forall t \in [0, T].
\end{aligned}
\end{equation*}
Similarly, using the estimates \eqref{eq:moment_boundness}, \eqref{eq:exponential_estimate_Theta}, \eqref{eq:exponential_estimate_tilde_Theta}, \eqref{eq:exponential_estimate_theta}, and \eqref{eq:X_hat_mean_estimate}, we deduce the following estimate for $\Sigma_2(t)$:
\begin{equation*}
\begin{aligned}
\mathbb{E} \Big[\int_I |\Sigma_2^u(t)|^2 du \Big] & \leq K \Big\{ \|\Theta_T^*(t) - \Theta^*\|^2 \int_I \big|\mathbb{E}[X_{\infty}^u(t)] \big|^2 du + \|C + \bar{C} + D \bar{\Theta}_T^*(t)\|^2 \int_I \big|\mathbb{E}[\wh{X}^u(t)] \big|^2 du \\
& \hspace{0.3in} + \|\theta_T^*(t) - \theta^*\|_{L^2}^2 + \big(\|G\|_{L^2}^2 + \|\wt{\Theta}_T^*(t)\|_{L^2}^2 \big) \int_I \big|\mathbb{E}[\wh{X}^u(t)] \big|^2 du \\
& \hspace{0.3in} + \|\wt{\Theta}_T^*(t) - \wt{\Theta}^*\|_{L^2}^2 \int_I \big|\mathbb{E}[X_{\infty}^v(t)] \big|^2 dv \Big\} \\
& \leq K \big(e^{-2 \lambda t} + e^{-2\lambda(T-t)} \big), \quad \forall t \in [0, T].
\end{aligned}
\end{equation*}
Next, let $\bar{P} \in \mathbb{S}_{++}^d$ be the solution to the Lyapunov equation \eqref{eq:Lyapunov_equation}. Applying It\^o's formula, we obtain
\begin{equation*}
\begin{aligned}
&\frac{d}{dt} \mathbb{E} \big[ \big\langle \bar{P}\big(\wh{X}^u(t) - \mathbb{E}[\wh{X}^u(t)] \big), \wh{X}^u(t) - \mathbb{E}[\wh{X}^u(t)] \big\rangle \big] \\
= \ &  \mathbb{E} \big[  \big\langle \{(A + B \Theta^*)^\top \bar{P} + \bar{P}(A + B \Theta^*) + (C + D \Theta^*)^\top \bar{P}(C + D \Theta^*)\} \big(\wh{X}^u(t) - \mathbb{E}[\wh{X}^u(t)] \big), \\
& \hspace{0.3in} \wh{X}^u(t) - \mathbb{E}[\wh{X}^u(t)] \big\rangle + \big\langle \big\{ (\Theta_T^*(t) - \Theta^*)^\top B^\top \bar{P} + \bar{P} B(\Theta_T^*(t) - \Theta^*) + (C + D \Theta^*)^\top \bar{P} D(\Theta_T^*(t) - \Theta^*) \\
& \hspace{0.3in} + (\Theta_T^*(t) - \Theta^*)^\top D^\top \bar{P}(C + D \Theta^*) + (\Theta_T^*(t) - \Theta^*)^\top D^\top \bar{P} D(\Theta_T^*(t) - \Theta^*) \big\} \big(\wh{X}^u(t) - \mathbb{E}[\wh{X}^u(t)] \big), \\
& \hspace{0.3in} \wh{X}^u(t) - \mathbb{E}[\wh{X}^u(t)] \big\rangle + 2 \big\langle \bar{P} \Sigma_1^u(t), \wh{X}^u(t) - \mathbb{E}[\wh{X}^u(t)] \big\rangle \\
& \hspace{0.3in} + 2 \big\langle \big\{(C + D \Theta^*)^\top \bar{P} + (\Theta_T^*(t) - \Theta^*)^\top D^\top \bar{P} \big\} \Sigma_2^u(t), \wh{X}^u(t) - \mathbb{E}[\wh{X}^u(t)] \big\rangle + \langle \bar{P} \Sigma_2^u(t), \Sigma_2^u(t) \rangle \big] \\
:= \ & \mathbb{E} \big[ -\big|\wh{X}^u(t) - \mathbb{E}[\wh{X}^u(t)] \big|^2 + \big\langle H_1(t) \big(\wh{X}^u(t) - \mathbb{E}[\wh{X}^u(t)] \big), \wh{X}^u(t) - \mathbb{E}[\wh{X}^u(t)] \big\rangle \\
& \hspace{0.3in} + 2 \big\langle H_2^u(t), \wh{X}^u(t) - \mathbb{E}[\wh{X}^u(t)] \big\rangle + \langle \bar{P} \Sigma_2^u(t), \Sigma_2^u(t) \rangle \big],
\end{aligned}
\end{equation*}
where
\begin{equation*}
\begin{aligned}
H_1(t) &:= (\Theta_T^*(t) - \Theta^*)^\top B^\top \bar{P} + \bar{P} B(\Theta_T^*(t) - \Theta^*) + (C + D \Theta^*)^\top \bar{P} D(\Theta_T^*(t) - \Theta^*) \\
& \hspace{0.3in} + (\Theta_T^*(t) - \Theta^*)^\top D^\top \bar{P}(C + D \Theta^*) + (\Theta_T^*(t) - \Theta^*)^\top D^\top \bar{P} D(\Theta_T^*(t) - \Theta^*),
\end{aligned}
\end{equation*}
and
\begin{equation*}
H_2^u(t) := \bar{P} \Sigma_1^u(t) + \big\{(C + D \Theta^*)^\top \bar{P} + (\Theta_T^*(t) - \Theta^*)^\top D^\top \bar{P} \big\} \Sigma_2^u(t)
\end{equation*}
for all $t \in [0, T]$ and a.e. $u \in I$. Then, using the inequality \eqref{eq:exponential_estimate_Theta}, we have $\|H_1(t)\| \leq K e^{-\lambda(T-t)}$ for all $t \in [0, T]$. Similarly, the following estimate holds:
\begin{equation*}
\begin{aligned}
\mathbb{E} \Big[\int_I |H_2^u(t)|^2 du \Big] & \leq K \Big\{ \mathbb{E} \Big[\int_I |\Sigma_1^u(t)|^2 du \Big] + \big(1 + \|\Theta_T^*(t) - \Theta^*\|^2 \big) \mathbb{E}\Big[\int_I |\Sigma_2^u(t)|^2 du \Big] \Big\} \\
& \leq K \big(e^{-2 \lambda t} + e^{-2\lambda(T-t)} \big), \quad \forall t \in [0, T].
\end{aligned}
\end{equation*}
Hence, by Young's inequality and the above estimates for $\Sigma_2, H_1$, and $H_2$, we deduce that
\begin{equation*}
\begin{aligned}
&\frac{d}{dt} \mathbb{E} \Big[ \int_I \big\langle \bar{P}\big(\wh{X}^u(t) - \mathbb{E}[\wh{X}^u(t)] \big), \wh{X}^u(t) - \mathbb{E}[\wh{X}^u(t)] \big\rangle du \Big] \\
\leq \ & \Big( -\frac{1}{2} + K e^{-\lambda(T-t)} \Big)  \mathbb{E} \Big[\int_I \big|\wh{X}^u(t) - \mathbb{E}[\wh{X}^u(t)] \big|^2 du \Big] + K \big(e^{-\lambda t} + e^{-\lambda(T-t)} \big)
\end{aligned}
\end{equation*}
for all $t \in [0, T]$. Let $\beta^{*} > 0$ and $\beta' > 0$ denote the largest and the smallest eigenvalues of $\bar{P}$, respectively. Then, 
$$\beta' \big|\wh{X}^u(t) - \mathbb{E}[\wh{X}^u(t)] \big|^2 \leq \big\langle \bar{P} \big(\wh{X}^u(t) - \mathbb{E}[\wh{X}^u(t)] \big), \wh{X}^u(t) - \mathbb{E}[\wh{X}^u(t)] \big\rangle \leq \beta^{*} \big|\wh{X}^u(t) - \mathbb{E}[\wh{X}^u(t)] \big|^2$$
for all $u \in I$, which yields the following inequality:
\begin{equation*}
\begin{aligned}
&\frac{d}{dt} \mathbb{E} \Big[ \int_I \big\langle \bar{P}\big(\wh{X}^u(t) - \mathbb{E}[\wh{X}^u(t)] \big), \wh{X}^u(t) - \mathbb{E}[\wh{X}^u(t)] \big\rangle du \Big] \\
\leq \ & g_2(t)  \mathbb{E} \Big[ \int_I \big\langle \bar{P}\big(\wh{X}^u(t) - \mathbb{E}[\wh{X}^u(t)] \big), \wh{X}^u(t) - \mathbb{E}[\wh{X}^u(t)] \big\rangle du \Big] + K \big(e^{-\lambda t} + e^{-\lambda(T-t)} \big)
\end{aligned}
\end{equation*}
for all $t \in [0, T]$. Here, we define $g_2(t) := -\frac{1}{2 \beta^{*}} + \frac{K}{\beta'} e^{-\lambda(T-t)}$ for all $t \in [0, T]$. Then, for $0 \leq s \leq t \leq T$, we have $e^{\int_s^t g_2(r) dr} \leq e^{\frac{K}{\lambda \beta'} - \frac{1}{2 \beta^{*}}(t-s)}$. Thus,
\begin{equation*}
\int_0^t e^{\int_s^t g_2(r) dr} K \big(e^{-\lambda s} + e^{-\lambda(T-s)} \big) ds \leq K \big(e^{-\lambda t} + e^{-\lambda(T-t)} \big)
\end{equation*}
for possibly different positive constants $K$ and $\lambda$. By Gr\"onwall's inequality, we derive that
\begin{equation*}
\begin{aligned}
& \mathbb{E} \Big[ \int_I \big\langle \bar{P}\big(\wh{X}^u(t) - \mathbb{E}[\wh{X}^u(t)] \big), \wh{X}^u(t) - \mathbb{E}[\wh{X}^u(t)] \big\rangle du \Big] \\
\leq \ & \mathbb{E} \Big[ \int_I \big\langle \bar{P}\big(\wh{X}^u(0) - \mathbb{E}[\wh{X}^u(0)] \big), \wh{X}^u(0) - \mathbb{E}[\wh{X}^u(0)] \big\rangle du \Big] e^{\int_0^t g_2(r) dr} \\
& \hspace{0.2in} + \int_0^t e^{\int_s^t g_2(r) dr} K \big(e^{-\lambda s} + e^{-\lambda(T-s)} \big) ds \\
\leq \ & K \big(e^{-\lambda t} + e^{-\lambda(T-t)} \big)
\end{aligned}
\end{equation*}
for all $t \in [0, T]$. Since $\bar{P} \in \mathbb{S}_{++}^d$, it follows that
\begin{equation*}
\mathbb{E} \Big[ \int_I \big|\wh{X}^u(t) - \mathbb{E}[\wh{X}^u(t)] \big|^2 du \Big] \leq K \big(e^{-\lambda t} + e^{-\lambda(T-t)} \big), \quad \forall t \in [0, T].
\end{equation*}
Combining this estimate with \eqref{eq:X_hat_mean_estimate}, we conclude the turnpike inequality for the optimal state:
\begin{equation*}
\mathbb{E} \Big[\int_I \big|X_T^u(t) - X_{\infty}^u(t) \big|^2 du \Big] = \mathbb{E} \Big[\int_I \big|\wh{X}^u(t) \big|^2 du \Big] \leq K \big(e^{-\lambda t} + e^{-\lambda(T-t)} \big)
\end{equation*}
for all $t \in [0, T]$, where $K$ and $\lambda$ are independent of $t$ and $T$.

Next, we establish the turnpike estimate for the optimal controls. From the feedback form of the controls $\alpha_T$ in \eqref{eq:optimal_control_finite} and $\alpha_{\infty}$ in \eqref{eq:optimal_control_ergodic}, we obtain
\begin{equation*}
\begin{aligned}
\alpha_T^u(t) - \alpha_{\infty}^u(t) &= \Theta_T^*(t) \wh{X}^u(t) + (\Theta_T^*(t) - \Theta^*) X^u_{\infty}(t) + (\bar{\Theta}_T^*(t) - \Theta_T^*(t)) \mathbb{E}[\wh{X}^u(t)] \\
& \hspace{0.3in} + \big[(\bar{\Theta}_T^*(t) -\bar{\Theta}^*) - (\Theta_T^*(t) - \Theta^*) \big] \mathbb{E}[X_\infty^u(t)] + (\theta_T^*(t) - \theta^*)(u) \\
& \hspace{0.3in} + \int_I \wt{\Theta}_T^*(t)(u, v) \mathbb{E}[\wh{X}^v(t)] dv + \int_I \big(\wt{\Theta}_T^*(t)(u, v) - \wt{\Theta}^*(u, v) \big) \mathbb{E}[X_{\infty}^v(t)] dv,
\end{aligned}
\end{equation*}
which yields
\begin{equation*}
\begin{aligned}
\mathbb{E} \Big[ \int_I \big|\alpha_T^u(t) - \alpha_{\infty}^u(t) \big|^2 du \Big] & \leq K \Big\{ \|\Theta_T^*(t)\|^2 \mathbb{E} \Big[\int_I \big|\wh{X}^u(t) \big|^2 du \Big] + \|\Theta_T^*(t) - \Theta^*\|^2 \mathbb{E} \Big[\int_I \big|X_{\infty}^u(t) \big|^2 du \Big] \\
& \hspace{0.3in} + \big(\|\bar{\Theta}^*_T(t)\|^2 + \|\Theta_T^*(t)\|^2 \big) \int_I \big|\mathbb{E}[\wh{X}^u(t)] \big|^2 du \\
&\hspace{0.3in} + \big( \|\bar{\Theta}_T^*(t) - \bar{\Theta}^*\|^2 + \|\Theta^*_T(t) -\Theta^*\|^2 \big) \int_I \big|\mathbb{E}[X_\infty^u(t)] \big|^2 du \\ 
& \hspace{0.3in} + \|\theta_T^*(t) - \theta^*\|^2_{L^2} + \|\wt{\Theta}_T^*(t)\|^2_{L^2} \int_I \big|\mathbb{E}[\wh{X}^v(t)] \big|^2 dv \\
& \hspace{0.3in} + \|\wt{\Theta}_T^*(t) - \wt{\Theta}^*\|^2_{L^2} \int_I \big|\mathbb{E}[X_{\infty}^v(t)] \big|^2 dv \Big\} \\
& \leq K \big(e^{-\lambda t} + e^{-\lambda(T-t)} \big)
\end{aligned}
\end{equation*}
for all $t \in [0, T]$. This follows from the inequalities \eqref{eq:moment_boundness}, \eqref{eq:exponential_estimate_Theta}, \eqref{eq:exponential_estimate_tilde_Theta}, \eqref{eq:exponential_estimate_theta}, and the preceding estimate for $\wh{X}$. Moreover, if $\mathbb{E}[|\xi_1^u - \xi_2^u|^2] = 0$ for a.e. $u \in I$, the same argument yields \eqref{eq:turnpike_property_same_initial}. This completes the proof.
\end{proof}

\begin{remark}
\label{r:pointwise_result}
If we assume that $G \in L^{\infty}(I \times I; \mathbb{R})$, i.e., $|G(u, v)| \leq K$ for a.e. $u, v \in I$, then we can show that
\begin{equation*}
\Lambda_T \in C^1([0, T]; L^{\infty}(I \times I; \mathbb{R}^{d \times d})), \quad p_T \in C^1([0, T]; L^{\infty}(I; \mathbb{R}^d))
\end{equation*}
and
\begin{equation*}
\Lambda \in L^{\infty}(I \times I; \mathbb{R}^{d \times d}), \quad p \in L^{\infty}(I; \mathbb{R}^d).
\end{equation*}
Suppose in addition that $\mathbb{E}[|\xi^u_i|^2] < \infty$ for $i = 1, 2$ and a.e. $u \in I$. 
Then, by a similar argument, the turnpike property in \eqref{eq:turnpike_property} can be strengthened to the pointwise estimate
\begin{equation*}
\mathbb{E} \big[|X_T^u(t) - X_{\infty}^u(t)|^2 + |\alpha_T^u(t) - \alpha_{\infty}^u(t)|^2 \big] \leq K \big(e^{-\lambda t} + e^{-\lambda(T-t)} \big)
\end{equation*}
for all $t \in [0, T]$ and a.e. $u \in I$, where $K$ and $\lambda$ are independent of $u$, $t$, and $T$. 
\end{remark}

\subsection{Convergence of the time-averaged value function} 
\label{s:convergence_value}

In this section, we show that the time-averaged value function $\frac{1}{T} V_T(\xi)$ for the finite-horizon GMFC problem converges to the value $V_{\infty}$ of the ergodic control problem. 

\begin{lemma}
\label{l:convergence_value_function}
Suppose that Assumptions \ref{a:coefficients_solvability_finite}, \ref{a:stabilizability}, and \ref{a:solvability_Lambda_p} hold. Then, for all $\xi \in \kL_0^2$, 
\begin{equation}
\label{eq:convergence_value}
\lim_{T \to \infty} \frac{1}{T} V_{T}(\xi) = V_{\infty}.
\end{equation}
Here, $V_T(\xi) = \mathbb{V}_T(0, \xi)$ is the value function of the finite-horizon GMFC problem, where $\mathbb{V}_T$ is given explicitly by \eqref{eq:value_function_ansatz_finite}, while $V_{\infty}$, defined in \eqref{eq:ergodic_value}, denotes the value of the ergodic control problem. 
\end{lemma}
\begin{proof}
For all $\xi \in \kL_0^2$, the explicit representations of $V_{T}(\xi) = \mathbb{V}_T(0, \xi)$ in \eqref{eq:value_function_ansatz_finite} and $V_{\infty}$ in \eqref{eq:ergodic_value} yield
\begin{equation*}
\begin{aligned}
&\lim_{T \to \infty} \Big| \frac{1}{T} V_T(\xi) - V_{\infty} \Big| \\
\leq \ & \lim_{T \to \infty} \frac{1}{T} \Big(\int_{I} \big| \mathbb{E}[\langle \xi^u, P_T(0) \xi^u \rangle] \big| du + \int_{I} \big| \langle \bar{\xi}^u, \Pi_T(0) \bar{\xi}^u \rangle \big| du + \int_{I} \int_{I} \big| \langle \bar{\xi}^u, \Lambda_T(0)(u, v) \bar{\xi}^v \rangle \big| dv du \\
& \hspace{0.6in} + 2 \int_{I} \big|\langle \bar{\xi}^u, p_T(0)(u) \rangle \big| du \Big) + \lim_{T \to \infty} \Big| \frac{1}{T} \int_{I} \kappa_T(0)(u) du - V_{\infty} \Big| \\
= \ & \lim_{T \to \infty} \Big| \frac{1}{T} \int_{I} \kappa_T(0)(u) du - V_{\infty} \Big|
\end{aligned}
\end{equation*}
since 
$$\|P_T(0)\| + \|\Pi_T(0)\| + \|\Lambda_T(0)\|_{L^2} + \|p_T(0)\|_{L^2} \leq K$$
for some $K > 0$, independent of $T$. The above inequality follows from the estimates \eqref{eq:exponential_estimate_P_Pi}, \eqref{eq:exponential_estimate_Lambda}, and \eqref{eq:exponential_estimate_p}. Then, from the differential equation satisfied by $\kappa_T$ in \eqref{eq:ODE_kappa_T}, we obtain
\begin{equation*}
\begin{aligned}
\kappa_T(0)(u) & = \int_{0}^{T} \big( -\langle \Gamma(P_T(t), p_T(t))(u), \mathcal{R}(P_T(t))^{-1} \Gamma(P_T(t), p_T(t))(u) \rangle \\
& \hspace{0.6in} + \langle P_T(t)\sigma, \sigma \rangle  + 2 \langle p_T(t)(u), b \rangle \big) dt,
\end{aligned}
\end{equation*}
for a.e. $u \in I$, which gives that
\begin{equation*}
\begin{aligned}
\frac{1}{T} \int_{I} \kappa_T(0)(u) du - V_{\infty} = \ & \frac{1}{T} \int_{0}^{T} \Big\{ \int_{I} \big(\langle \Gamma(P, p)(u), \mathcal{R}(P)^{-1} \Gamma(P, p)(u) \rangle \\
& \hspace{0.5in} - \langle \Gamma(P_T(t), p_T(t))(u), \mathcal{R}(P_T(t))^{-1} \Gamma(P_T(t), p_T(t))(u) \rangle \big) du \\
& \hspace{0.5in} + \langle (P_T(t) - P) \sigma, \sigma \rangle + 2 \int_{I} \langle p_T(t)(u) - p(u), b \rangle du \Big\} dt.
\end{aligned}
\end{equation*}
Note that
\begin{equation*}
\begin{aligned}
& \big\langle \Gamma(P, p)(u), \mathcal{R}(P)^{-1} \Gamma(P, p)(u) \big\rangle - \big\langle \Gamma(P_T(t), p_T(t))(u), \mathcal{R}(P_T(t))^{-1} \Gamma(P_T(t), p_T(t))(u) \big\rangle \\
= \ & \big\langle D^\top (P - P_T(t))\sigma + B^\top (p(u) - p_T(t)(u)), \mathcal{R}(P)^{-1} \Gamma(P, p)(u) \big\rangle \\
& \hspace{0.2in} + \big\langle \Gamma(P_T(t), p_T(t))(u), \mathcal{R}(P)^{-1} D^\top(P_T(t) - P) D \mathcal{R}(P_T(t))^{-1} \Gamma(P, p)(u) \big\rangle \\
& \hspace{0.2in} + \big\langle \Gamma(P_T(t), p_T(t))(u), \mathcal{R}(P_T(t))^{-1} \big\{ D^\top (P - P_T(t))\sigma + B^\top (p(u) - p_T(t)(u)) \big\} \big\rangle.
\end{aligned}
\end{equation*}
Consequently, we derive the following inequality:
\begin{equation}
\label{eq:estimate_integral_kappa}
\begin{aligned}
& \lim_{T \to \infty} \Big| \frac{1}{T} \int_{I} \kappa_T(0)(u) du - V_{\infty} \Big| \\
\leq \ & \lim_{T \to \infty} \frac{1}{T} \Big(\int_{0}^{T} | \langle (P_T(t) - P)\sigma, \sigma \rangle | dt + \int_{0}^{T} \int_{I} 2 | \langle p_T(t)(u) - p(u), b \rangle | du dt \Big) \\
& + \lim_{T \to \infty} \frac{1}{T} \int_{0}^{T} \int_{I} \big| \big\langle D^\top(P - P_T(t))\sigma + B^\top(p(u) - p_T(t)(u)), \mathcal{R}(P)^{-1} \Gamma(P, p)(u) \big\rangle \big| du dt \\
& + \lim_{T \to \infty} \frac{1}{T} \int_{0}^{T} \int_{I} \big| \big\langle \Gamma(P_T(t), p_T(t))(u), \mathcal{R}(P)^{-1} D^\top(P_T(t) - P) D \mathcal{R}(P_T(t))^{-1} \Gamma(P, p)(u) \big\rangle \big| du dt \\
& + \lim_{T \to \infty} \frac{1}{T} \int_{0}^{T} \int_{I} \big| \big\langle \Gamma(P_T(t), p_T(t))(u), \mathcal{R}(P_T(t))^{-1} \big\{ D^\top(P - P_T(t))\sigma \\
& \hspace{1.2in} + B^\top(p(u) - p_T(t)(u)) \big\} \big\rangle \big| du dt.
\end{aligned}
\end{equation}
From the estimate \eqref{eq:exponential_estimate_P_Pi}, we obtain that
\begin{equation*}
\begin{aligned}
\lim_{T \to \infty} \frac{1}{T} \int_{0}^{T} \big| \langle (P_T(t) - P)\sigma, \sigma \rangle \big| dt & \leq \lim_{T \to \infty} \frac{1}{T} \int_{0}^{T} \|P_T(t) - P\| |\sigma|^2 dt \\
& \leq \lim_{T \to \infty} \frac{1}{T} \frac{K |\sigma|^2}{\lambda} (1 - e^{-\lambda T}) = 0.
\end{aligned}
\end{equation*}
Similarly, applying H\"older's inequality together with the estimate \eqref{eq:exponential_estimate_p}, we have
\begin{equation*}
\begin{aligned}
\lim_{T \to \infty} \frac{1}{T} \int_{0}^{T} \int_{I} 2 | \langle p_T(t)(u) - p(u), b \rangle | du dt
& \leq \lim_{T \to \infty} \frac{1}{T} \int_{0}^{T} 2 \|p_T(t) - p\|_{L^2} |b| dt \\
& \leq \lim_{T \to \infty} \frac{1}{T} \int_{0}^{T} 2 |b| K e^{-\lambda(T-t)} dt = 0.
\end{aligned}
\end{equation*}
Next, since $\|P_T(t)\| + \|\mathcal{R}(P_T(t))^{-1}\| + \|\mathcal{R}(P)^{-1}\| + \|p_T(t)\|_{L^2} \leq K$ for all $t \in [0, T]$, where $K > 0$ is independent of $T$, and $\Gamma(P, p)$ is linear in $P$ and $p$, a similar argument shows that all the remaining terms on the right-hand side of \eqref{eq:estimate_integral_kappa} converge to 0. Hence, the following limit holds:
\begin{equation*}
\lim_{T \to \infty} \Big| \frac{1}{T} \int_{I} \kappa_T(0)(u) du - V_{\infty} \Big| = 0,
\end{equation*}
which yields that
\begin{equation*}
\lim_{T \to \infty} \Big| \frac{1}{T} V_T(\xi) - V_{\infty} \Big| = 0
\end{equation*}
for all $\xi \in \kL^2_0$. This proves \eqref{eq:convergence_value}.
\end{proof}


\section{Unique solvability of the system of algebraic equations}
\label{s:solvability_algebraic}

Recall that Assumption \ref{a:solvability_Lambda_p} requires the system of algebraic equations \eqref{eq:Riccati_Lambda_ergodic}-\eqref{eq:Riccati_p_ergodic} to admit a unique solution such that the kernel $\Lambda(u, v)$ induces a positive operator on $L^2(I; \mathbb{R}^d)$. Moreover, for each fixed $t \geq 0$, we assume that the solution $\Lambda_T(t)$ of \eqref{eq:ODE_Lambda_T} converges to the solution $\Lambda$ of \eqref{eq:Riccati_Lambda_ergodic} in $L^2(I \times I; \mathbb{R}^{d \times d})$ as $T \to \infty$. To demonstrate that this assumption is reasonable, we provide a sufficient condition for the unique solvability of the system \eqref{eq:Riccati_Lambda_ergodic}-\eqref{eq:Riccati_p_ergodic}. To proceed, given $P, \Pi \in \mathbb{S}^d$ and $G(u, v) \in L^2(I \times I; \mathbb R)$, we define
\begin{equation*}
\begin{aligned}
\Psi_{3}(P, \Pi)(u,v) &= (P+\Pi)\wt{A}G(u,v) + \wt{A}^\top(P+\Pi)G(v, u) \\
&\hspace{0.3in} + (C+\bar{C})^\top P\wt{C}G(u,v) + \wt{C}^\top P(C+\bar{C})G(v,u) \\
&\hspace{0.3in} + \int_{I} \wt{C}^\top G(w,u)P \wt{C} G(w,v) \, dw + \wt{Q} G(u,v) \\
&\hspace{0.3in} + \int_{I} G(w,u)\check{Q} G(w,v) \, dw.
\end{aligned}
\end{equation*}
Moreover, we define the following operators on the finite horizon $[0, T]$: For $\phi \in L^2(I;\RR^d)$ or $\phi \in L^2(I;\RR^m)$, let
\begin{equation}
\label{eq:operators_finite}
\begin{aligned}
(\mathcal O_{1}(t) \phi)(u) &= \int_{I} \Psi_{3}(P_{T}(t), \Pi_{T}(t))(u,v) \phi^{v} dv, \\
(\mathcal O_{2}(t) \phi)(u) &= \widehat{\cM}(P_{T}(t), \Pi_{T}(t))^\top \cR(P_{T}(t))^{-1} \widehat{\cM}(P_{T}(t), \Pi_{T}(t)) \phi^{u}, \\
(\mathcal O_{3}(t) \phi)(u) &= \int_{I} D^\top P_{T}(t)\wt{C}G(u,v) \phi^{v} dv, \\
(\mathcal O_{4}(t) \phi)(u) &= \widehat{\cM}(P_{T}(t), \Pi_{T}(t)) \phi^{u}, \\
(\mathcal O_{5}(t) \phi)(u) &= \cR(P_{T}(t)) \phi^{u}.
\end{aligned}
\end{equation}
Similarly, for the infinite-horizon setting, we define the following operators. For $\phi \in L^2(I;\RR^d)$ or $\phi \in L^2(I;\RR^m)$, let
\begin{equation}
\label{eq:operators_ergodic}
\begin{aligned}
(\mathcal O_{1} \phi)(u) &= \int_{I} \Psi_{3}(P, \Pi)(u,v) \phi^{v} dv, \\
(\mathcal O_{2} \phi)(u) &= \widehat{\cM}(P, \Pi)^\top \cR(P)^{-1} \widehat{\cM}(P, \Pi) \phi^{u}, \\
(\mathcal O_{3} \phi)(u) &= \int_{I} D^\top P\wt{C}G(u,v) \phi^{v} dv, \\
(\mathcal O_{4} \phi)(u) &= \widehat{\cM}(P, \Pi) \phi^{u}, \\
(\mathcal O_{5} \phi)(u) &= \cR(P) \phi^{u}.
\end{aligned}
\end{equation}

We impose the following assumption to ensure the unique solvability of the algebraic system \eqref{eq:Riccati_Lambda_ergodic}-\eqref{eq:Riccati_p_ergodic}.

\begin{assumption}
\label{a:operators}
Let $\mathcal{O}_1(t), \dots, \mathcal{O}_5(t)$ be the operators defined in \eqref{eq:operators_finite} for $t \in [0, T]$, and let $\mathcal{O}_1, \dots, \mathcal{O}_5$ be the operators defined in \eqref{eq:operators_ergodic}. We assume that
\begin{equation}
\label{eq:operator_condition_1}
\begin{pmatrix} \mathcal O_{1}(t) + \mathcal O_{2}(t) & (\mathcal O_{3}(t) + \mathcal O_{4}(t))^{*} \\ \mathcal O_{3}(t) + \mathcal O_{4}(t) & \mathcal O_{5}(t) \end{pmatrix} \geq 0
\end{equation}
for all $t \in [0,T]$, and
\begin{equation}
\label{eq:operator_condition_2}
\begin{pmatrix} \mathcal O_{1}(t) + \mathcal O_{2}(t) & (\mathcal O_{3}(t) + \mathcal O_{4}(t))^{*} \\ \mathcal O_{3}(t) + \mathcal O_{4}(t) & \mathcal O_{5}(t) \end{pmatrix} \leq \begin{pmatrix} \mathcal O_{1} + \mathcal O_{2} & (\mathcal O_{3} + \mathcal O_{4})^{*} \\ \mathcal O_{3} + \mathcal O_{4} & \mathcal O_{5} \end{pmatrix}
\end{equation}
for all $t \in [0,T]$, where $(\mathcal O_{3}(t) + \mathcal O_{4}(t))^{*}$ and $(\mathcal O_{3} + \mathcal O_{4})^{*}$ are the adjoint operators of $\mathcal O_{3}(t) + \mathcal O_{4}(t)$ and $\mathcal O_{3} + \mathcal O_{4}$, respectively.
\end{assumption}

\begin{remark}
\label{re:operator_assump}
Since $R > 0$, $P > 0$, and $P_T(t) \geq 0$, the operators $\mathcal O_5(t)$ and $\mathcal O_5$ are both positive and therefore invertible. Denote
$$\mathcal O(t) := \mathcal O_{1}(t) + \mathcal O_{2}(t) - (\mathcal O_{3}(t) + \mathcal O_{4}(t))^{*} \mathcal O_5(t)^{-1} (\mathcal O_{3}(t) + \mathcal O_{4}(t))$$
for $t \in [0, T]$ and 
$$\mathcal O := \mathcal O_{1} + \mathcal O_{2} - (\mathcal O_{3} + \mathcal O_{4})^{*} \mathcal O_5^{-1} (\mathcal O_{3} + \mathcal O_{4}).$$
Then, the kernel of the operator $\mathcal O(t)$ is given by
\begin{equation*}
\begin{aligned}
\mathcal O(t)(u, v) &= \Psi_3(P_T(t), \Pi_T(t))(u, v) - \widehat{\cM}(P_{T}(t), \Pi_{T}(t))^\top \cR(P_T(t))^{-1} D^\top P_T(t) \wt{C} G(u, v) \\
& \hspace{0.3in} - \wt{C}^\top G(u, v) P_T(t) D \cR(P_T(t))^{-1} \widehat{\cM}(P_{T}(t), \Pi_{T}(t)) \\
& \hspace{0.3in} - \int_I \wt{C}^\top G(u, w) P_T(t) D \cR(P_T(t))^{-1} D^\top P_T(t) \wt{C} G(w, v) dw.
\end{aligned}
\end{equation*}
Then, the condition \eqref{eq:operator_condition_1} is equivalent to 
$$\int_I \int_I \lan \phi^u, \mathcal O(t)(u, v) \phi^v \ran dv du \geq 0, \quad \forall t \in [0, T], \, \phi \in L^2(I; \mathbb R^d).$$
Similarly, the condition \eqref{eq:operator_condition_2} is equivalent to
$$\int_I \int_I \lan \phi^u, (\mathcal O(u, v) - \mathcal O(t)(u, v)) \phi^v \ran dv du \geq 0, \quad \forall t \in [0, T], \, \phi \in L^2(I; \mathbb R^d).$$

Let us illustrate the condition by considering two simple cases. First, suppose that $\wt{C} = 0$. Then
$$\mathcal O(t) = \Psi_3(P_T(t), \Pi_T(t)).$$
If, in addition, the kernel $\wt{A}G(u, v)$ induces a positive operator on $L^2(I; \mathbb{R}^d)$, then Assumption \ref{a:coefficients_solvability_finite} implies that $\mathcal O(t) \geq 0$ for all $t \in [0, T]$. Moreover, the comparison condition \eqref{eq:operator_condition_2} can be verified by the monotonicity of $P_T$ and $\bar{\Pi}_T$ toward $P$ and $\bar{\Pi}$, respectively. Second, suppose that $D = 0$. In this case, the same identity $\mathcal O(t) = \Psi_3(P_T(t), \Pi_T(t))$ holds. Hence, the required condition \eqref{eq:operator_condition_1} is fulfilled whenever $\Psi_3(P_T(t), \Pi_T(t)) \geq 0$ for all $t \in [0, T]$.
\end{remark}

In the following proposition, we establish the unique solvability of the system of algebraic equations \eqref{eq:Riccati_Lambda_ergodic}-\eqref{eq:Riccati_p_ergodic} under Assumption \ref{a:operators}.

\begin{proposition}
\label{p:solvability_system_ergodic}
Let Assumptions \ref{a:coefficients_solvability_finite}, \ref{a:stabilizability}, and \ref{a:operators} hold. Then, the abstract algebraic Riccati equation \eqref{eq:Riccati_Lambda_ergodic} admits a unique solution $\Lambda \in L^2(I \times I; \RR^{d \times d})$ such that the kernel $\Lambda(u, v)$ induces a positive operator on $L^2(I; \mathbb{R}^d)$. Moreover, for each fixed $t \geq 0$, the solution $\Lambda_T(t)$ of \eqref{eq:ODE_Lambda_T} converges to $\Lambda$ in $L^2(I \times I; \mathbb{R}^{d \times d})$ as $T \to \infty$. Furthermore, the linear equation \eqref{eq:Riccati_p_ergodic} admits a unique solution $p \in L^2(I; \RR^d)$.
\end{proposition}
\begin{proof}
First, we recall that, under Assumptions \ref{a:coefficients_solvability_finite} and \ref{a:stabilizability}, \cite[Lemma 2.2]{Sun-Yong-2024} implies that the system of standard algebraic Riccati equations \eqref{eq:Riccati_P_Pi_ergodic} admits a unique solution pair $(P, \bar{\Pi}) \in \mathbb S^{d}_{++} \times \mathbb S^d_{++}$. Moreover, for all $T > 0$ and $t \in [0, T]$, one has $P_T(t) \leq P$ and $\bar{\Pi}_T(t) \leq \bar{\Pi}$,  where $P_T$ and $\bar{\Pi}_T$ are the solutions to the differential Riccati equations \eqref{eq:ODE_P_T} and \eqref{eq:ODE_Pi_T}, respectively. Furthermore, for each fixed $t \geq 0$, $P_T(t)$ converges to $P$ and $\bar{\Pi}_T(t)$ converges to $\bar{\Pi}$ as $T \to \infty$. 

Next, we show that the abstract algebraic Riccati equation \eqref{eq:Riccati_Lambda_ergodic} admits a unique solution $\Lambda$ in $L^{2}(I \times I; \mathbb{R}^{d \times d})$. We also show that, for each fixed $t \geq 0$, $\Lambda_{T}(t)$ converges to $\Lambda$ in $L^2(I \times I; \mathbb{R}^{d \times d})$ as $T \to \infty$, where $\Lambda_T$ is the solution to the abstract Riccati differential equation \eqref{eq:ODE_Lambda_T}. The idea is to construct a finite-horizon control problem whose value function takes the form
\begin{equation*}
\wt{V}_{T}(t,\bar{\xi}) = \int_{I}\int_{I} \langle \bar{\xi}^{u}, \Lambda_{T}(t)(u,v) \bar{\xi}^{v} \rangle dv du,
\end{equation*}
where $\Lambda_{T}$ satisfies the same equation as \eqref{eq:ODE_Lambda_T}. We then pass to the limit as $T \to \infty$ to construct a solution to \eqref{eq:Riccati_Lambda_ergodic}.

Let $\mathcal{A}^D[0, T]$ denote the collection of deterministic measurable functions $\bar{\alpha} = (\bar{\alpha}^{u})_{u \in I}$ satisfying $\int_{I} \int_{0}^T |\bar{\alpha}^{u}(s)|^{2} ds du < \infty$. We consider the following deterministic control problem. Given $\bar{\xi} \in L^2(I; \mathbb{R}^d)$ and $\bar{\alpha} \in \mathcal{A}^D[0, T]$, the state dynamics are governed by the controlled ODE \eqref{eq:homo_ode}, i.e.,
\begin{equation*}
\begin{cases}
d\bar{X}^{u}(t) = \{ (A+\bar{A})\bar{X}^{u}(t) + \wt{A}[G\bar{X}(t)]^{u} + B\bar{\alpha}^{u}(t)\} dt, \\
\bar{X}^{u}(0) = \bar{\xi}^{u},
\end{cases}
\end{equation*}
and the associated cost functional is given by
\begin{equation*}
\begin{aligned}
\wt{J}_{T}(t,\bar{\xi},\bar{\alpha}) &= \int_{t}^T \Big\{ \int_{I}\int_{I} \langle \bar{X}^{u}(s), \Psi_{3}(P_{T}(s), \Pi_{T}(s))(u,v) \bar{X}^{v}(s) \rangle dv du \\
&\hspace{0.3in} + 2\int_{I} \Big\langle \bar{\alpha}^{u}(s), \int_{I} D^\top P_{T}(s)\wt{C} G(u,v) \bar{X}^{v}(s) dv \Big\rangle du \\
&\hspace{0.3in} + 2\int_{I} \langle \bar{\alpha}^{u}(s), \widehat{\cM}(P_{T}(s), \Pi_{T}(s)) \bar{X}^{u}(s) \rangle du \\
&\hspace{0.3in} + \int_{I} \big\langle \bar{X}^{u}(s), \widehat{\cM}(P_{T}(s), \Pi_{T}(s))^\top \cR(P_{T}(s))^{-1} \widehat{\cM}(P_{T}(s), \Pi_{T}(s)) \bar{X}^{u}(s) \big\rangle \, du \\
&\hspace{0.3in} + \int_{I} \langle \bar{\alpha}^{u}(s), \cR(P_{T}(s)) \bar{\alpha}^{u}(s) \rangle du \Big\} ds,
\end{aligned}
\end{equation*}
where $P_{T}$ and $\bar{\Pi}_{T}$ are the unique solutions to \eqref{eq:ODE_P_T} and \eqref{eq:ODE_Pi_T}, respectively, and $\Pi_{T}(t) = \bar{\Pi}_{T}(t) - P_{T}(t)$. For the above deterministic control problem, we consider the following ansatz for the value function:
\begin{equation*}
\wt{\mathbb V}_{T}(t, \bar{\xi}) = \int_{I}\int_{I} \langle \bar{\xi}^{u}, \Lambda_{T}(t)(u,v) \bar{\xi}^{v} \rangle dv du,
\end{equation*}
where $\Lambda_{T} \in C^1([0,T]; L^{2}(I \times I; \mathbb{R}^{d \times d}))$ is a suitable function to be determined. To proceed, we define
$\wt{\mathcal V}_{T}(t) := \int_{I}\int_{I} \langle \bar{X}^{u}(t), \Lambda_{T}(t)(u,v) \bar{X}^{v}(t) \rangle dv du$
and let
\begin{equation*}
\begin{aligned}
\wt{\mathcal Y}_{T}(t) &= \wt{\mathcal V}_{T}(t) + \int_{0}^t \Big\{ \int_{I} \int_{I} \langle \bar{X}^{u}(s), \Psi_{3}(P_{T}(s), \Pi_{T}(s))(u,v) \bar{X}^{v}(s) \rangle dv du  \\
&\hspace{0.3in} + 2\int_{I} \Big\langle \bar{\alpha}^{u}(s), \int_{I} D^\top P_{T}(s)\wt{C}G(u,v)\bar{X}^{v}(s) dv \Big\rangle du \\
&\hspace{0.3in} + 2\int_{I} \big\langle \bar{\alpha}^{u}(s), \widehat{\cM}(P_{T}(s), \Pi_{T}(s)) \bar{X}^{u}(s) \big\rangle du  \\
&\hspace{0.3in} + \int_{I} \big\langle \bar{X}^{u}(s), \widehat{\cM}(P_{T}(s), \Pi_{T}(s))^\top \cR(P_{T}(s))^{-1} \widehat{\cM}(P_{T}(s), \Pi_{T}(s)) \bar{X}^{u}(s) \big\rangle du \\
&\hspace{0.3in} + \int_{I} \langle \bar{\alpha}^{u}(s), \cR(P_{T}(s)) \bar{\alpha}^{u}(s) \rangle du \Big\} ds.
\end{aligned}
\end{equation*}
Then, taking the derivative with respect to $t$, we obtain
\begin{equation*}
\begin{aligned}
\frac{d}{dt} \wt{\mathcal Y}_{T}(t) &= \int_{I}\int_{I} \langle \bar{X}^{u}(t), \dot{\Lambda}_{T}(t)(u,v) \bar{X}^{v}(t) \rangle dv du \\
&\hspace{0.3in} + \int_{I}\int_{I} \langle (A+\bar{A})\bar{X}^{u}(t) + \wt{A}[G\bar{X}(t)]^{u} + B\bar{\alpha}^{u}(t), \Lambda_{T}(t)(u,v) \bar{X}^{v}(t) \rangle dv du \\
&\hspace{0.3in} + \int_{I}\int_{I} \langle \bar{X}^{u}(t), \Lambda_{T}(t)(u,v) \{(A+\bar{A})\bar{X}^{v}(t) + \wt{A}[G\bar{X}(t)]^{v} + B\bar{\alpha}^{v}(t) \} \rangle dv du \\
&\hspace{0.3in} + \int_{I}\int_{I} \langle \bar{X}^{u}(t), \Psi_{3}(P_{T}(t), \Pi_{T}(t))(u,v) \bar{X}^{v}(t) \rangle dv du \\
&\hspace{0.3in} + 2\int_{I} \Big\langle \bar{\alpha}^{u}(t), \int_{I} D^\top P_{T}(t)\wt{C}G(u,v) \bar{X}^{v}(t) dv \Big\rangle du \\
&\hspace{0.3in} + 2\int_{I} \big\langle \bar{\alpha}^{u}(t), \widehat{\cM}(P_{T}(t), \Pi_{T}(t)) \bar{X}^{u}(t) \big\rangle du \\
&\hspace{0.3in} + \int_{I} \langle \bar{X}^{u}(t), \widehat{\cM}(P_{T}(t), \Pi_{T}(t))^\top \cR(P_{T}(t))^{-1} \widehat{\cM}(P_{T}(t), \Pi_{T}(t)) \bar{X}^{u}(t) \rangle du \\
&\hspace{0.3in} + \int_{I} \langle \bar{\alpha}^{u}(t), \cR(P_{T}(t)) \bar{\alpha}^{u}(t) \rangle du.
\end{aligned}
\end{equation*}

Set
\begin{equation*}
\wt{Z}_{T}^{u}(t) = \widehat{\cM}(P_{T}(t), \Pi_{T}(t)) \bar{X}^{u}(t) + \int_{I} \Upsilon(P_{T}(t), \Lambda_{T}(t))(u,v) \bar{X}^{v}(t) dv.
\end{equation*}
Then, $\frac{d}{dt} \wt{\mathcal Y}_{T}(t)$ can be rewritten as follows:
\begin{equation*}
\begin{aligned}
\frac{d}{dt} \wt{\mathcal Y}_{T}(t) &= \int_{I}\int_{I} \langle \bar{X}^{u}(t), \{ \dot{\Lambda}_{T}(t)(u,v) + \Psi_{1}(P_{T}(t), \Pi_{T}(t), \Lambda_{T}(t))(u,v) \} \bar{X}^{v}(t) \rangle dv du \\
&\hspace{0.3in} + \int_{I} \langle \bar{\alpha}^{u}(t), \cR(P_{T}(t)) \bar{\alpha}^{u}(t) \rangle du + 2 \int_{I} \langle \bar{\alpha}^{u}(t), \wt{Z}_{T}^{u}(t) \rangle du \\
&\hspace{0.3in} + \int_{I} \langle \bar{X}^{u}(t), \widehat{\cM}(P_{T}(t), \Pi_{T}(t))^\top \cR(P_{T}(t))^{-1} \widehat{\cM}(P_{T}(t), \Pi_{T}(t)) \bar{X}^{u}(t) \rangle du.
\end{aligned}
\end{equation*}
Completing the square with respect to the control terms, we have
\begin{equation*}
\begin{aligned}
&\int_{I} \langle \bar{\alpha}^{u}(t), \cR(P_{T}(t)) \bar{\alpha}^{u}(t) \rangle \, du + 2 \int_{I} \langle \bar{\alpha}^{u}(t), \wt{Z}_{T}^{u}(t) \rangle du \\
= \ & \int_{I} \langle \cR(P_{T}(t)) \{ \bar{\alpha}^{u}(t) + \cR(P_{T}(t))^{-1} \wt{Z}_{T}^{u}(t) \}, \{ \bar{\alpha}^{u}(t) + \cR(P_{T}(t))^{-1} \wt{Z}_{T}^{u}(t) \} \rangle du \\
&\quad - \int_{I} \langle \wt{Z}_{T}^{u}(t), \cR(P_{T}(t))^{-1} \wt{Z}_{T}^{u}(t) \rangle du.
\end{aligned}
\end{equation*}
Expanding the term $\int_{I} \langle \wt{Z}_{T}^{u}(t), \cR(P_{T}(t))^{-1} \wt{Z}_{T}^{u}(t) \rangle \, du$, we observe that
\begin{equation*}
\begin{aligned}
&\int_{I} \langle \wt{Z}_{T}^{u}(t), \cR(P_{T}(t))^{-1} \wt{Z}_{T}^{u}(t) \rangle du \\
= \ & \int_{I} \langle \bar{X}^{u}(t), \widehat{\cM}(P_{T}(t), \Pi_{T}(t))^\top \cR(P_{T}(t))^{-1} \widehat{\cM}(P_{T}(t), \Pi_{T}(t)) \bar{X}^{u}(t) \rangle du \\
& \hspace{0.1in} + \int_{I}\int_{I} \langle \bar{X}^{u}(t), \widehat{\cM}(P_{T}(t), \Pi_{T}(t))^\top \cR(P_{T}(t))^{-1} \Upsilon(P_{T}(t), \Lambda_{T}(t))(u,v) \bar{X}^{v}(t) \rangle dv du \\
& \hspace{0.1in} + \int_{I}\int_{I} \langle \bar{X}^{u}(t), \Upsilon(P_{T}(t), \Lambda_{T}(t))(v,u)^\top \cR(P_{T}(t))^{-1} \widehat{\cM}(P_{T}(t), \Pi_{T}(t)) \bar{X}^{v}(t) \rangle dv du \\
& \hspace{0.1in} + \int_{I}\int_{I} \Big\langle \bar{X}^{u}(t), \int_{I} \Upsilon(P_{T}(t), \Lambda_{T}(t))(w,u)^\top \cR(P_{T}(t))^{-1} \Upsilon(P_{T}(t), \Lambda_{T}(t))(w,v) dw \bar{X}^{v}(t) \Big\rangle dv du.
\end{aligned}
\end{equation*}
Hence, we deduce the following identity:
\begin{equation*}
\begin{aligned}
\frac{d}{dt} \wt{\mathcal Y}_{T}(t) &= \int_{I}\int_{I} \langle \bar{X}^{u}(t), \{ \dot{\Lambda}_{T}(t)(u,v) + F(t, \Lambda_{T}(t))(u,v) \} \bar{X}^{v}(t) \rangle dv du \\
&\hspace{0.3in} + \int_{I} \langle \cR(P_{T}(t)) \{ \bar{\alpha}^{u}(t) + \cR(P_{T}(t))^{-1} \wt{Z}_{T}^{u}(t) \}, \{ \bar{\alpha}^{u}(t) + \cR(P_{T}(t))^{-1} \wt{Z}_{T}^{u}(t) \} \rangle du.
\end{aligned}
\end{equation*}
Therefore, taking $\Lambda_{T}$ to be the unique solution to \eqref{eq:ODE_Lambda_T} and integrating over $[0, T]$, we obtain
\begin{equation*}
\begin{aligned}
& \int_{0}^T \int_{I} \langle \cR(P_{T}(t)) \{ \bar{\alpha}^{u}(t) + \cR(P_{T}(t))^{-1} \wt{Z}_{T}^{u}(t) \}, \{ \bar{\alpha}^{u}(t) + \cR(P_{T}(t))^{-1} \wt{Z}_{T}^{u}(t) \} \rangle du dt \\
= \ & \wt{\mathcal Y}_{T}(T) - \wt{\mathcal Y}_{T}(0) = \wt{\mathcal V}_T(T) + \wt{J}_T(0,\bar{\xi},\bar{\alpha}) - \wt{\mathcal V}_T(0).
\end{aligned}
\end{equation*}
Since $\wt{\mathcal V}_{T}(T) = 0$ and $\wt{\mathcal V}_{T}(0) = \wt{\mathbb V}_{T}(0, \bar{\xi}) = \int_{I}\int_{I} \langle \bar{\xi}^{u}, \Lambda_{T}(0)(u,v) \bar{\xi}^{v} \rangle dv du$, we conclude that
\begin{equation*}
\begin{aligned}
\wt{J}_{T}(0, \bar{\xi}, \bar{\alpha}) & = \int_{0}^T \int_{I} \langle \cR(P_{T}(t)) \{ \bar{\alpha}^{u}(t) + \cR(P_{T}(t))^{-1} \wt{Z}_{T}^{u}(t) \}, \{ \bar{\alpha}^{u}(t) + \cR(P_{T}(t))^{-1} \wt{Z}_{T}^{u}(t) \} \rangle du dt \\
& \hspace{0.3in} + \int_{I}\int_{I} \langle \bar{\xi}^{u}, \Lambda_{T}(0)(u,v) \bar{\xi}^{v} \rangle dv du.
\end{aligned}
\end{equation*}
Under the feedback control
$$\bar{\alpha}_{T}^{u}(t) = -\cR(P_{T}(t))^{-1} \wt{Z}_{T}^{u}(t) = \bar{\Theta}_T^*(t) \bar{X}^u(t) + \int_I \wt{\Theta}^*_T(t)(u,v) \bar{X}^v(t) dv$$ 
for all $t\in[0,T]$ and a.e. $u \in I$, the cost functional attains its minimum, and we have
\begin{equation*}
\wt{V}_{T}(0, \bar{\xi}) = \inf_{\bar{\alpha} \in \mathcal{A}^D[0, T]} \wt{J}_{T}(0, \bar{\xi}, \bar{\alpha}) = \int_{I}\int_{I} \langle \bar{\xi}^{u}, \Lambda_{T}(0)(u,v) \bar{\xi}^{v} \rangle dv du.
\end{equation*}

Using the operators defined in \eqref{eq:operators_finite}, the cost functional $\wt{J}_T$ can be rewritten as
\begin{equation*}
\wt{J}_{T}(t, \bar{\xi}, \bar{\alpha}) = \int_{t}^T \Big\langle \begin{pmatrix} \bar{X}(s) \\ \bar{\alpha}(s) \end{pmatrix}, \begin{pmatrix} \mathcal O_{1}(s) + \mathcal O_{2}(s) & (\mathcal O_{3}(s) + \mathcal O_{4}(s))^{*} \\ \mathcal O_{3}(s) + \mathcal O_{4}(s) & \mathcal O_{5}(s) \end{pmatrix} \begin{pmatrix} \bar{X}(s) \\ \bar{\alpha}(s) \end{pmatrix} \Big\rangle_{L^2} ds.
\end{equation*}
By \eqref{eq:operator_condition_1} in Assumption \ref{a:operators}, for all $\bar{\xi}\in L^2(I; \RR^d)$ and $\bar{\alpha}\in \cA^{D}[0, T]$, we have
$$\wt{J}_T(t,\bar{\xi},\bar{\alpha})\ge 0, \quad \forall  t\in[0,T].$$
It follows that, for arbitrary $T$ and $t \in [0, T]$, the kernel $\Lambda_T(t)(u, v)$ induces a positive operator on $L^2(I; \mathbb{R}^d)$. Moreover, for $0 \leq T_1 \leq T_2$ and $t \in [0, T_1]$,
$$\int_I \int_I \lan \bar{\xi}^u, \Lambda_{T_1}(t)(u, v) \bar{\xi}^v \ran dv du \leq \int_I \int_I \lan \bar{\xi}^u, \Lambda_{T_2}(t)(u, v) \bar{\xi}^v \ran dv du$$
for all $\bar{\xi} \in L^2(I; \mathbb R^d)$. We define
$$\wt{\Lambda}(t) = \Lambda_T(T-t), \quad \forall t \in [0, T].$$
Then, $\Lambda_t(0) = \wt{\Lambda}(t)$. Since $\bar{\xi}$ is arbitrary, the above inequality implies that $\wt{\Lambda}(t)$ is a non-decreasing operator with respect to $t$.

Let $\cA^D[0, \infty)$ denote the collection of deterministic measurable functions $\bar{\alpha} = (\bar{\alpha}^{u})_{u \in I}$ satisfying $\int_{I} \int_{0}^{\infty} |\bar{\alpha}^{u}(s)|^{2} ds du < \infty$. Similarly, for a given $\bar{\alpha} \in \cA^D[0, \infty)$, we consider the corresponding deterministic control problem over the infinite horizon:
\begin{equation*}
\wt{J}_{\infty}(\bar{\xi}, \bar{\alpha}) = \int_{0}^{\infty} \Big\langle \begin{pmatrix} \bar{X}(s) \\ \bar{\alpha}(s) \end{pmatrix}, \begin{pmatrix} \mathcal O_{1} + \mathcal O_{2} & (\mathcal O_{3} + \mathcal O_{4})^{*} \\ \mathcal O_{3} + \mathcal O_{4} & \mathcal O_{5} \end{pmatrix} \begin{pmatrix} \bar{X}(s) \\ \bar{\alpha}(s) \end{pmatrix} \Big\rangle_{L^2} ds.
\end{equation*}
By Assumption \ref{a:stabilizability}, there exists a stabilizer $(\bar{\Theta}, \wt{\Theta})$ to the controlled ODE \eqref{eq:homo_ode}, where $\bar{\Theta} \in \RR^{m \times d}$ and $\wt{\Theta} \in L^2(I \times I; \RR^{m \times d})$. Define
$$\bar{\alpha}_{\infty}^{u}(s) = \bar{\Theta} \bar{X}_{\infty}^u(s) + \int_I \wt{\Theta}(u, v) \bar{X}_{\infty}^v(s) dv, \quad \forall s \geq 0, \text{ and a.e. }u \in I,$$
where $\bar{X}_{\infty} = (\bar{X}^u_{\infty})_{u \in I}$ is the solution to the corresponding closed-loop system. Then, by the definition of the linear operator $\mathcal{L}_{\text{s}}$, we have $\bar{X}_{\infty}(t) = e^{t \mathcal{L}_\text{s}} \bar{\xi}$ for all $t \geq 0$. The estimate \eqref{eq:operator_s_uniform_stable} implies that there exist constants $K > 0$ and $\lambda > 0$ such that, for all $\bar{\xi} \in L^2(I; \mathbb{R}^d)$,
$$\|\bar{X}_{\infty}(t)\|^2_{L^2} \leq K e^{-\lambda t} \|\bar{\xi}\|^2_{L^2}, \quad \forall t \geq 0.$$
Hence, there exists a constant $K > 0$ such that
$$\wt{J}_{\infty}(\bar{\xi}, \bar{\alpha}_{\infty}) \leq K \|\bar{\xi}\|^2_{L^2} < \infty, \quad \forall \bar{\xi} \in L^2(I; \mathbb R^d).$$ 
By using \eqref{eq:operator_condition_2} in Assumption \ref{a:operators}, we deduce that
\begin{equation*}
\begin{aligned}
& \int_{I} \int_I \lan \bar{\xi}^u, \Lambda_{T}(0)(u, v) \bar{\xi}^v \ran dv du \leq \wt{J}_{T}(0, \bar{\xi}, \bar{\alpha}_{\infty}) \\
\leq \ & \int_0^T \Big\langle \begin{pmatrix} \bar{X}_{\infty}(s) \\ \bar{\alpha}_{\infty}(s) \end{pmatrix}, \begin{pmatrix} \mathcal O_{1} + \mathcal O_{2} & (\mathcal O_{3} + \mathcal O_{4})^{*} \\ \mathcal O_{3} + \mathcal O_{4} & \mathcal O_{5} \end{pmatrix} \begin{pmatrix} \bar{X}_{\infty}(s) \\ \bar{\alpha}_{\infty}(s) \end{pmatrix} \Big\rangle_{L^2} ds \\
\leq \ & \int_0^{\infty} \Big\langle \begin{pmatrix} \bar{X}_{\infty}(s) \\ \bar{\alpha}_{\infty}(s) \end{pmatrix}, \begin{pmatrix} \mathcal O_{1} + \mathcal O_{2} & (\mathcal O_{3} + \mathcal O_{4})^{*} \\ \mathcal O_{3} + \mathcal O_{4} & \mathcal O_{5} \end{pmatrix} \begin{pmatrix} \bar{X}_{\infty}(s) \\ \bar{\alpha}_{\infty}(s) \end{pmatrix} \Big\rangle_{L^2} ds \\
\leq \ & K \|\bar{\xi}\|^2_{L^2} < \infty.
\end{aligned}
\end{equation*}
Thus, $\wt{\Lambda}(t) = \Lambda_t(0)$ is uniformly bounded above for all $t \geq 0$. Together with the monotonicity of $\wt{\Lambda}(t)$, this implies that there exists $\wt{\Lambda}_{\infty} \in L^2(I \times I; \mathbb{R}^{d \times d})$ such that $\wt{\Lambda}(t)$ converges to $\wt{\Lambda}_{\infty}$ in $L^2(I \times I; \mathbb{R}^{d \times d})$. Passing to the limit in the finite-horizon Riccati equation \eqref{eq:ODE_Lambda_T} shows that $\wt{\Lambda}_{\infty}$ satisfies the same equation as $\Lambda$ in \eqref{eq:Riccati_Lambda_ergodic}, which proves the existence and uniqueness of a solution to \eqref{eq:Riccati_Lambda_ergodic}. Moreover, since the kernel $\Lambda_T(t)(u, v)$ induces a positive operator on $L^2(I; \mathbb{R}^d)$ for all $T > 0$ and $t \in [0, T]$, the kernel $\Lambda(u, v)$ also induces a positive operator on $L^2(I; \mathbb{R}^d)$. 

Finally, given $(P, \bar{\Pi}) \in \mathbb{S}^d_{++} \times \mathbb{S}^d_{++}$ and $\Lambda \in L^2(I \times I; \mathbb{R}^{d \times d})$, the equation \eqref{eq:Riccati_p_ergodic} is linear in $p$. Furthermore, the linear term in $p$ appearing in \eqref{eq:Riccati_p_ergodic} is exactly $\mathcal{L}^*(p)(u)$, where $\mathcal{L}^*$ is the adjoint operator of $\mathcal{L}$ defined in \eqref{eq:linear_operator}. Under Assumption \ref{a:stabilizability}, the operator $\mathcal{L}^*$ is uniformly stable. Consequently, the equation \eqref{eq:Riccati_p_ergodic} admits a unique solution.
\end{proof}

{\small
\bibliographystyle{abbrv}
\bibliography{bib_control}
}

\end{document}